\documentclass[article,11pt]{amsart}
\usepackage[top=25mm,bottom=25mm,left=25mm,right=25mm]{geometry}
\usepackage{tikz}
\usetikzlibrary{arrows,decorations.markings}
\usepackage{esvect}

\newtheorem{lem}{Lemma}[section]
\newtheorem{prop}{Proposition}[section]
\newtheorem{thm}{Theorem}[section]
\newtheorem{cor}{Corollary}[section]

\newtheorem{example}{Example}[section]
\newtheorem{rmk}{Remark}[section]
\newtheorem{assumption}{Assumption}[section]

\usepackage{color}

\newcommand{\Norm}[2]{\left\Vert #1 \right\Vert_{#2}}
\DeclareMathOperator{\well}{\widetilde{\ell}}

\theoremstyle{remark}
\theoremstyle{definition}

\DeclareMathOperator{\N}{\mathbb{N}}
\DeclareMathOperator{\R}{\mathbb{R}}

\DeclareMathOperator{\cJ}{\mathcal{J}}

\DeclareMathOperator{\cZ}{\mathcal{Z}}

\DeclareMathOperator{\wV}{\widetilde{V}}

\DeclareMathOperator{\cR}{\mathcal{R}}

\DeclareMathOperator{\Dom}{Dom}

\newcommand{\pd}[2]{\frac{\partial #1}{\partial #2}}
\newcommand{\pdsup}[3]{\frac{\partial^{#3} #1}{\partial #2^{#3}}}

\makeatletter
\def\section{\@startsection{section}{1}%
\z@{1\linespacing\@plus\linespacing}{1\linespacing}%
{\bf\centering}}
\def\subsection{\@startsection{subsection}{0}%
\z@{\linespacing\@plus\linespacing}{\linespacing}%
{\bf}}
\def\subsubsection{\@startsection{subsubsection}{0}%
\z@{\linespacing\@plus\linespacing}{\linespacing}%
{\bf}}
\makeatother

\makeatletter
\@addtoreset{equation}{section}
\makeatother

\DeclareMathOperator{\supp}{supp}
\DeclareMathOperator{\Spec}{Spec}

\providecommand{\pro}[1]{(#1_t)_{t \geq 0}}

\newcommand{\ex}{\mathbb{E}}

\newcommand{\Rd}{\mathbb{R}^d}

\newcommand{\Dv}{{\rm D}_h\varphi (x)}
\newcommand{\Dvp}{{\rm D}_h\varphi_+ (0)}
\newcommand{\Dvm}{{\rm D}_h\varphi_- (0)}
\newcommand{\DDv}{{\rm D}_h f(x)}

\newcommand{\J}{\mathcal J}

\begin{document}
\title[Potentials for non-local Schr\"odinger operators with zero eigenvalues]
{Potentials for non-local Schr\"odinger operators with zero eigenvalues}
\author{Giacomo Ascione and J\'ozsef L\H orinczi}
\address{Giacomo Ascione,
Dipartimento di Matematica e Applicazioni ``Renato Caccioppoli" \\
Universit\`a degli Studi di Napoli Federico II, 80126 Napoli, Italy}
\email{giacomo.ascione@unina.it}

\address{J\'ozsef L\H orinczi,
Department of Mathematical Sciences, Loughborough University \\
Loughborough LE11 3TU, United Kingdom}
\email{J.Lorinczi@lboro.ac.uk}

\thanks{
\noindent
2010 {\it MS Classification}: Primary 47D08, 60G51; Secondary 47D03, 47G20
\\
\noindent
GA was supported by GNAMPA-INdAM and thanks Loughborough University for hospitality during a research visit.
}

\begin{abstract}
The purpose of this paper is to give a systematic description of potentials decaying to zero at infinity, which generate
eigenvalues at the edge of the absolutely continuous spectrum when combined with non-local operators defined by Bernstein
functions of the Laplacian. By introducing suitable H\"older-Zygmund type spaces with different scale functions than usual,
we study the action of these non-local Schr\"odinger operators in terms of second-order centered differences of
eigenfunctions integrated with respect to singular kernels. First we obtain conditions under which the potentials
decay at all, and are bounded continuous functions. Next we derive decay rates at infinity separately for operators with
regularly varying and exponentially light L\'evy jump kernels. We show situations in which no decay occurs, implying that
zero-energy eigenfunctions with specific decay properties cannot occur. Then we obtain detailed results on the sign of
potentials at infinity which, apart from asymptotic behaviour at infinity, is a second main feature responsible for the
occurrence or absence of zero eigenvalues. Finally, we study the behaviour of potentials at the origin, and analyze a delicate
interplay between the pinning effect resulting from a well at zero combined with decay and sign at infinity, as a main
mechanism in the formation of zero-energy bound states. Among the many possible examples of non-local operators, we single
out the fractional Laplacian and the massive relativistic operator, and we will derive and make extensive use of an additive
relationship between the two. In the paper we propose a unified framework and develop a purely analytic approach.

\end{abstract}

\maketitle

\baselineskip 0.5 cm

\newpage
\tableofcontents


\vspace{-1cm}
\section{Introduction}
\noindent

The main aim of this paper is to investigate the following type of problem: Given a non-local operator such as
the fractional Laplacian $(-\Delta)^{\alpha/2}$, $0 \! < \alpha < \! 2$, under what conditions does there
exist a real-valued bounded continuous function $V(x)$ on $\R^d$, tending to zero as $|x|\to\infty$, such
that the equation
$$
(-\Delta)^{\alpha/2}\varphi + V\varphi = 0
$$
has a non-zero solution $\varphi \in L^2(\R^d)$, and what are its more detailed properties. We will address this
problem in a greater generality replacing the fractional Laplacian by a Bernstein function $\Phi$ of the Laplacian,
motivated by applications as explained below and by our interest in seeing how the answer depends on the symbol
of the operator. Since, as it will be seen, $\Spec \Phi(-\Delta) = \Spec_{\rm ess} \Phi(-\Delta) = \Spec_{\rm ac}
\Phi(-\Delta) = [0,\infty)$, the problem equivalently means that we look for describing potentials $V$ decaying
to zero at infinity for which the non-local Schr\"odinger operator $H = \Phi(-\Delta) + V$ has an eigenvalue at
zero, which is thus embedded in its continuous spectrum. The technical challenge is to establish all these
properties for the object
\begin{equation}
\label{Veq}
V(x) = -\frac{1}{\varphi(x)}\Phi(-\Delta)\varphi(x), \quad x \in \R^d,
\end{equation}
in the conditions of $\Phi(-\Delta)$ being a non-local operator. In the following first we explain these aspects
and their context, and next give an outline of the results in this paper.

\vspace{0.1cm}
\noindent
\textbf{Non-local Schr\"odinger operators.}
A primary source of non-local operators is relativistic quantum theory, in which the square-root of the Laplacian
plays a special role. Generalized to other exponents,
$$
\Phi_{m,\alpha}(-\Delta) = L_{m,\alpha} = (-\Delta + m^{2/\alpha})^{\alpha/2} - m, \quad 0 < \alpha < 2, \; m > 0,
$$
and
$$
\Phi_{0,\alpha}(-\Delta) = L_{0,\alpha} = (-\Delta)^{\alpha/2}, \quad 0 < \alpha < 2,
$$
give the \emph{massive relativistic} (with rest mass $m>0$) and \emph{massless relativistic} (i.e., \emph{fractional
Laplace}) operators of index $\frac{\alpha}{2}$. The case $\alpha=1$, reproducing the \emph{square-root Klein-Gordon
operator} $L_{m,1}$, has been much studied in both the physics and mathematics literature (see below). For cases of
other values of $\alpha$ relevant in laser cooling, optics, anomalous kinetic theory etc we refer to \cite{BBAC,BG,L15,Z2}
and the references therein.
A second motivation to a study of non-local operators is provided by potential theory. Like the classical negative
Laplacian is the infinitesimal generator of Brownian motion, the operators $-L_{0,\alpha}$ and $-L_{m,\alpha}$ are
the Markov generators of rotationally symmetric $\alpha$-stable and rotationally symmetric relativistic
$\alpha$-stable processes, respectively. These are recently much studied jump L\'evy processes, with many applications
going beyond relativistic quantum theory. For a review we refer to \cite{Betal} and the subsequent developments.

These and many other non-local operators, as well as the classical Laplacian, can be seen as specific cases of a
family defined in terms of Bernstein functions of the Laplacian. A Bernstein function $\Phi:(0,\infty) \to\R$ can be 
canonically represented as
$$
\Phi(u) = k + bu + \int_{(0,\infty)} (1 - e^{-yu}) \mu(dy),
$$
where $k, b \geq 0$ and $\mu$ is a Borel measure with mass on the strictly positive semi-axis only (see Section 2.1
below for the details). In this paper we make the choice $k = 0 = b$, and retain the measure $\mu$ only as an input
parameter. Using functional calculus we can then define the operators
\begin{equation}
\label{subord}
\Phi(-\Delta) = \int_{(0,\infty)} (1 - e^{t\Delta}) \mu(dt).
\end{equation}
Their prominent interest in potential theory is due to the fact that such $\Phi$ are Laplace exponents of
subordinators, i.e., of almost surely non-decreasing jump L\'evy processes $\pro {S^\Phi}$ so that $\ex^0[e^{-uS^\Phi_t}] =
e^{-t\Phi(u)}$ holds for every $u,t\geq 0$, where the expectation is taken with respect to the probability measure
of the subordinator. This implies $e^{-t\Phi(-\Delta)}f(x) = \ex^x[f(B_{S_t^\Phi})]$, where $\pro B$ is $\R^d$-valued
Brownian motion and the new process $(B_{S^\Phi_t})_{t\geq 0}$ is a jump L\'evy process called subordinate Brownian
motion, i.e., Brownian motion sampled at the random times given by the paths of the subordinator $\pro {S^\Phi}$.
We have explored this useful representation previously \cite{HIL12,HL12}, for a detailed discussion see also
\cite[Ch. 4]{LHB}, but in this paper we take another approach and will not use a probabilistic language. Besides
potential theory, see \cite{SSV,KSV}, Bernstein functions of the Laplacian are more recently used also in other
directions such as maximum principles for non-local equations \cite{BL19a,BL19b}, a generalization of the
Caffarelli-Silvestre extension technique \cite{KM}, and the blow-up of solutions of stochastic PDE with white or
coloured noise \cite{DLN}.

Take now a Borel measurable function $V: \R^d\to\R$ called potential, and using it as a multiplication operator,
consider the non-local Schr\"odinger operator
\begin{equation}
\label{nonlocSch}
H = \Phi(-\Delta) + V,
\end{equation}
which is a counterpart of the classical (local) Schr\"odinger operator obtained for $\Phi(u) = \frac{1}{2}u$.
Fractional Schr\"odinger operators of the type $H = L_{0,\alpha} + V = (-\Delta)^{\alpha/2} + V$ have been first
studied from a potential theory point of view in \cite{BB99,BB00}. An analysis from a different perspective has been
made in \cite{KL12}, and explicit solutions of the eigenvalue problem for the cases $d=\alpha=1$ and $V(x) = x^2$
(massless relativistic harmonic oscillator), and $V(x) = x^4$ (massless relativistic quartic oscillator) have been
presented in \cite{LM,DL}. The relativistic Schr\"odinger operator $H = L_{m,1} + V$ has been much studied, see e.g.
\cite{W74,H77,CMS,LS10,HIL17,RSV}, also with extra terms involving magnetic fields or spin \cite{HIL13,HIL17}. For more
aspects of the spectral theory of non-local Schr\"odinger operators we refer to \cite{HIL12,HL12,HL14,KKM,AB,FLS,JW,KMV}.
For a further class of operators which are Markov generators of what we call jump-paring L\'evy processes, having a
partial overlap with Bernstein functions of the Laplacian, we have obtained detailed results on the asymptotic decay of
eigenfunctions at infinity in \cite{KL15,KL17}, and further studied a hierarchy of contractivity properties of the related
Schr\"odinger semigroups in \cite{KKL}. Several of our results show that non-local Schr\"odinger operators can produce some
qualitatively different behaviours from classical Schr\"odinger operators. We note that further developments include random
non-local Schr\"odinger operators \cite{KPP18,KPP20}, and also variants of the fractional Laplacian or the fractional
$p$-Laplacian in the context of non-linear Schr\"odinger operators, which go beyond the scope of our paper.

\vspace{0.1cm}
\noindent
\textbf{Behaviour at the spectral edge.}
The spectral behaviour at and around the continuum edge of classical Schr\"odinger operators is known to produce
intricate phenomena, whose analysis led to the development of sophisticated methods. Results on the two aspects
of occurrence or absence of embedded eigenvalues in the continuous spectrum are of equal interest. For classical
Schr\"odinger operators $H=-\frac{1}{2}\Delta + V$ on $L^2(\R^d)$ we refer to \cite{A70,JK79,KS,L81,S81,R87,K89,KN00,
KT02,FS04,DS09,SW}, and the syntheses in \cite{RS3,EK,CK01,DS07} providing further references. The results indicate
that the existence of positive embedded eigenvalues is a long range effect, and the appearance of positive point
spectrum is a combination of slow decay and oscillations of the potential.

In the borderline case of zero eigenvalues the understanding of the nature of potentials is more limited and calls
for further investigation. Zero-eigenvalue/resonance cases relate in a sense to incipient bound states
($L^2$-eigenfunctions), giving important insight into the mechanisms of the formation of stable quantum
states, such as enhanced binding and the Efimov effect \cite{T91,S93,JRF}, but such potentials appear, for instance,
also in wetting phenomena \cite{ZLK}. For work on the existence of zero-energy eigenvalues and the spatial decay of
zero-energy bound states of classical Schr\"odinger operators we refer to \cite{BY90,ABG,DS09,FS04,HS,SW}, and for
related low-energy scattering theory see also \cite{Y82,N94}. The absence of zero eigenvalues has a further relevance.
One reason for which zero eigenvalues can become important is that they may be accumulation points of negative
eigenvalues. We refer to \cite{Sch07} for a discussion of dispersive estimates related to time evolutions of projections
to the continuous spectrum under the unitary Schr\"odinger semigroup in the absence of a zero eigenvalue. For another
application see \cite{A09,AH}, in which the existence of time operators is considered, which is a rigorous description
of the energy-time complementarity principle in quantum mechanics, and in which the non-existence of zero eigenvalues
is an important condition.

While the study of zero eigenvalues of classical Schr\"odinger operators involves various technical complications,
the (partially rigorously proven) picture that emerges from the above and related results is the following. Suppose
that the potential has a decay at infinity like $V(x) \sim C|x|^{-\gamma}$, $\gamma > 0$, with a definite sign for large
enough $|x|$. Roughly speaking, there
are two qualitatively different cases. On the one hand, if $V$ is positive at infinity, then there is no zero eigenvalue
if $\gamma > 2$, there may be one if $\gamma < 2$, and if $\gamma = 2$, then there is a critical coupling constant $C^*
> 0$ such that there is no zero eigenvalue if $C \leq C^*$ and there is an eigenfunction at zero eigenvalue if $C > C^*$.
On the other hand, if $V$ is negative at infinity, then there is no zero eigenvalue if  $0 < \gamma < 2$ and there may be
otherwise. While a truly rigorous proof explaining the mechanisms behind these behaviours does not seem readily available,
one may think that for $V$
with a positive tail the main reason for the decay cannot be too rapid is that the potential barrier would otherwise be
too ``thin" to be sufficiently efficient to repel motion back in the bulk and thus support an eigenfunction, while a
slowly decaying potential negative at infinity is attracting to a too large degree arbitrarily far from the origin. In
a Feynman-Kac representation this picture can be appreciated even more meaningfully.

Apart from classical Schr\"odinger operators, more recent work started to extend also to relativistic Schr\"odinger
operators. For massless operators $H = (-\Delta)^{1/2} + V$, by an analysis of the resolvent around zero it has been
shown in \cite{RU16} that for $V \in L^3(\R^3)$ the set of potentials for which zero is not
in the point spectrum contains an open and dense subset of $L^3(\R^3)$. A further result is that if zero is not an
eigenvalue, then it cannot be the accumulation point of positive eigenvalues, which has no analogue for classical
Schr\"odinger operators. Moreover, if $|V(x)| \leq (1+|x|^2)^{-\gamma/2}$, $x \in \R^3$, for $\gamma > 1$, and zero
is not an eigenvalue, then $H$ has no zero-resonances. Also, for $d=3$ the same operator has no non-negative
eigenvalues provided $|V|$, $|x\cdot\nabla V|$ and $|x\cdot\nabla (x\cdot\nabla V)|$ are jointly bounded by
$C(1+|x|^2)^{-1/2}$ and $C > 0$ is small enough. In \cite{LS} we have obtained further results on non-existence of
embedded eigenvalues. Related work on unique continuation for fractional Schr\"odinger equations imply further
non-existence results \cite{S14,FF15,S15,R15,R17,RW}.

In contrast, in the work \cite{LS17} we have constructed potentials for the massive relativistic operator $L_{m,1}$
such that for a sufficiently large rest mass $m > 0$ a strictly positive eigenvalue $\sqrt{1+m^2}-m$ exists, and
another set of potentials for which zero is an eigenvalue for $L_{0,1}$. The possibility of zero eigenvalues for the
massive relativistic operator has been studied for a non-positive potential with compact support in \cite{M06}. In
a further development \cite{JL18} we constructed explicit potentials on $\R^d$, $d\geq 1$, for $(-\Delta)^{\alpha/2}$
with arbitrary order $0 < \alpha < 2$. Specifically, let $\kappa>0$, $\alpha \in (0,2)$, and $P$ be a harmonic polynomial,
homogeneous of degree $l \geq 0$, i.e., satisfying $P(cx) = c^l P(x)$ for all $c>0$, and $\Delta P = 0$. Denote $\delta
= d+2l$, and consider the potentials and functions
\begin{equation}
\begin{aligned} \label{eq:motiv_ex}
V_{\kappa,\alpha}(x) & =  -\frac{2^\alpha}{\Gamma(\kappa)} \Gamma\left(  \frac{\delta+\alpha}{2}\right)
\Gamma\left( \frac{\alpha}{2} + \kappa \right)(1+|x|^2)^\kappa \,_2\textbf{F}_1\left( \left. \begin{array}{c}
				\frac{\delta+\alpha}{2} \quad \frac{\alpha}{2} + \kappa \\ \frac{\delta}{2}
			\end{array}  \right| -|x|^2 \right) \\
\varphi_\kappa(x) & = \frac{P(x)}{(1+|x|^2)^\kappa},
\end{aligned}
\end{equation}
where $_2\textbf{F}_1$ is Gauss' hypergeometric function. Then
\begin{equation*}
(-\Delta)^\frac{\alpha}{2} \varphi_\kappa + V_{\kappa,\alpha} \varphi_\kappa = 0
\end{equation*}
holds in distributional sense with $\varphi_\kappa \in L^2(\R^d)$ if $\kappa \geq \frac{\delta}{4}$, and
\begin{equation}
\label{exdecays}
|V_{\kappa,\alpha}(x)| =
\left\{
\begin{array}{lcl}
O\left( |x|^{-\alpha} \right) &\mbox{if}& \kappa \in (l, \frac{\delta}{2}) \setminus \{ \frac{\delta-\alpha}{2} \} \vspace{0.1cm} \\
O\left( |x|^{-2\alpha} \right) &\mbox{if}& \kappa = \frac{\delta-\alpha}{2} \vspace{0.1cm} \\
O\left( |x|^{-\alpha}\log|x| \right) &\mbox{if}& \kappa = \frac{\delta}{2} \vspace{0.1cm} \\
O\left( |x|^{2\kappa-\delta-\alpha} \right) &\mbox{if}& \kappa \in (\frac{\delta}{2}, \frac{\delta+\alpha}{2}).
\end{array}
\right.
\end{equation}
Furthermore, for large $|x|$ we have that
\begin{eqnarray}
&& \hspace{-4.4cm} V_{\kappa,\alpha}(x) < 0 \quad \mbox{if} \quad \kappa \in \big( l , \frac{\delta-\alpha}{2} \big]
\label{neg} \\
&& \hspace{-4.4cm} V_{\kappa,\alpha}(x) > 0 \quad \mbox{if} \quad \kappa \in \big(\frac{\delta-\alpha}{2},\frac{\delta+\alpha}{2}\big).
\label{pos}
\end{eqnarray}

More recently, by using methods of path integration, in \cite{KL19} we have obtained detailed decay estimates on
eigenfunctions at zero eigenvalue for fractional and related Schr\"odinger operators. We have identified three different
decay scenarios distinguished by the subtle interplays between $\Phi$ and the L\'evy measure of the non-local operator,
for the fractional Laplacian reproducing specifically the behaviours \eqref{exdecays} and \eqref{neg}-\eqref{pos}. Due
to our probabilistic techniques we were also able to gain some insight into the mechanism of the decay behaviours. On
the one hand, unlike for bound states at negative eigenvalues, where the decay rates are governed by the distance of the
eigenvalue from the edge of the continuous spectrum, zero-energy bound states depend on vestigial effects of the potential
such as its sign at infinity. On the other hand, the specifics of the decay regimes depend on the large $|x|$ behaviour of
the quantity
$$
\frac{\Phi\left(\frac{1}{|x|}\right)}{V(x)} \asymp
\frac{\ex^x \left[\int_0^{\tau_{B(x,|x|/2)}} e^{- \int_0^t V(X_s) ds} dt\right]}{\ex^0[\tau_{B(0,|x|)}]},
$$
where $\pro X$ is the jump L\'evy process generated by $-\Phi(-\Delta)$, $\ex^x$ denotes expectation with respect to
the probability measure of the process starting at $x \in \R^d$, $B_r(a)$ denotes a ball of radius $r$ centered in $a$,
and $\tau_{B_r(a)} = \inf\{t > 0: \, X_t \in B_r(a)^c\}$ means the first exit time of the process from the given ball.
The result shows that the decay of zero-energy bound states are governed by global survival times (in large balls of
radii proportional to $|x|$). This is in sharp contrast with the case of bound states with negative eigenvalues
\cite{KL17,KL16} or for confining potentials \cite{KL15}, where the decay is played by local survival times (in balls
of unit radius), i.e., how soon paths leave local neighbourhoods far out. This also further indicates that a potential
leading to a zero eigenvalue must have special features such as discussed above, more than a potential creating a
negative eigenvalue which has a comfortable energy margin from the spectral edge. In analytic terms, this translates
into a competition at infinity between the symbol of the operator and the potential.

\medskip
\noindent
\textbf{Outline of results.}
Motivated by these developments, in this paper our goal is to describe the potentials for which non-local Schr\"odinger
operators of the form \eqref{nonlocSch} with \eqref{subord} have a zero eigenvalue or a zero resonance. Our approach is
purely analytic. Classical Schr\"odinger operators will be discussed in a future work, but our results in this paper
provide a context against which perturbations of the Laplacian can be understood in a different light. Assuming that $H$
has an eigenfunction $\varphi \in \Dom (\Phi(-\Delta)) \subset L^2(\R^d)$, $\varphi \not \equiv 0$, at eigenvalue zero, 
solving the eigenvalue equation
$$
H\varphi = 0,
$$
we obtain for the expression of the potential formula \eqref{Veq} or its equivalent
\begin{equation}
\label{Veqq}
V(x) = \frac{1}{\varphi(x)} \int_{(0,\infty)} \left(e^{t\Delta}\varphi(x) - \varphi(x)\right) \mu(dt),
\end{equation}
which will be the main object of our analysis. Note that $e^{t\Delta}\varphi(x) =
\int_{\R^d}p_t(x,y)\varphi(y)dy$, $t>0$,
with integral kernel (in a slight abuse of notation)
\begin{equation}
\label{pt}
p_t(x,y) = p_t(x-y) = \frac{1}{4\pi t} e^{-\frac{|x-y|^2}{4t}}, \quad t > 0, \; x,y \in \R^d,
\end{equation}
which is weakly convergent to $\delta(x-y)$ in the $t\downarrow 0$ limit. Informally, the integral at the right-hand
side of \eqref{Veqq} can be expressed as $\int_{(0,\infty)}\int_{\R^d} (p_t(x,y) - \delta(x-y))\varphi(y)dy \mu(dt)$,
so the regularization of the ``diagonal part" on subtracting the delta-kernel acts to compensate the $t=0$ singularity
of $\mu$ allowing the resulting function to be well-defined for suitable $\varphi$. A further difficulty in obtaining a
bounded potential $V$ as given by \eqref{Veqq} is that by division through $\varphi$ the zeroes of the expressions
involved in the ratio need to be matched, which is generally hard to control for non-local operators.

For the purposes of our analysis, in Section 2.3 below we will derive the representation
\begin{align*}
	\Phi(-\Delta)f(x)&= -\frac{1}{2}\int_{\R^d}\DDv  j(|h|)dh
	\end{align*}
with
$$
\DDv = f(x+h)-2f(x)+f(x-h)
$$
and $j(r)= \int_0^{\infty} p_t(r)\mu(dt)= (4\pi)^{-d/2}\int_0^{\infty}t^{-\frac{d}{2}}e^{-\frac{r^2}{4t}}\mu(dt)$,
and introduce suitable H\"older-Zygmund type spaces on which it holds, however, which have different scale
functions than usual. While we will develop our investigation in the generality of the operators \eqref{nonlocSch}
with kinetic term
$\Phi(-\Delta)$ in order to see the contribution of the symbol in the behaviours, we will also single out the cases
$L_{0,\alpha}$ and $L_{m,\alpha}$ due to their special interest. In these cases there are explicit formulae available
(see \eqref{mumassless}-\eqref{mumassive} below), and they also highlight the
qualitatively different behaviours of operators with polynomially decaying L\'evy measures $\mu$ at infinity allowing
a heavy tail (massless operator), and those with exponential decay leading to a much lighter tail (massive operator).
Furthermore, we will explore and develop the following relationship between the two operators, which can be deduced
from a result in \cite{Ryz}, where the potential theory of the relativistic stable process has been investigated.
Formally we can decompose the massive operator in terms of the massless and another term like
\begin{equation}
\label{Ryznar}
L_{m,\alpha}f = L_{0,\alpha} f - (\sigma_{m,\alpha} - \delta_0) \ast f,
\end{equation}
where $\sigma_{m,\alpha}(x)dx$ is a finite measure for which we give an explicit representation (see \eqref{sigma}
below), $\delta_0$ is Dirac delta concentrated on zero, and the star denotes convolution. As far as we are aware,
this relationship has not been observed in the literature apart from its form cited, and we will discuss it rigorously
in Section \ref{maasivemassless} below and use it repeatedly.

With input parameters $\mu$ and $\varphi$, we are interested in the following main questions:
\begin{trivlist}
\item[\; (1)]
\emph{Decaying potentials.} A first question is under what conditions $V(x) \to 0$ as $|x| \to \infty$ at all.
This will be answered in Theorems \ref{thm3}-\ref{thm4-5} below, and it will be seen that this is decided by how
the control function in $x$ of the centered second-order differences $\Dv$, the tails and the second moment of
the L\'evy measure in dilated balls compare with the decay rate of the eigenfunction $\varphi$.

\vspace{0.1cm}
\item[\; (2)]
\emph{Decay rates.}
In Theorem \ref{thm6} we obtain the rates of decay of potentials for operators with regularly varying jump measures.
Apart from an extra slowly varying factor which $\Phi$ may include, the leading terms coincide with the behaviours
given by \eqref{exdecays}, showing that the operators in this category behave like the fractional Laplacian. However,
by tracking the constants one can see the contributions of the individual features of the non-local operator
$\Phi(-\Delta)$ and of the control functions related to the eigenfunctions. The case of operators with exponentially
concentrated L\'evy measures shows a very different behaviour. Whereas for operators with regularly varying L\'evy
measures the contribution of local terms (i.e., in bounded neighbourhoods of positions where the decay is analyzed)
plays almost no role while the far-away behaviour is crucial, for exponentially light L\'evy measures the local terms
dominate. In Theorem \ref{rateexp} we prove a generic $O(|x|^{-2})$ behaviour. Requiring a control on $\Dv$ also from
below, under an excess (lack of balance) condition between the moments of the L\'evy measure in regions where $\Dv$ is
positive or negative, in Theorem \ref{rateexp2} we show that an $|x|^{-2}$--like behaviour holds also from below. In
Remarks \ref{shortrange1} and \ref{shortrange2} (3) we also comment on the possibility of any shorter-range potentials,
and further refine these results in Theorems \ref{expdecaype} and \ref{rateexp3}. Interesting consequences of these
behaviours are $L^p$-integrability properties of the potentials, see Theorems \ref{Lp1}-\ref{Lp2}.

\vspace{0.1cm}
\item[\; (3)]
\emph{Non-decaying potentials.}
The expression \eqref{Veq} does not decrease by default on increasing $|x|$, and this can be used to show non-existence
of eigenfunctions at zero eigenvalue of particular properties. In Theorem \ref{nogo1} we show that the potential goes
to infinity if the operator has a slowly varying jump measure and the eigenfunctions are in leading order subexponential
or faster decaying. In Theorems \ref{nogogen}-\ref{nogogen2} we obtain similar results for exponentially light jump
measures, with further degrees of no-decay (e.g., two-sided bounded potentials with bounds away from zero) and dependent
on the values of input parameters.

\vspace{0.1cm}
\item[\; (4)]
\emph{Regularity of the potential.}
Although it is possible to extend many results to Kato-class potentials, which may have local singularities, we
limit our study to bounded continuous potentials not to make this paper longer. These properties are shown in
Theorems \ref{thm1}-\ref{thm2} and Corollary \ref{corbound}.

\vspace{0.1cm}
\item[\; (5)]
\emph{Sign at infinity.}
While von Neumann-Wigner potentials generating strictly positive embedded eigenvalues in the case of classical
Schr\"odinger operators are oscillating and this is a key feature, one may wonder if this also holds
for zero eigenvalues. From the examples \eqref{exdecays} and previous work \cite{LS,KL19} there was indication
that, at least for non-local cases, potentials giving rise to zero eigenvalues may have a definite sign at
infinity, however, positivity or negativity makes a qualitative difference. Here we prove the sign properties
of the potentials at infinity separately for operators with regularly varying and exponentially light jump
measures. For regularly varying cases, in Theorems \ref{positive1}-\ref{V+} we show positivity, and in Theorem
\ref{V-} we show negativity of the potentials at infinity. Since in the latter two theorems we arrive at somewhat
less immediate criteria, in Proposition \ref{suffi+-} we give some sufficient conditions more directly in terms
of the input parameters, and show that positivity generally follows for sufficiently rapidly decaying
eigenfunctions, and negativity for low values of what would be the exponent of $(-\Delta)^{\alpha/2}$. For
exponentially light cases we have Theorems \ref{signexp1} and \ref{signexp3} showing positivity, and Theorem
\ref{signexp2} showing negativity at infinity. For the massive relativistic operator we have also Proposition
\ref{massive+-} showing the sign patterns for annular regions approaching infinity. In Corollaries
\ref{corsign1},  \ref{corsign2} and  \ref{corsign3} we note that the possible zeroes of the potentials are away
from particular regions. From the proofs we can see that the impact of the local and remote terms is similar to
how they act in determining the decay behaviours.

\vspace{0.1cm}
\item[\; (6)]
\emph{Qualitative properties of the potential.}
In Theorems \ref{radsym}-\ref{reflsym2} we show rotation and reflection symmetries of the potentials. In Theorem
\ref{minatzero} we prove that the origin is the location of a strict local minimum of the potential if it is a
(local) maximizer of the eigenfunction. In Theorem \ref{Vnought} we show that the potential is negative in the origin,
which then inspires the analysis of the pinning effect in Section \ref{pin}. Indeed, complementing the picture described
above for classical Schr\"odinger operators, we show in our set-up that there is an interesting interplay between the
depth of the negative well around zero on the one hand, and the sign and decay rate of potentials at infinity, on the
other. A potential having a positive part at infinity can ``afford" to develop a shallower well in the origin due to
its repelling part, however, a purely negative potential must follow another mechanism, only relying on the attracting
force resulting from a sufficiently deep well around zero.
\end{trivlist}

The remainder of this paper is organized as follows. In Section 2 we discuss all the necessary estimates on the jump
kernels entering the class of non-local operators, and give the representations of the operators on suitable
H\"older-Zygmund type function spaces. In Section 3 we obtain conditions under which the potentials are bounded and
continuous. In Section 4 we consider decay and $L^p$-properties of the potentials. Section 5 is devoted to cases when
the potentials do not decay to zero at infinity, ruling out the existence of eigenfunctions at zero eigenvalue with
particular properties. In Section 6 we discuss the sign of the potentials at infinity. Finally, in Section 7 we show
symmetry properties of the potentials, and further focus on their behaviour at zero and the pinning effect.

\section{Non-local Schr\"odinger operators with Bernstein functions of the Laplacian}

\subsection{Bernstein functions and L\'evy measures}

Recall that a Bernstein function is a non-negative completely monotone function, i.e., an element of the
convex cone
$$
\mathcal B = \left\{f \in C^\infty((0,\infty)): \, f \geq 0 \;\; \mbox{and} \:\; (-1)^{n}\frac{d^n f}{dx^n} \leq 0,
\; \mbox{for all $n \in \mathbb N$}\right\}.
$$
In particular, Bernstein functions are increasing and concave. A standard reference on Bernstein functions is
\cite{SSV}. Below we will restrict to the subset
$$
{\mathcal B}_0 = \left\{f \in \mathcal B: \, \lim_{u\downarrow 0} f(u) = 0 \right\}.
$$
Let $\mathcal M$ be the set of Borel measures $\mu$ on $\R \setminus \{0\}$ with the property that
\begin{equation}
\label{Mset}
\mu((-\infty,0)) = 0 \quad \mbox{and} \quad \int_{\R\setminus\{0\}} (y \wedge 1) \mu(dy) < \infty.
\end{equation}
Note that every $\mu \in \mathcal M$ is a L\'evy measure. Bernstein functions $\Phi \in {\mathcal B}_0$ can be
canonically represented in the form
\begin{equation}
\Phi(u) = bu + \int_{(0,\infty)} (1 - e^{-yu}) \mu(dy)
\label{bernrep}
\end{equation}
with $b \geq 0$, and the map $[0,\infty) \times \mathcal M \ni (b,\mu) \mapsto \Phi \in {\mathcal B}_0$ is
bijective, for details we refer to \cite[Th. 3.2]{SSV}. The parameters $(b,\mu)$ make the characteristic L\'evy
pair of $\Phi$.

Next recall that $\Phi \in \mathcal B_0$ is called a complete Bernstein function if its L\'evy measure $\mu(dt)$
is absolutely continuous with respect to Lebesgue measure and its density (called L\'evy intensity), which we
will keep denoting for simplicity by $\mu(t)$, is a completely monotone function. The following result holds in
particular for complete Bernstein functions.

\begin{lem}\label{lem1}
Let $\mu \in \mathcal M$ such that $\mu(dt)=\mu(t)dt$, with a decreasing non-negative density $\mu(t)$. Then
for every $C>0$ there exists $t_0(C) \in (0,1)$ such that $\mu(t)\le C t^{-2}$, for every $t \in (0,t_0(C))$.
\end{lem}
\begin{proof}
Suppose, to the contrary, that there exists $\widetilde{C}>0$ and a decreasing sequence
$(t_n)_{n \ge 1}$ such that $t_1<1$, $t_n>0$ for all $n \in \N$, $\mu(t_n)>\widetilde{C}t_n^{-2}$,
$t_{n-1}-t_n>\frac{t_{n-1}}{2}$ and $t_n \to 0$. Since $\mu \in \mathcal M$, it satisfies $\int_0^1 t\mu(t)dt<
\infty$. Write
	$\int_0^1t\mu(t)dt=\int_{t_1}^1t\mu(t)+\sum_{n=1}^{\infty}\int_{t_{n+1}}^{t_n}t\mu(t)dt$.
	Since $\mu$ is decreasing, we have
	$\int_{t_1}^1t\mu(t)\ge t_1\mu(1)(1-t_1)$.
	Moreover, for any $n\ge 1$,
	\begin{align*}\label{lem1p3}
	\begin{split}
	\int_{t_{n+1}}^{t_n}t\mu(t)dt&\ge \mu(t_{n+1})\int_{t_{n+1}}^{t_n}tdt=\mu(t_{n+1})
\frac{(t_{n}-t_{n+1})(t_{n}+t_{n+1})}{2}>\widetilde{C}\frac{t_nt_{n+1}}{4t_{n+1}^2}>\frac{1}{4}.
	\end{split}
	\end{align*}
	Hence we get
	$\int_0^1t\mu(t)dt=\int_{t_1}^1t\mu(t)+\sum_{n=1}^{\infty}\int_{t_{n+1}}^{t_n}t\mu(t)dt =\infty$,
	which is a contradiction.
\end{proof}

In the following we will be interested, among others, in asymptotic properties of the expression \eqref{Veq},
which can be related to the asymptotics of $\Phi$ and of the L\'evy measure $\mu$. Recall that a measurable
function $\ell:[0,\infty) \to [0,\infty)$ is said to be slowly varying at infinity, if for every $c>0$
	\begin{equation*}
	\lim_{t \to \infty}\frac{\ell(ct)}{\ell(t)}=1
	\end{equation*}
holds. Furthermore, a measurable function $\ell:(0,\infty) \to (0,\infty)$ is said to be slowly varying at zero if the
function $\well(t)=\ell(1/t)$ is a slowly varying function at infinity. It is known that for every $\eta>0$ and any
slowly varying function $\ell$ at infinity
\begin{equation*}
\label{slowinf}
\lim_{t \to \infty}t^{-\eta}\ell(t)=0, \quad \lim_{t \to \infty}t^\eta \ell(t)=\infty
\end{equation*}
hold. Also, the following estimate due to Potter is known for $\ell$ slowly varying at infinity: for every $\delta, A>0$
there exists $M(\delta,A) > 0$ such that for all $t,s>M(\delta,A)$
	\begin{equation}
\label{potbound}
	\frac{\ell(t)}{\ell(s)}\le A \max\left\{\Big(\frac{t}{s}\Big)^\delta,\left(\frac{s}{t}\right)^{\delta}\right\}.
	\end{equation}
For details and proofs we refer to \cite{BGT}.
We will use the following standard notations. Given two real functions $f,g$ defined on $(0,\infty)$ and $t_0 \in
[0,\infty]$, we write $f(t) \sim g(t)$ as $t \to t_0$ to mean that
$\lim_{t \to t_0}g(t)/f(t)=1$.
Also, given two functions $f:\R^d \to \R$ and $g:\R^d \to [0,\infty)$ the notation $f(x)=O(g(x))$ means that there
exist $R,C>0$ such that $|f(x)|\le Cg(x)$ for $|x| \geq R$. Finally, $f \asymp g$ means that there exist constants
$0 < C_1 < C_2$ such that $C_1 g \leq f \leq C_2 g$.

The following asymptotic result on the L\'evy intensity of a complete Bernstein function is known
\cite[Prop. 2.23(ii)]{KSV}.
\begin{prop}\label{prop1}
Let $\Phi \in \mathcal B_0$ be a complete Bernstein function with L\'evy measure $\mu(dt)=\mu(t)dt$. Suppose that
there exist $\alpha \in (0,2]$ and a function $\ell$ slowly varying at zero such that $\Phi(u)\sim u^{\alpha/2}\ell(u)$
as $u \downarrow 0$. Then
	\begin{equation*}
	\mu(t)\sim \frac{\frac{\alpha}{2}}{\Gamma\left(1-\frac{\alpha}{2}\right)}t^{-1-\frac{\alpha}{2}}\well(t),
\quad t \to \infty,
	\end{equation*}
holds, where $\well(t)=\ell(1/t)$.
\end{prop}

\begin{example}\label{Eg2.1}
	\rm{
		Some examples of complete Bernstein functions include:
		\begin{enumerate}
			\item
			$\Phi(u)=u^{\alpha/2}, \, \alpha\in(0, 2]$ \vspace{0.1cm}
			\item
			$\Phi(u)=(u+m^{2/\alpha})^{\alpha/2}-m$, $m> 0$, $\alpha\in (0, 2]$  \vspace{0.1cm}
			\item
			$\Phi(u)=u^{\alpha/2} + u^{\beta/2}, \, \alpha, \beta \in(0, 2]$  \vspace{0.1cm}
			\item
			$\Phi(u)=u^{\alpha/2}(\log(1+u))^{-\beta/2}$, $\alpha \in (0,2)$, $\beta \in [0,\alpha)$  \vspace{0.1cm}
			\item
			$\Phi(u)=u^{\alpha/2}(\log(1+u))^{\beta/2}$, $\alpha \in (0,2)$, $\beta \in (0, 2-\alpha)$.
		\end{enumerate}
		In contrast, the Bernstein function $\Phi(u) = 1-e^{-u}$ is not a complete Bernstein function.
	}
\end{example}
\subsection{Jump kernels of complete Bernstein functions}
\subsubsection{\textbf{Properties of subordinate jump kernels}}
Recall \eqref{pt}, and let $\Phi \in \mathcal B_0$ with L\'evy pair $(0,\mu)$. We define the \emph{jump kernel} on
$\R^d$ associated with $\Phi$ by
\begin{equation}
\label{jumpker}
(0,\infty) \ni r \mapsto j(r)= \int_0^{\infty} p_t(r)\mu(dt) =
\frac{1}{(4\pi)^\frac{d}{2}}\int_0^{\infty}t^{-\frac{d}{2}}e^{-\frac{r^2}{4t}}\mu(dt),
\end{equation}
and the related L\'evy measure
\begin{equation}
\label{nu}
\nu(dx)=j(|x|)dx, \quad  x \in \mathbb R^d \setminus \{0\}.
\end{equation}

The following basic properties of jump kernels will be used throughout below and therefore we provide a proof for
self-containedness. We will make use of the upper and lower incomplete Gamma functions
	\begin{equation}
	\label{gamma}
	\Gamma(s,x)=\int_{x}^{\infty}t^{s-1}e^{-t}dt \quad \mbox{and} \quad
	\gamma(s,x)=\int_0^x t^{s-1}e^{-t}dt,
	\end{equation}
respectively.

\begin{prop}
Let $\Phi \in \mathcal B_0$ with L\'evy pair $(0,\mu)$, and $j$ be the jump kernel associated with
$\Phi$. The following properties hold:
\begin{enumerate}
\item
$j(x)<\infty$, for every $x \in \R^d$;
\item
$j(x)$ is a decreasing function;
\item
$j(x)$ is continuous in $\mathbb R^d$ and $\lim_{|x| \to \infty}j(x)=0$;
\item
$\int_{1}^{\infty}r^{d-1}j(r)dr<\infty$ and $\int_0^1r^{d+1}j(r)dr<\infty$;
\item
the measure $\nu(dx)=j(|x|)dx$  on $\R^d\setminus \{0\}$ is a L\'evy measure;
\item
the function $\R^d \ni x \mapsto j(|x|) \in \R$ belongs to $L^p(B_{\varepsilon}^c(0))$, for every
$\varepsilon>0$ and $p \in [1,\infty]$.
\end{enumerate}
\end{prop}
\begin{proof}
(1) Fix $r>0$ and with $|x|=r$ write
	\begin{align*}
	j(r) = \frac{1}{(4\pi)^{d/2}} \left(\int_{0}^1 + \int_{1}^{\infty}\right) t^{-\frac{d}{2}}e^{-\frac{r^2}{4t}}\mu(dt)
	=\frac{1}{(4\pi)^{d/2}}(I_1(r)+I_2(r)).
	\end{align*}
Since the integrand goes to zero as $t \downarrow 0$, we can choose $C(r)$ for which $t^{-\frac{d}{2}}e^{-\frac{r}{4t}}
\le C(r)t$ for every $t \in [0,1]$. Using the integrability of $\mu \in \mathcal M$ we get $I_1(r) \le C(r)\int_0^{1}t
\mu(dt)<\infty$. Furthermore, since $t^{-\frac{d}{2}}e^{-\frac{r}{4t}}\le 1$ for $t \in [1,\infty)$, similarly $I_2(r)
\le\int_1^{\infty} \mu(dt)<\infty$.
(2) is straightforward, and (3) is immediate by monotone convergence. Consider (4) and the first integral.
	\begin{equation*}
	\int_1^{\infty}j(r)r^{d-1}dr=\frac{1}{(4\pi)^{d/2}}
\int_{0}^{\infty}t^{-\frac{d}{2}}\left(\int_{1}^{\infty}r^{d-1}e^{-\frac{r^2}{4t}}dr\right)\mu(dt).
	\end{equation*}
	In the inner integral we make the change of variable $r=2\sqrt{st}$ giving
	$2^{d-1}t^{\frac{d}{2}}\int_{1/(4t)}^{\infty}s^{\frac{d}{2}-1}e^{-s}ds =
    2^{d-1}t^{\frac{d}{2}}\Gamma\left(\frac{d}{2},\frac{1}{4t}\right)$,
	with the upper incomplete Gamma function \eqref{gamma}.
	Hence we have
	\begin{equation*}
	\int_1^{\infty}j(r)r^{d-1}dr\le
      \frac{1}{2\pi^{\frac{d}{2}}} \left(\int_0^{1} + \int_{1}^{\infty}\right) \Gamma\left(\frac{d}{2},\frac{1}{4t}\right)\mu(dt) =
     \frac{1}{2\pi^{\frac{d}{2}}}(I_3+I_4).
	\end{equation*}
	Using $\Gamma\left(\frac{d}{2},\frac{1}{4t}\right)\le \Gamma\left(\frac{d}{2}\right)$ and the integrability of $\mu$ imply
	$I_4\le \Gamma\left(\frac{d}{2}\right)\int_1^{\infty}\mu(dt)<\infty$.
	On the other hand, since
    $\lim_{x \to \infty}x^{1-s}e^{x}\Gamma(s,x)=1$, there exists  $C>0$ such that
	$\Gamma\left(\frac{d}{2},\frac{1}{4t}\right)\le C 2^{-d+2}t^{-\frac{d}{2}+1}e^{-\frac{1}{4t}}$ for every $t \in (0,1)$.
	Moreover, since $\lim_{t \downarrow 0}t^{-\frac{d}{2}+2}e^{-\frac{1}{4t}}=0$, there exists another constant $C>0$ such that
	$\Gamma\left(\frac{d}{2},\frac{1}{4t}\right)\le Ct$ for $t \in (0,1)$.
	Thus we have
	$I_3\le C \int_0^1 t\mu(dt)<\infty$, using again the integrability of $\mu$.
Considering the second integral in (4), by Fubini's theorem again
	\begin{equation*}
	\int_0^{1}j(r)r^{d+1}dr=\frac{1}{(4\pi)^{d/2}}\int_{0}^{\infty}t^{-\frac{d}{2}}\left(\int_{0}^{1}r^{d+1}e^{-\frac{r^2}{4t}}dr\right)\mu(dt).
	\end{equation*}
	Making the same change of variable in the inner integral as before and using the lower incomplete Gamma function \eqref{gamma} leads to
	\begin{equation*}
	\int_0^{1}j(r)r^{d+1}dr= \frac{2}{\pi^{\frac{d}{2}}}\left(\int_0^1+\int_{1}^{\infty}\right)t\gamma\left(\frac{d}{2}+1,\frac{1}{4t}\right)\mu(dt)
    = \frac{2}{\pi^{\frac{d}{2}}}(I_5+I_6).
	\end{equation*}
    Through similar steps as before, using $\gamma\left(\frac{d}{2}+1,\frac{1}{4t}\right)\le \Gamma\left(\frac{d}{2}+1\right)$, we get
	$I_5\le \Gamma\left(\frac{d}{2}+1\right)\int_0^{1}t\mu(dt)<\infty$. Since $\lim_{x \to 0}\frac{\gamma(s,x)}{x^s}=\frac{1}{s}$, there is a
    constant $C>0$ such that $\gamma\left(\frac{d}{2}+1,\frac{1}{4t}\right)\le C 2^{-d-2}t^{-\frac{d}{2}-1} \le C t^{-1}$ for $t \in
    [1,\infty)$, and so $I_6 \le C \int_{1}^{\infty}\mu(dt)<\infty$. Next considering part (5), it is immediate by using (4) that
	\begin{align*}
	\int_{\R^d\setminus \{0\}}(1 \wedge |x|^2)\nu(dx)
	=d\omega_d\left(\int_0^{1}r^{d+1}j(r)dr+\int_1^{\infty}r^{d-1}j(r)dr\right)<\infty.
	\end{align*}

\noindent
(6) The case $p=1$ is an immediate consequence of the fact that $\nu(dx)=j(|x|)dx$ is a L\'evy measure.
For $p=\infty$ it follows from (2) since for every $x \in B_\varepsilon^c(0)$ we have $j(|x|)\le j(\varepsilon)$. Consider
$p \in (1,\infty)$. Since $j$ is decreasing and $\lim_{r \to \infty}j(r)=0$, there exists $M>0$ such that $j(|x|)<1$ for every
$x \in B_M^c(0)$. This gives $j^p(|x|)< j(|x|)$ for every $x \in B_M^c(0)$, and if $\varepsilon\ge M$, completes the proof. If
$\varepsilon<M$, then we can use that
	\begin{align*}
	\int_{B_\varepsilon^c(0)}j^p(|x|)dx&= \left(\int_{B_M(0)\setminus B_\varepsilon(0)} + \int_{B_M^c(0)}\right)j^p(|x|)dx
	\le j^p(\varepsilon)\omega_d (M^d-\varepsilon^d)+\nu(B_M^c(0))<\infty.
	\end{align*}
\end{proof}
In the following we will be interested in the asymptotic behaviour of the jump kernel $j$ under various basic assumptions on
the related complete Bernstein function $\Phi$ or L\'evy intensity $\mu$. We will focus on two different types of behaviours
of the L\'evy measure: first L\'evy intensities which are regularly varying at infinity (i.e., with polynomially heavy tails),
and secondly exponentially light L\'evy intensities.

\subsubsection{\textbf{Regularly varying L\'evy intensities}}
Instead of requiring the L\'evy intensity $\mu$ to be regularly varying, we will directly use regularly varying Bernstein
functions $\Phi$ since by Proposition \ref{prop1} this implies our assumption on $\mu$.
\begin{prop}\label{prop2}
Let $\Phi \in \mathcal B_0$ be a complete Bernstein function with L\'evy pair $(0,\mu)$, $\alpha \in (0,2]$, $\ell$ a slowly
varying function at zero such that $\Phi(u)\sim u^{\alpha/2}\ell(u)$ as $u \downarrow 0$, and $j$ the jump kernel
associated with $\Phi$. Then
	\begin{equation*}
	j(r)\sim \frac{\alpha \Gamma\left(\frac{d+\alpha}{2}\right)}{2^{2-\alpha}
\pi^{\frac{d}{2}}\Gamma\left(1-\frac{\alpha}{2}\right)}r^{-d-\alpha}\well(r^2), \quad r \to \infty,
	\end{equation*}
	where $\well(t)=\ell(1/t)$.
\end{prop}
\begin{proof}
Choose $\delta=\frac{\alpha}{2}$ and $A=1$ in Potter's bound \eqref{potbound}. Since $\well(t)$ is slowly varying at
infinity, there exists a constant $M_1(\alpha)$ such that
	\begin{equation*}
	\frac{\well(t)}{\well(s)}\le
\max\left\{\Big(\frac{t}{s}\Big)^{\frac{\alpha}{2}},\left(\frac{s}{t}\right)^{\frac{\alpha}{2}}\right\},
\quad  t>M_1(\alpha).
	\end{equation*}
	Moreover, by Proposition \ref{prop1} there exists a constant $M_2>0$ such that
	\begin{equation*}
	\frac{\mu(t)}{\frac{\frac{\alpha}{2}}{\Gamma\left(1-\frac{\alpha}{2}\right)}t^{-1-\frac{\alpha}{2}}\well(t)}
\le 2, \quad t>M_2.
	\end{equation*}
Define $M(\alpha)=\max\{M_1(\alpha),M_2\}$, and denote it by $M$ for simplicity. By Lemma \ref{lem1} there
exists $t_0 \in (0,1)$ such that for every $t \in (0,t_0)$ we have $\mu(t)\le t^{-2}$, where $\mu$ is the density
of the L\'evy measure for $\Phi$. The change of variable $s=r^2/(4t)$ in \eqref{jumpker} and a division on
both sides give
	\begin{align*}
	\frac{j(r)}{r^{-d-\alpha}\well(r^2)}
&=\frac{1}{4\pi^{\frac{d}{2}}}\frac{r^{\alpha+2}}{\well(r^2)}
\left(\int_0^{\frac{r^2}{4M}} + \int_{\frac{r^2}{4M}}^{\frac{r^2}{4t_0}} + \int_{\frac{r^2}{4t_0}}^{\infty} \right)
s^{\frac{d}{2}-2}e^{-s}\mu\left(\frac{r^2}{4s}\right)ds\\&
	=\frac{1}{4\pi^{\frac{d}{2}}}(I_1(r)+I_2(r)+I_3(r)),
	\end{align*}
where we assumed $t_0<M$ with no loss of generality.

Consider first the integral $I_3(r)$. Observe that since $s>\frac{r^2}{4t_0}$, we have $\mu\left(\frac{r^2}{4s}\right)<
\frac{16s^2}{r^4}$. This gives
	\begin{equation*}\label{prop2pass1}
	0 \le I_3(r)\le \frac{16}{r^{2-\alpha}\well(r^2)}\int_{\frac{r^2}{4t_0}}^{\infty}s^{\frac{d}{2}}e^{-s}ds\le
\frac{16\Gamma\left(\frac{d}{2}+1\right)}{r^{2-\alpha}\well(r^2)}.
	\end{equation*}
	Since $2-\alpha \ge 0$ and $\well(r^2)$ is again a slowly varying function at infinity, we have $\lim_{r \to \infty}
r^{2-\alpha}\well(r^2)=\infty$ and thus $\lim_{r \to \infty}I_3(r)=0$. Next consider $I_2(r)$ and notice that
	\begin{equation*}\label{prop2pass2}
	0 \le I_2(r)\le \frac{r^{\alpha+2}}{\well(r^2)}\mu(t_0)\int_{\frac{r^2}{4M}}^{\frac{r^2}{4t_0}}s^{\frac{d}{2}-2}e^{-s}\le \frac{r^{\alpha+2}}{\well(r^2)}\mu(t_0)\Gamma\left(\frac{d}{2}-1,\frac{r^2}{4M}\right).
	\end{equation*}
	Using that $\lim_{x \to \infty}\frac{\Gamma(s,x)}{x^{s-1}e^{-x}}=1$, there exists a constant $C_1>0$ such that
	\begin{equation*}
	\Gamma\left(\frac{d}{2}-1,\frac{r^2}{4M}\right)\le \frac{C_1}{(4M)^{\frac{d}{2}-2}}r^{d-4}e^{-\frac{r^2}{4M}}.
	\end{equation*}
Since $\lim_{r \to \infty}r^{\alpha+d+\gamma-2}e^{-\frac{r^2}{4M}}=0$ for any $\gamma>0$, for large enough $r$ there exists
a constant $C_2>0$ such that $r^{\alpha+d-2}e^{-\frac{r^2}{4M}}\le C_2 r^{-\gamma}$. Combining this with the above estimates
we obtain
	\begin{equation*}
	0 \le I_2(r)\le \frac{C_1 \mu(t_0)}{(4M)^{\frac{d}{2}-2}}\frac{r^{\alpha+d-2}}{\well(r^2)}e^{-\frac{r^2}{4M}} \leq \frac{C_1 C_{2} \mu(t_0)}{(4M)^{\frac{d}{2}-2}}\frac{1}{r^\gamma\well(r^2)},
	\end{equation*}
which shows that $\lim_{r \to \infty}I_2(r)=0$. Finally, consider $I_1(r)$. Denoting
\begin{equation*}\label{prop2pass5}
F(r,s) =
	\frac{\mu\left(\frac{r^2}{4s}\right)}{\frac{\frac{\alpha}{2}}{\Gamma\left(1-\frac{\alpha}{2}\right)}
\left(\frac{r^2}{4s}\right)^{-1-\frac{\alpha}{2}} \well\left(\frac{r^2}{4s}\right)},
\end{equation*}
we get
\begin{equation*}\label{prop2pass4}
I_1(r)
= \frac{2^\alpha\alpha }{\Gamma\left(1-\frac{\alpha}{2}\right)}
\int_0^{\infty}s^{\frac{d+\alpha}{2}-1}e^{-s} F(r,s) \frac{\well\left(\frac{r^2}{4s}\right)}
{\well(r^2)}1_{\left(0,\frac{r^2}{4M}\right)}(s)ds.
\end{equation*}
Observe that for $s<\frac{r^2}{4M}$ we have $F(r,s) \leq 2$,
moreover, we may suppose $r^2>M$ and then by Potter's bound
	\begin{equation*}\label{prop2pass6}
	\frac{\well\left(\frac{r^2}{4s}\right)}{\well(r^2)}\le
\max\{(4s)^{\frac{\alpha}{2}},(4s)^{-\frac{\alpha}{2}}\}\le (4s)^{\frac{\alpha}{2}}+(4s)^{-\frac{\alpha}{2}}.
	\end{equation*}
A combination of the above gives
	\begin{equation*}
s^{\frac{d+\alpha}{2}-1}e^{-s} F(r,s) \frac{\well\left(\frac{r^2}{4s}\right)}
{\well(r^2)}1_{\left(0,\frac{r^2}{4M}\right)}(s)
\le  2^{\alpha+1}s^{\frac{d}{2}+\alpha-1}e^{-s}+2^{-\alpha+1}s^{\frac{d}{2}-1}e^{-s},
	\end{equation*}
and $\int_0^{\infty}(2^{\alpha+1}s^{\frac{d}{2}+\alpha-1}e^{-s}+2^{-\alpha+1}s^{\frac{d}{2}-1}e^{-s})ds
=2^{\alpha+1}\Gamma\left(\frac{d}{2}+\alpha\right)+2^{-\alpha+1}\Gamma\left(\frac{d}{2}\right)$.
Thus by dominated convergence
\begin{equation*}
\lim_{r \to \infty}I_1(r)=\frac{2^\alpha\alpha}{\Gamma\left(1-\frac{\alpha}{2}\right)}\int_0^{\infty}
 s^{\frac{d+\alpha}{2}-1}e^{-s}ds
=\frac{2^\alpha\alpha\Gamma\left(\frac{d+\alpha}{2}\right)}{\Gamma\left(1-\frac{\alpha}{2}\right)},
\end{equation*}
giving
\begin{equation*}
\lim_{r \to \infty}\frac{j(r)}{r^{-d-\alpha}\well(r^2)}=
\frac{\alpha\Gamma\left(\frac{d+\alpha}{2}\right)}{2^{2-\alpha}\pi^{\frac{d}{2}}\Gamma\left(1-\frac{\alpha}{2}\right)}.
\end{equation*}
\end{proof}
A useful direct consequence of this result is that we can study the asymptotic behaviour of the L\'evy measure $\nu$ outside
balls, some integrals related to the second moment of $\nu$ within balls, and the $L^p$ norm of $j(|x|)$ on $B_R^c(0)$ for
large radii. We will use these estimates below.
\begin{cor}\label{cor2}
Let $\Phi \in \mathcal B_0$ be a complete Bernstein function, $\alpha \in (0,2]$, and $\ell$ a slowly varying function
at zero such that $\Phi(u)\sim u^{\frac{\alpha}{2}}\ell(u)$ as $u \downarrow 0$. Also, let $\nu(dx)=j(|x|)dx$ and $p>1$.
Then the following properties hold:
	\begin{enumerate}
		\item
		\begin{equation*}
		\nu(B_R^c(0))\sim
\frac{d\omega_d\Gamma\left(\frac{d+\alpha}{2}\right)}{2^{2-\alpha}\pi^{\frac{d}{2}}
\Gamma\left(1-\frac{\alpha}{2}\right)}R^{-\alpha}\well(R^2),
\quad R \to \infty,
		\end{equation*}
		where $\well(t)=\ell(1/t)$.
		\item
		\begin{equation}
\label{defJ}
		\mathcal{J}(R):=\int_{B_R(0)}|x|^2j(|x|)dx\sim \frac{d\omega_d\alpha \Gamma\left(\frac{d+\alpha}{2}\right)}{(2-\alpha)2^{2-\alpha}\pi^{\frac{d}{2}}\Gamma\left(1-\frac{\alpha}{2}\right)}
		R^{2-\alpha}\well(R^2), \quad R \to \infty.
		\end{equation}
		\item
		\texttt{\begin{equation*}
\Norm{j}{L^p(B_R^c(0))} \sim \frac{(d\omega_d)^{\frac{1}{p}}\alpha \Gamma\left(\frac{d+\alpha}{2}\right)}
{2^{2-\alpha}\pi^{\frac{d}{2}}\Gamma\left(1-\frac{\alpha}{2}\right)}
		\frac{R^{-\frac{d}{q}-\alpha}\well(R^2)}{\left((p-1)d+p\alpha\right)^{\frac{1}{p}}}, \quad R \to \infty,
		\end{equation*}}
		where $q>1$ is the H\"older-conjugate exponent of $p$, i.e., $\frac{1}{q}+\frac{1}{p}=1$.
	\end{enumerate}
\end{cor}
\begin{proof}
	We prove (1), the other parts can be shown completely similarly. First, we write
	\begin{equation}\label{cor2pass-1}
	\nu(B_R^c(0))=d\omega_d \int_{R}^{\infty}j(r)r^{d-1}dr,
	\end{equation}
and denote for a shorthand
$$
A = \frac{\alpha \Gamma\left(\frac{d+\alpha}{2}\right)}{2^{2-\alpha}
\pi^{\frac{d}{2}}\Gamma\left(1-\frac{\alpha}{2}\right)}.
$$
By Proposition \ref{prop2} we have that for every $\varepsilon \in \left(0,\frac{1}{2}\right)$ there exists
$R_1(\varepsilon)$ such that for $r>R_1(\varepsilon)$
	\begin{equation}\label{cor2pass0}
	1-\varepsilon \le \frac{j(r)}{A r^{-d-\alpha}\well(r^2)}\le 1+\varepsilon
	\end{equation}
holds.
	Moreover, by Karamata's Tauberian theorem \cite{BGT}
	\begin{equation*}
	\int_{R}^{\infty}r^{-1-\alpha}\well(r^2)dr\sim \frac{R^{-\alpha}}{\alpha}\well(R^2),
	\end{equation*}
	hence for every $\varepsilon \in \left(0,\frac{1}{2}\right)$ there exists $R_2(\varepsilon)$ such that for all
    $R>R_2(\varepsilon)$ we have
	\begin{equation}\label{cor2pass01}
	1-\varepsilon \le \frac{\int_{R}^{\infty}r^{-1-\alpha}\well(r^2)dr}{\frac{R^{-\alpha}}{\alpha}\well(R^2)}\le 1+\varepsilon.
	\end{equation}
	Take $\varepsilon \in \left(0,\frac{1}{2}\right)$ and $R>\max\{R_1(\varepsilon),R_2(\varepsilon)\}$. By \eqref{cor2pass-1},
	\begin{equation*}
	\nu(B_R^c(0)) = d\omega_d A \int_{R}^{\infty}\frac{j(r)}{Ar^{-d-\alpha}\well(r^2)}r^{-1-\alpha}\well(r^2)dr,
	\end{equation*}
	and so
	\begin{equation*}
	\frac{\nu(B_R^c(0))}{d\omega_d A R^{-\alpha}\well(R^2)} =\frac{\displaystyle{\int_{R}^{\infty}
    \frac{j(r)}{Ar^{-d-\alpha}\well(r^2)}r^{-1-\alpha}\well(r^2)dr}}
    {\frac{R^{-\alpha}}{\alpha}\well(R^2)}.
	\end{equation*}
	Choosing $r>R>R_1(\varepsilon)$, we get by \eqref{cor2pass0}
	\begin{equation*}
	(1-\varepsilon)\frac{\int_{R}^{\infty}r^{-\alpha-1}\well(r^2)}{\frac{R^{-\alpha}}{\alpha}\well(R^2)}\le
    \frac{\nu(B_R^c(0))}{d\omega_d AR^{-\alpha}\well(R^2)}\le (1+\varepsilon)\frac{\int_{R}^{\infty}r^{-\alpha-1}\well(r^2)}
    {\frac{R^{-\alpha}}{\alpha}\well(R^2)}
	\end{equation*}
	and for $R>R_2(\varepsilon)$ by \eqref{cor2pass01}
	\begin{equation*}
	(1-\varepsilon)^2\le\frac{\nu(B_R^c(0))}{d\omega_d A R^{-\alpha}\well(R^2)}\le (1+\varepsilon)^2,
	\end{equation*}
    for arbitrary $\varepsilon > 0$, giving the statement in (1).
\end{proof}
{\begin{rmk}
{\rm
We can actually show that for every $\gamma \in (\alpha,2]$ and $0 < R_0 < R$
	\begin{equation*}
	\int_{B_R(0)\setminus B_{R_0}(0)}|x|^\gamma j(|x|)dx\sim \frac{d\omega_d\alpha
\Gamma\left(\frac{d+\alpha}{2}\right)}{(\gamma-\alpha)2^{2-\alpha}
\pi^{\frac{d}{2}}\Gamma\left(1-\frac{\alpha}{2}\right)}
	R^{\gamma-\alpha}\well(R^2), \quad R \to \infty.
	\end{equation*}
The choice $\gamma=2$ is suggested by the fact that since $j(|x|)$ is the density of a L\'evy measure, $|x|^2j(|x|)$
is in $L^1_{\rm loc}(\R^d)$. In general, we may consider any function $\beta(|x|)$ such that $\beta(|x|)j(|x|)$
belongs to $L^1(B_{R_0}(0))$ for some $R_0>0$ and $\beta(r)\sim r^\gamma$ as $r \to \infty$. In this case we can show
that
	\begin{equation*}
	\int_{B_R(0)}\beta(|x|)j(|x|)dx\sim \frac{d\omega_d\alpha \Gamma\left(\frac{d+\alpha}{2}\right)}
{(\gamma-\alpha)2^{2-\alpha}\pi^{\frac{d}{2}}\Gamma\left(1-\frac{\alpha}{2}\right)}
	R^{\gamma-\alpha}\well(R^2), \quad R \to \infty.
	\end{equation*}
In what follows, we will use for simplicity $\beta(|x|)=|x|^2$. In cases when other $\beta$ will be needed, we will use
the notation
\begin{equation}
\label{Jbeta}
\cJ_\beta(R)=\int_{B_R(0)}\beta(|h|)j(|h|)dh
\end{equation}
for $R>0$, whenever $\beta(|h|)$ belongs to $L^1_{\rm loc}(\R^d,\nu)$.
}
\end{rmk}

\subsubsection{\textbf{Exponentially light L\'evy intensities}}
Next we turn to the asymptotic behaviour of jump kernels for L\'evy intensities with exponentially short tails. Recall
the modified Bessel function of the third kind given by
\begin{equation}
\label{bessel3}
K_\rho (z) = \frac{1}{2} \left(\frac{z}{2}\right)^\rho \int_0^\infty t^{-\rho - 1} e^{-t-\frac{z^2}{4t}} dt, \quad z > 0,
\; \rho > -\frac{1}{2}.
\end{equation}
We will make repeated use of the asymptotic formula \cite{GR}
\begin{equation}
\label{bessel3asymp}
K_\rho (z) \sim \sqrt{\frac{\pi}{2|z|}}\, e^{-|z|}, \quad |z|\to\infty.
\end{equation}

\begin{prop}
\label{light}
	Let $\Phi \in \mathcal B_0$ be a complete Bernstein function with L\'evy pair $(0,\mu)$, and suppose that there exist
$\alpha \in (0,2]$, $\eta,\theta>0$ 	such that $\mu(t)\sim \theta t^{-1-\frac{\alpha}{2}} e^{-\eta t}$ as $t \to \infty$.
	Then
	\begin{equation*}
	j(r)\sim \theta \pi^{\frac{1-d}{2}}2^{\frac{\alpha-1-d}{2}}\eta^{\frac{d+\alpha+2}{4}}r^{-\frac{d+\alpha+1}{2}}e^{-\sqrt{\eta}r},
\quad r \to \infty.
	\end{equation*}	
	\end{prop}
\begin{proof}
	With no loss of generality we may set $\theta=1$. Since $\mu(t)\sim  t^{-1-\frac{\alpha}{2}}e^{-\eta t}$ as $t \to \infty$,
for fixed $\varepsilon \in (0,1)$ there exists a constant $t_0 = t_0(\varepsilon)$ such that
	\begin{equation*}
	(1-\varepsilon) t^{-1-\frac{\alpha}{2}}e^{-\eta t}\le \mu(t)\le (1+\varepsilon) t^{-1-\frac{\alpha}{2}}e^{-\eta t}
	\end{equation*}
	for every $t>t_0$. We write \eqref{jumpker} as
	\begin{align*}
	j(r)&
	=\frac{1}{(4\pi)^{\frac{d}{2}}}\left(\int_0^{t_0} +\int_{t_0}^{\infty}\right) t^{-\frac{d}{2}}e^{-\frac{r^2}{4t}}\mu(t)dt
	=\frac{1}{(4\pi)^{\frac{d}{2}}}(I_1(r)+I_2(r)).
	\end{align*}
	Clearly,
	$I_2(r)\le (4\pi)^{\frac{d}{2}}j(r)\le I_1(r)+I_2(r)$.
	For the second integral we have
	\begin{align*}
	I_2(r)&\le(1+\varepsilon)\int_{t_0}^{\infty} t^{-1-\frac{d+\alpha}{2}}e^{-\eta t-\frac{r^2}{4t}}dt\\
	&\le
(1+\varepsilon)r^{-\frac{d+\alpha}{2}}2^{\frac{d+\alpha}{2}}\eta^{\frac{d+\alpha}{4}}\left(\frac{r}{2\sqrt{\eta}}\right)^{\frac{d+\alpha}{2}}
\int_{0}^{\infty} t^{-1-\frac{d+\alpha}{2}}e^{-\frac{2\eta}{2}\left(t-\frac{r^2}{4\eta t}\right)}dt\\
	&=(1+\varepsilon)r^{-\frac{d+\alpha}{2}}2^{\frac{d+\alpha}{2}}\eta^{\frac{d+\alpha}{4}}K_{\frac{d+\alpha}{2}}\left(\sqrt{\eta}r\right).
	\end{align*}
	Using \eqref{bessel3asymp}, we have
	\begin{equation*}
r^{-\frac{d+\alpha}{2}}K_{\frac{d+\alpha}{2}}\left(\sqrt{\eta}r\right)\sim
\sqrt{\frac{\pi\eta}{2}}r^{-\frac{d+\alpha+1}{2}}e^{-\sqrt{\eta}r}.
	\end{equation*}
Consider next the integral $I_1(r)$. We have
	\begin{equation*}
	\frac{I_1(r)}{r^{-\frac{d+\alpha+1}{2}}e^{-\sqrt{\eta}r}}=
\int_0^{t_0(\varepsilon)}t^{-\frac{d}{2}}r^{\frac{d+\alpha+1}{2}}e^{-\frac{r^2}{4t}+\sqrt{\eta}r}\mu(t)dt.
	\end{equation*}
	Define the function
	\begin{equation*}
	f(r)=r^{\frac{d+\alpha+1}{2}}e^{-\frac{r^2}{8t}+\sqrt{\eta}r}
	\end{equation*}
	and notice that it attains its maximum at
	\begin{equation*}
	r_{\rm max}(t)=2\sqrt{\eta}t+\sqrt{4\eta t^2+2(d+\alpha+1)t}.
	\end{equation*}
	In particular, we have
	\begin{align*}
	f(r_{\rm max}(t))= g(t) e^{-\frac{\sqrt{\eta}}{2}\left(2\sqrt{\eta}t+\sqrt{4\eta t^2+2(d+\alpha+1)t}\right)
                        -\frac{\sqrt{\eta}}{4}(d+\alpha+1)}
	\end{align*}
	with
	\begin{equation*}
	g(t)=\left(2\sqrt{\eta}t+\sqrt{4\eta t^2+2(d+\alpha+1)t}\right)^{\frac{d+\alpha+1}{2}}
e^{4\eta t+\sqrt{4\eta^2 t^2+2\eta(d+\alpha+1)t}}.
	\end{equation*}
	Note that $g(t)$ is an increasing function and we have $f(r_{\rm max}(t))\le g(t_0(\varepsilon))$ for every
$t \in (0,t_0(\varepsilon))$,
	giving
	\begin{equation}
    \label{I111}
	\frac{I_1(r)}{r^{-\frac{d+\alpha+1}{2}}e^{-\sqrt{\eta}r}}\le  g(t_0(\varepsilon))
    \int_0^{t_0(\varepsilon)}t^{-\frac{d}{2}}e^{-\frac{r^2}{8t}}\mu(t)dt.
	\end{equation}
	Since we take $r \to \infty$, we may choose that $r>1$ to get $e^{-\frac{r^2}{8t}}\le
    e^{-\frac{1}{8t}}\le C(d) t^\frac{d+4}{2}$ with a constant $C(d)>0$. Hence, since $\mu(t)$ is the density
    of a L\'evy measure, we can
    use dominated convergence and conclude that the left hand side in \eqref{I111} goes to zero in the limit.
	On division by $\sqrt{\pi}2^{\frac{d+\alpha-1}{2}}\eta^{\frac{d+\alpha+2}{4}}r^{-\frac{d+\alpha+1}{2}}e^{-\sqrt{\eta}r}$
    giving
	\begin{equation*}
	\frac{j(r)}{\pi^{\frac{1-d}{2}}2^{\frac{\alpha-1-d}{2}}\eta^{\frac{d+\alpha+2}{4}}r^{-\frac{d+\alpha+1}{2}}e^{-\sqrt{\eta}r}}\le \frac{I_1(r)}{\pi^{\frac{1-d}{2}}2^{\frac{\alpha+1-d}{2}}\eta^{\frac{d+\alpha+4}{4}}r^{-\frac{d+\alpha+3}{2}}
e^{-\sqrt{\eta}r}}+(1+\varepsilon)\frac{K_{\frac{d+\alpha}{2}+1}(\sqrt{\eta}r)}{\sqrt{\frac{\pi \eta}{2}}r^{-\frac{1}{2}}e^{-\sqrt{\eta}r}},
	\end{equation*}
	and using \eqref{bessel3asymp}, we obtain
	\begin{equation}
     \label{bzz1}
	\limsup_{r \to \infty}\frac{j(r)}{\pi^{\frac{1-d}{2}}2^{\frac{\alpha+1-d}{2}}
\eta^{\frac{d+\alpha+4}{4}}r^{-\frac{d+\alpha+3}{2}}e^{-\sqrt{\eta}r}}\le 1+\varepsilon.
	\end{equation}
	To get the lower bound, observe that
	\begin{align*}
	I_2(r)&\ge (1-\varepsilon)\int_{t_0(\varepsilon)}^{+\infty}t^{-1-\frac{d+\alpha}{2}}e^{-\eta t-\frac{r^2}{4t}}dt
	=(1-\varepsilon)\left(\int_{0}^{\infty}-\int_{0}^{t_0(\varepsilon)}\right)t^{-1-\frac{d+\alpha}{2}}e^{-\eta t-\frac{r^2}{4t}}dt\\
	&=(1-\varepsilon)I_3(r)-(1-\varepsilon)I_4(r).
	\end{align*}
	We have, as before,
	\begin{equation*}
	I_3(r)=r^{-\frac{d+\alpha}{2}}2^{\frac{d+\alpha}{2}}\eta^{\frac{d+\alpha}{4}}K_{\frac{d+\alpha}{2}}\left(\sqrt{\eta}r\right).
	\end{equation*}
	Furthermore,
	\begin{equation*}
	\frac{I_4(r)}{r^{-\frac{d+\alpha+1}{2}}e^{-\sqrt{\eta}r}}\le g(t_0(\varepsilon))
\int_0^{t_0(\varepsilon)}t^{-1-\frac{d+\alpha}{2}}e^{-\frac{r^2}{8t}}dt,
	\end{equation*}
	and then again dominated convergence implies that the left hand side goes to zero as $r \to \infty$.
    Proceeding similarly as with $I_1$, we see that	
    \begin{align}\label{bzz2}
	\liminf_{r \to \infty}\frac{j(r)}{\pi^{\frac{1-d}{2}}2^{\frac{\alpha-1-d}{2}}\eta^{\frac{d+\alpha+2}{4}}
    r^{-\frac{d+\alpha+1}{2}}e^{-\sqrt{\eta}r}}&\ge 1-\varepsilon.
	\end{align}
    A combination of \eqref{bzz1}-\eqref{bzz2} completes the proof. 	
\end{proof}

As before, we will need some information on the asymptotic behaviour of the L\'evy measure outside of
balls and on the second moment of our L\'evy measures, which we discuss next.
\begin{cor}\label{corexp}
	Let $\Phi \in \mathcal B_0$ be a complete Bernstein function, and suppose there exist $\alpha \in (0,2]$, $\eta,\theta>0$
	such that
	$\mu(t)\sim \theta t^{-1-\frac{\alpha}{2}} e^{-\eta t}$ as $t \to \infty$.
	The following hold.
	\begin{enumerate}
	\item As $R \to \infty$,
	\begin{equation*}
	\nu(B_R^c(0))\sim \frac{d\omega_d\eta^\frac{3d-\alpha-4}{4}}{\pi^\frac{d-1}{2}2^\frac{d+1-\alpha}{2}}R^\frac{d-\alpha-4}{2}e^{-\sqrt{\eta}R}.
	\end{equation*}
\vspace{0.1cm}
	\item
Denote
$
C = d\omega_d\eta^\frac{2\alpha-1}{4}\pi^{-\frac{d-1}{2}}2^{-\frac{d+1-\alpha}{2}}
\Gamma\left(\frac{d+3-\alpha}{2},\sqrt{\eta}M(\varepsilon)\right).
$
For every $\varepsilon \in (0,1)$ there exists a constant $M(\varepsilon)$ such that
	\begin{equation*}
	\int_{B_{M(\varepsilon)}(0)}|h|^2j(|h|)dh+(1-\varepsilon)C
\le \int_{\R^d}|h|^2j(|h|)dh\le\int_{B_{M(\varepsilon)}(0)}|h|^2j(|h|)dh+(1+\varepsilon)C.
	\end{equation*}
	In particular, both
$\int_{B_{M(\varepsilon)}(0)}|h|^2j(|h|)dh$ and $\int_{\R^d}|h|^2j(|h|)dh$ are finite.
	\end{enumerate}
	\end{cor}
\begin{proof}
The proof proceeds similarly to the proof of Proposition \ref{light} and is left to the reader.
\end{proof}

\begin{rmk}
{\rm Part $(2)$ of Corollary \ref{corexp} holds also for $\int_{R^d}\beta(|h|)j(|h|)dh$, whenever $\beta j \in
L^1_{\rm loc}(\R^d)$ and $\beta(r) \sim r^{\gamma}$ as $r \to \infty$ for every $\gamma>0$, with a different
constant $C = C(\beta, \gamma)$ and writing $\cJ_\beta(M(\varepsilon))$ in the integrals appearing in the upper
and lower bounds.}
\end{rmk}

\subsection{The operator $\Phi(-\Delta)$ and function spaces}
\subsubsection{\textbf{General expression}}
We define the operator $\Phi(-\Delta)$ by Fourier transform. For $\Phi \in \mathcal B_0$, we write
\begin{equation}
\label{phiop}
\widehat{(\Phi(-\Delta) f)}(y) = \Phi(|y|^2)\widehat f(y), \quad y \in \R^d, \; f \in  \Dom(\Phi(-\Delta)),
\end{equation}
with domain
\begin{equation}
\label{domphi}
\Dom(\Phi(-\Delta))=\Big\{f \in L^2(\Rd): \Phi(|\cdot|^2) \widehat f \in L^2(\R^d) \Big\}.
\end{equation}

In the following we will be interested in a more convenient representation of this operator, which we discuss
next. Here and below we will frequently use the notation
\begin{equation}
\label{defDv}
\DDv = f(x+h)-2f(x)+f(x-h)
\end{equation}
for centered second order differences of functions $f$ in given spaces.

Given a measure $\nu$ on $\R^d$, consider a modulus of continuity $\beta: \R^+ \to \R^+$ and the space
\begin{equation*}
\label{Lrad}
L^1_{\rm rad}(\R^d,\nu) = \big\{\beta: \R^+ \to \R^+, \, \beta \in L^1_{\rm loc}(\R^d,\nu) \big\}.
\end{equation*}
Let $\Omega \subset \R^d$ be an open set, and define the space of functions
\begin{equation*}
\label{Zspace}
\mathcal{Z}^\beta(\Omega)=\left\{f \in L^\infty(\R^d): \exists \, L_{f,\beta}, R_{f,\beta} :\Omega \to \R^+_*,
\  |\DDv|\le L_{f,\beta}(x) \beta(|h|), \ h \in B_{R_{f,\beta}(x)}(0), \ x \in \Omega\right\},
\end{equation*}
where $\R^+_*=(0,\infty)$. We also use the notations
$$
\mathcal{Z}^\beta(x) \;\; \mbox{for $\Omega = \{x\}$} \qquad \mbox{and} \qquad
\mathcal{Z}_{\rm b}^\beta(\Omega) \;\; \mbox{for $R_{f,\beta}(x) = R_{f,\beta}>0$, $x\in\Omega$, and
$L_{f,\beta}\in L^\infty(\Omega)$}.
$$
For our purposes below, often it will be sufficient to choose $\beta(r) = r^2$, for which we write simply
$\mathcal{Z}(\Omega)$, and similarly $\mathcal{Z}(x)$, $\mathcal{Z}_{\rm b}(\Omega)$. Note that for this choice
of $\beta$ there exist functions $f \in \mathcal{Z}(x)$ which are not continuous at $x$. For instance, taking
$d=1$, such a function is $f(x) = 1_{\{x < 0\}} + 2 \cdot 1_{\{x=0\}} + 3 \cdot 1_{\{x > 0\}}$, while
$|D_h f (0)|=0$ for every $h$, and thus $f \in \mathcal{Z}(0)$.

{\begin{rmk}\label{rmkallR}
{\rm
It is readily seen that if $f \in \mathcal{Z}(x)$, then for every $R>0$ there exists a constant $L_{R,f}$ such
that $|\DDv|\le L_{R,f}|h|^2$ for every $h \in B_{R}(0)$. Indeed, if $R\le R_f$ we can take $L_{R,f}=L_f$. If
$R>R_f$, then we define $L'_{R,f}=\frac{4\Norm{f}{\infty}}{R_f^2}$ and set $L_{R,f}=\max\{L'_{R,f},L_f\}$.
}
\end{rmk}
By Remark \ref{rmkallR} above it follows that if $f \in \cZ(\Omega)$, then there exists a function $L_f(x,R)$
such that $|\DDv|\le L_f(x,R)|h|^2$. Below we will often need a control of the form $L_f(x,R(x))$ where $R:\R^d
\to \R^+$ is itself a function of $x$, typically for large $|x|$. We define for a fixed constant $C \in (0,1)$
\begin{align}
\label{ZCspace}
\begin{split}
\cZ^\beta_{C}(\R^d) &= \left \{f \in L^\infty(\R^d): \ \exists \, M_{f,\beta}>0, \, L_{f,\beta} \in
L^\infty(B_{M_{f,\beta}}^c(0)), \ |\DDv|\le L_{f,\beta}(x)\beta(|h|),  \right. \\ & \left. \hspace{7cm} \ x \in
B_{M_{f,\beta}}^c(0), \ h \in  B_{C|x|}(0)\right\}.
\end{split}
\end{align}
For $\beta(r)=r^2$ we simply write $\cZ_{C}(\R^d)$, with $L_{f,\beta}=L_{f}$, $ M_{f,\beta}=M_{f}$. Clearly, if
$f \in \cZ_{C}(\R^d)$, then $f \in \cZ_{\rm b}(B^c_{M_f}(0))$.

Define a semi-norm on $\mathcal{Z}_{\rm b}(\R^d)$ given by
\begin{equation*}
\Norm{f}{\dot{\mathcal{Z}}_{\rm b}}:=\sup_{h \in \R^d\setminus\{0\}}\frac{|\DDv|}{|h|^2}
\end{equation*}
and a norm
\begin{equation*}
\Norm{f}{\mathcal{Z}_{\rm b}}:=\Norm{f}{\infty}+\Norm{f}{\dot{\mathcal{Z}}_{\rm b}}
\end{equation*}
under which $\mathcal{Z}_{\rm b}(\R^d)$ is a Banach space. We also define a local version by requiring that
$f \in \mathcal{Z}_{{\rm b, loc}}(\Omega)$ if and only if $f \in \mathcal{Z}_{\rm b}(K)$ for every compact set
$K \subset \Omega$. Clearly, $C^2(\Omega)\subset \mathcal{Z}(\Omega)$, and for every $\Omega'$ such that
$\overline{\Omega}'\subset \Omega$, we have $C^2(\Omega)\subset \mathcal{Z}_{\rm b}(\Omega')$.

The conditions defining the above spaces are reminiscent of Zygmund-H\"older spaces. Recall the Zygmund-H\"older
space $\mathcal C^\gamma(\R^d)$ of order $\gamma = k+s$, with $k \in \mathbb N_0 = \{0\} \cup \mathbb N$ and $s
\in (0,1]$, given by the set of real-valued functions $f \in C^k(\R^d)$ such that the norm
$$
\|f\|_{\mathcal C^\gamma} = \|f\|_{C^k} + \max_{\xi \in \mathbb N_0^d \atop |\xi|=k}
\sup_{x,h \in \R^d \atop h\neq 0} \frac{\partial^\xi f(x+h) - 2\partial^\xi f(x) + \partial^\xi f(x-h)}{|h|^s}
$$
is finite, where $\|f\|_{C^k}=\sum_{i=0}^k \sum_{\xi \in \mathbb N_0^d, |\xi|=i}\|\partial^\xi f\|_{\infty}$
(see, e.g., \cite{T}.) Then for $\beta(|h|)=|h|^{s}$ with $s \in (0,2)$, we have that $\mathcal{Z}_{\rm b}^\beta
(\Omega)$ coincides with the Zygmund space $\mathcal C^{s+\lfloor s \rfloor}(\Omega)$ (alternatively, the Besov
space $B^{s}_{\infty,\infty}(\Omega)$). In our space $\mathcal{Z}_{\rm b}(\R^d)$ we have a quadratic scale $s=2$
instead of the linear or sublinear scales appearing in $\mathcal C^\gamma(\R^d)$ with $\gamma=s$.

From the following we see that the space $\mathcal{Z}_{\rm b}(\R^d)$ contains functions that are quite regular.
\begin{prop}
\label{Zbspace}
If $f \in \mathcal{Z}_{\rm b}(\R^d)$, then $f \in C(\R^d)\cap W^{1,\infty}(\R^d)$.
\end{prop}
\begin{proof}
Define $D_h^1 f(x)=f(x+h)-f(x)$,
\begin{align*}
R_h^1 f(x)=\frac{D_h^1 f(x)}{|h|}, \quad
\omega(f, t)=\sup_{|h|\le t}|D_h^1 f(x)| \quad \mbox{and} \quad  \eta(f, t)=\sup_{|h|\le t}|\DDv|.
\end{align*}
	It has been shown by Marchaud \cite[Sect. 22]{M27} that there exists a constant $M$ such that
	\begin{equation*}
	\omega(f,t)\le Mt\int_t^{\infty}\eta(f,u)\frac{du}{u^2}.
	\end{equation*}
	Take $t<R_1$ and observe that since $f \in L^\infty(\R^d)$, we have $\eta(f,u)\le 4\Norm{f}{\infty}$.
Hence
	\begin{equation*}
	\omega(f,t)\le Mt\left(\int_t^{R_1}\eta(f,u)\frac{du}{u^2}+\frac{8}{R_1^3}\Norm{f}{\infty}\right).
	\end{equation*}
By the assumption it follows that $\eta(f,u)\le Lu^2$ for $u<R_f$, with $L=\Norm{f}{\dot{\mathcal{Z}}_b(\R^d)}$.
Hence we have
	\begin{equation*}
	\omega(f,t)\le Mt\left(L(R_f-t)+\frac{8}{R_f^3}\Norm{f}{\infty}\right).
	\end{equation*}
This gives $\lim_{t \downarrow 0}\omega(f,t) = 0$ and thus $f \in C(\R^d)$.
	Consider $h \in B_{R_1}(0)$. We have $|D_h^1f(x)|\le \omega(f,|h|)$ and
	\begin{equation*}
	|R^1_hf(x)|\le \frac{\omega(f,|h|)}{|h|}\le M\left(L(R_f-|h|)+\frac{8}{R_f^3}\Norm{f}{\infty}\right)
\le M\left(LR_f+\frac{8}{R_f^3}\Norm{f}{\infty}\right),
	\end{equation*}
	giving
	\begin{equation*}
	|D^1_hf(x)|\le M\left(LR_f+\frac{8}{R_f^3}\Norm{f}{\infty}\right)|h|, \quad x \in \R^d.
	\end{equation*}
Take now any $x,y \in \R^d$ and consider the segment denoted $\vv {[x,y]}$. Choose $h=\frac{y-x}{m}$
in such a way that $|h|<R_f$, i.e., $y=x+mh$. Define $x_i=x+ih$ for $i=0,\dots,m$. Thus for $i=1,\dots,m$
we have
	\begin{align*}
	|f(x_i)-f(x_{i-1})|&=|f(x_{i-1}+h)-f(x_{i-1})|=|D^1_{h}f(x_{i-1})|\\
&\le M\left(LR_f+\frac{8}{R_f^3}\Norm{f}{\infty}\right)|h|
	=M\left(LR_f+\frac{8}{R_f^3}\Norm{f}{\infty}\right)\frac{|x-y|}{m}.
	\end{align*}
	This implies
	\begin{equation*}
	|f(x)-f(y)|\le
\sum_{i=1}^{m}|f(x_i)-f(x_{i-1})|\le M\left(LR_f+\frac{8}{R_f^3}\Norm{f}{\infty}\right)|x-y|.
	\end{equation*}
\end{proof}

Note that by using the integrability condition \eqref{Mset} of the L\'evy measure $\mu$ and that $1-e^{-ux} \leq ux$
for $u,x\geq 0$, for $\Phi \in \mathcal B_0$ we have
\begin{align*}
\Phi(u)
&=
bu + \left(\int_0^1 + \int_1^\infty\right)(1-e^{-uy}) \mu(y)dy
\leq
bu + u \int_0^1 y\mu(y)dy + \int_1^\infty \mu(y)dy
\leq
C_\Phi(1+u).
\end{align*}
Thus $1+\Phi(y^2) \leq \widetilde C_\Phi(1+|y|^2)$, $y\in \R^d$, with suitable constants $C_\Phi, \widetilde C_\Phi > 0$,
implying by Proposition \ref{Zbspace}, Sobolev embedding, and \eqref{domphi} that
$$
\Dom (\Phi(-\Delta)) \supset H^2(\R^d) \supset W^{1,\infty}(\R^d) \supset \mathcal{Z}_{\rm b}(\R^d).
$$
Also, $C^\infty_{\rm c}(\R^d)\subset \mathcal{Z}_{\rm b}(\R^d)$ and thus $\mathcal{Z}_{\rm b}(\R^d)$ is a dense subspace
of $\Dom (\Phi(-\Delta))$.

We furthermore note that the Zygmund-H\"older spaces $\mathcal C^\gamma$ turn out to be of much interest in the study of
the domains of non-local operators. From \cite[Th. 3.2]{KS19} we deduce that for the operator $\Phi(-\Delta)$ such that
$\Phi \in \mathcal B_0$ with $b=0$, is a complete Bernstein function that is regularly varying at zero with exponent
$\alpha/2 \in (0,1]$, the following inclusions hold. For $0 < \alpha < 1$ we have $\Dom(\Phi(-\Delta)) \supset \mathcal
C^s(\R^d) \cap C_0(\R^d)$ for all $\alpha < s \leq 1$, however, for $1 \leq \alpha <2$ we have $\Dom(\Phi(-\Delta))
\supset \mathcal C^s(\R^d) \cap C^1_0(\R^d)$ for all $\alpha < s \leq 2$, where $C_0(\R^d)$ denotes the space of real-valued
continuous functions on $\R^d$ vanishing at infinity, and $C^1_0(\R^d)$ denotes the space of real-valued $C^1$-functions on
$\R^d$ such that both the function and its derivative are $C_0(\R^d)$.

Now we are in the position to give a representation of the action of the operator on functions in the above function spaces,
which will be convenient for our purposes below.
\begin{prop}\label{prop3}
Let $f\in \mathcal{Z}(x)$ for some $x \in \R^d$. Then
	\begin{align*}
	-\Phi(-\Delta)f(x)&=\frac{1}{2}\int_{\R^d}\DDv j(|h|)dh =
	\lim_{\varepsilon \downarrow 0}\int_{B^c_\varepsilon(0)}(f(x+h)-f(x))j(|h|)dh.
	\end{align*}
\end{prop}
\begin{proof}
Using \eqref{bernrep} and \eqref{pt}, see also the subordination formula \eqref{subord}, we have
	\begin{align*}\label{prop3pass1}
	\begin{split}
	-\Phi(-\Delta)f(x)&=\int_{0}^{\infty}(e^{t\Delta}-1)f(x)\mu(dt)\\
	&=\int_{0}^{\infty}\int_{\R^d}(f(y)-f(x))p_t(|x-y|)dy\mu(dt).
	\end{split}
	\end{align*}
	Making the change of variable $y-x=h$, splitting the integral into two halves of the same integrand,
    and replacing $h$ by $-h$ in one half gives
	\begin{equation}\label{prop3pass2}
	\int_{\R^d}(f(y)-f(x))p_t(|x-y|)dy=\frac{1}{2}\int_{\R^d}\DDv p_t(|h|)dh,
	\end{equation}
	and hence
	\begin{equation*}
	-\Phi(-\Delta)f(x)=\frac{1}{2}\int_{0}^{\infty}\int_{\R^d}\DDv p_t(|h|)dh\mu(dt).
	\end{equation*}
	We have
	\begin{equation*}
	\int_{\R^d}|\DDv |\int_0^{\infty}p_t(|h|)\mu(dt)dh=\left(\int_{B^c_{R_1}(0)} +\int_{B_{R_1}(0)}\right)|\DDv |j(|h|)dh,
	\end{equation*}
	where $R_1>0$ is given by the definition of $\mathcal{Z}(x)$.
The first integral is estimated as	
\begin{equation*}
	\int_{B^c_{R_1}(0)} |\DDv |j(|h|)dh \le 4\Norm{f}{\infty}\nu(B_{R_1}^c(0))<\infty
	\end{equation*}
	using $\nu(dx)=j(|x|)dx$. For the second integral we get
	\begin{equation*}
	\int_{B_{R_1}(0)}|\DDv|j(|h|)dh \le L\int_{B_{R_1}(0)}|h|^2j(|h|)dh<\infty,
	\end{equation*}
	using that $j$ is the density of a L\'evy measure.
	Hence Fubini's theorem gives
	\begin{equation*}
	-\Phi(-\Delta) f(x) =\frac{1}{2}\int_{\R^d}\DDv j(|h|)dh,
	\end{equation*}
    and the first equality in the expression
	\begin{equation*}
	-\Phi(-\Delta) f(x) =\frac{1}{2}\lim_{\varepsilon \downarrow 0}\int_{B^c_\varepsilon(0)}\DDv j(|h|)dh
    = \frac{1}{2}\lim_{\varepsilon \downarrow 0}\int_{B^c_\varepsilon(0)}(f(y)-f(x))j(|x-y|)dy
	\end{equation*}
    follows by dominated convergence, and the second follows by using the same transformations that led to
    \eqref{prop3pass2}.
\end{proof}
\begin{cor}\label{cor3}
	Let $f \in L^\infty(\R^d)$ and fix any open set $\Omega \subseteq \R^d$. If $f \in \mathcal{Z}(\Omega)$,
	then
	\begin{align*}
	-\Phi(-\Delta)f(x)&=\frac{1}{2}\int_{\R^d}\DDv j(|h|)dh
	=\lim_{\varepsilon \downarrow 0}\int_{B^c_\varepsilon(0)}(f(x+h)-f(x))j(|h|)dh, \quad x \in \Omega.
	\end{align*}
\end{cor}
\begin{proof}
The result is immediate by a combination of Propositions \ref{Zbspace}-\ref{prop3}.
\end{proof}

\begin{rmk}
\label{rmkimport}
\hspace{100cm}
\rm{
\begin{enumerate}
\item
We note that a similar representation in terms of $\DDv$ as in Proposition \ref{prop3} has been obtained for
the fractional Laplacian, see \cite{V}. Our representation formula in Proposition \ref{prop3} continues to
hold if $f \in \mathcal{Z}^\beta(x)$ with an increasing modulus of continuity $\beta \in L^1_{\rm rad}(\R^d,\nu)$.
It is not difficult to show that for every $R>0$, under the assumptions of Proposition \ref{prop3} the
expression
	\begin{equation*}
	-\Phi(-\Delta)f(x)=\int_{B^c_{R}(x)}(f(y)-f(x))j(|x-y|)dy+\int_{B_R(0)}\DDv j(|h|)dh
	\end{equation*}
also holds.

\vspace{0.1cm}
\item
Corollary \ref{cor3} also holds if for every $x \in \Omega$ there exists an increasing modulus of continuity
$\beta_x:[0,\infty) \to [0,\infty)$ such that $\beta_x j \in L^1_{\rm loc}(\R^d)$ and $f
\in \mathcal{Z}^{\beta_x}(x)$.

\vspace{0.1cm}
\item
If for every $x \in \Omega$ there exists an increasing modulus of continuity $\beta_x: \R^+ \to \R^+$ such that
$\beta_x j \in L^1_{\rm loc}(\R^d)$ and
\begin{equation*}
|f(x+h)-f(x)|\le L(x)\beta_x(|h|), \quad  h \in B_{R_1(x)}(0),
\end{equation*}
then $f \in \mathcal{Z}^{\beta_x}(x)$. Moreover, by dominated convergence we have
\begin{equation*}
-\Phi(-\Delta)f(x)=\int_{\R^d}(f(x+h)-f(x))j(|h|)dh.
\end{equation*}

\vspace{0.1cm}
\item
We do not need to require the modulus of continuity to be increasing. Indeed, given a modulus of continuity
$\beta:\R^+ \to \R^+$ for a function $f$, the function $\widetilde{\beta}(t)=\sup_{0\le s \le t}\beta(s)$
is again a modulus of continuity for $f$, which is increasing.

\end{enumerate}
}
\end{rmk}

Using the definition \eqref{phiop}, by general arguments it follows that $\Phi(-\Delta)$ is a non-local, positive, self-adjoint
operator with core $C_{\rm c}^\infty(\R^d)$. Since $\Phi$ is a continuous function, it also follows that $\Spec \Phi(-\Delta) =
\Spec_{\rm ess} \Phi(-\Delta) = \Spec_{\rm ac} \Phi(-\Delta) = [0,\infty)$. Let $V: \R^d \to \R$ be a Borel function, called
potential. Since, as it will be seen below, the potentials $V$ that we treat in this paper are bounded and continuous, the
non-local Schr\"odinger operator $H = \Phi(-\Delta)+V$ can be defined via perturbation theory as a self-adjoint operator on
$\Dom(\Phi(-\Delta))$. Alternatively, $H$ can be obtained as a self-adjoint operator in a greater generality, specifically the
class of Kato-decomposable potentials with respect to $\Phi$, as the infinitesimal generator of the operator semigroup
$\{e^{-tH}: t\geq 0\}$. For details we refer to \cite{HIL12}.

Below we will work under the following standing assumption on the non-local Schr\"odinger operators $H$.
\begin{assumption}
	\label{eveq}
	Let $\Phi \in \mathcal B_0$ be a complete Bernstein function with L\'evy pair $(0,\mu)$, $V: \R^d \to \R$ be a
potential, and consider the non-local Schr\"odinger operator $H=\Phi(-\Delta)+V$. We assume that there exists a function
$\varphi \in L^2(\R^d) \cap L^\infty(\R^d)$, $\varphi > 0$, such that it is an eigenfunction of $H$ at zero eigenvalue, i.e.,
	$$
	H\varphi = 0, \quad \varphi \not \equiv 0,
	$$
	holds.
\end{assumption}
\noindent
We note that the assumption $\varphi \in L^\infty(\R^d)$ is not restrictive. It is known that this automatically holds if
$V$ is in an appropriate Kato-class with respect to $\Phi$, see \cite{KL19,KL17}. Since we are interested in bounded and
continuous potentials, see Section 3 below, this property is natural. In this paper we will, however, restrict to $\varphi
> 0$ and comment when our statements below also hold for $\varphi$ having zeroes. To deal with general eigenfunctions with
nodes would require an extension to our methods below and would increase the length of this paper significantly. Strictly
positive $\varphi$ are an obvious first choice in constructing potentials with embedded eigenvalues at zero, see the examples
in \eqref{eq:motiv_ex} and more in \cite{BY90,LS,JL18}. Some of our statements will hold for what we call (by a slight abuse)
zero-resonances, i.e., $L^p$-functions satisfying the eigenvalue equation with $p\neq 2$.

\subsubsection{\textbf{The massive relativistic operator in terms of the massless}}
\label{maasivemassless}
The first two cases in Example \ref{Eg2.1} are of special interest, which we single out specifically as they have many
applications in functional analysis, stochastic processes, and mathematical physics.

It is known (see, e.g., \cite{SSV}) that the L\'evy measure corresponding to $\Phi(u)=u^{\alpha/2}, \, 0 < \alpha < 2$,
is
\begin{equation}
\label{mumassless}
\mu(dt) = \mu(t)dt = \frac{\alpha}{2\Gamma(1-\frac{\alpha}{2})} \frac{1_{\{t > 0\}}}{t^{1+ \frac{\alpha}{2}}} \, dt
\end{equation}
and the corresponding operator $\Phi(-\Delta) = (-\Delta)^{\alpha/2}$ is the \emph{fractional Laplacian}. The L\'evy
measure of the fractional Laplacian is
\begin{equation*}
\nu(dx) = j(x)dx = \frac{2^\alpha \Gamma(\frac{d+\alpha}{2})}{\pi^{d/2}|\Gamma(-\frac{\alpha}{2})|}
\frac{dx}{|x|^{d+\alpha}}, \quad x \in \R^d\setminus \{0\}.
\end{equation*}
For $\Phi(u) = \Phi_{m,\alpha}(u)=(u+m^{2/\alpha})^{\alpha/2}-m$, $m> 0$, $\alpha\in (0, 2)$, we have
\begin{equation}
\label{mumassive}
\mu_{m,\alpha}(dt) = \mu_{m,\alpha}(t)dt = \frac{\alpha}{2\Gamma(1-\frac{\alpha}{2})} \frac{e^{-m^{2/\alpha}t}}
{t^{1+ \frac{\alpha}{2}}} 1_{\{t > 0\}} dt
\end{equation}
and the corresponding operator $\Phi_{m,\alpha}(-\Delta) = (-\Delta + m^{2/\alpha})^{\alpha/2}-m$ is the
\emph{relativistic Laplacian} with rest mass $m > 0$, whose L\'evy measure is
\begin{equation*}
\nu_{m,\alpha}(dx) = j_{m,\alpha}(x)dx = \frac{2^{\frac{\alpha-d}{2}} m^{\frac{d+\alpha}{2\alpha}} \alpha}
{\pi^{d/2}\Gamma(1-\frac{\alpha}{2})}
\frac{K_{(d+\alpha)/2} (m^{1/\alpha}|x|)}{|x|^{(d+\alpha)/2}} \, dx, \quad x \in \R^d\setminus \{0\},
\end{equation*}
with the modified Bessel function as given in \eqref{bessel3}. Since the fractional Laplacian is also obtained
by choosing $m=0$, we will denote the L\'evy density of the fractional Laplacian by $j_{0,\alpha}$, and refer
to these operators as the \emph{massless and massive relativistic operators} using the notations $L_{0,\alpha}$,
$L_{m,\alpha}$, respectively.

Next we show a rigorous counterpart of formula \eqref{Ryznar} given in the Introduction, which will be used below.
Recall the notation \eqref{defDv} and the explicit formulae \eqref{mumassless}-\eqref{mumassive}.
\begin{prop}
\label{rigRyznar}
Let $L_{m,\alpha} = \Phi_{m,\alpha}(-\Delta)=(-\Delta+m^{2/\alpha})^{\alpha/2}-m$, $L_{0,\alpha} = \Phi_{0,\alpha}(-\Delta)
= (-\Delta)^{\alpha/2}$, and $f \in \mathcal{Z}(\Omega)$, $\Omega \subseteq \R^d$.
%
%
Then
	\begin{equation}
\label{relativ}
	L_{m,\alpha}f(x) = L_{0,\alpha}f(x)-\frac{1}{2}\int_{\R^d} \DDv \sigma_{m,\alpha}(|h|)dh,
	\end{equation}
where
\begin{align}
\label{sigma}
\begin{split}
\sigma_{m,\alpha}(r)
&=
\frac{\alpha 2^{1-\frac{d-\alpha}{2}}}{\Gamma\left(1-\frac{\alpha}{2}\right)
\pi^{\frac{d}{2}}} \left(\frac{2^{\frac{d+\alpha}{2}-1}\Gamma\left(\frac{d+\alpha}{2}\right)}{r^{d+\alpha}}-
\frac{m^{\frac{d+\alpha}{2\alpha}}K_{\frac{d+\alpha}{2}}\left(m^{1/\alpha}r\right)}{r^{\frac{d+\alpha}{2}}}\right)
\\&=
\frac{\alpha 2^{1-\frac{d-\alpha}{2}}}{\Gamma\left(1-\frac{\alpha}{2}\right)\pi^{\frac{d}{2}}}
{\frac{1}{r^{d+\alpha}}}
\int_0^{m^{1/\alpha} r}w^{\frac{d+\alpha}{2}}K_{\frac{d+\alpha}{2}-1}(w)dw.
\end{split}
\end{align}
\end{prop}
\begin{proof}
Let $j_{m,\alpha}$ and $j_{0,\alpha}$ be the jump kernels corresponding to $\Phi_{m,\alpha}(u)$ and $\Phi_{0,\alpha}(u)$,
respectively, and write $\sigma_{m,\alpha}(r)=j_{m,\alpha}(r)-j_{0,\alpha}(r)$. Take any $x \in \R^d$. By using Proposition
\ref{prop3} and Corollary \ref{cor3} we have
	\begin{align*}
	\Phi_{m,\alpha}(-\Delta)f(x)&=\frac{1}{2}\int_{\R^d}\DDv j_{m,\alpha}(|h|)dh
	=\frac{1}{2}\int_{\R^d}\DDv (j_{0,\alpha}(|h|)+(j_{m,\alpha}(|h|)-j_{0,\alpha}(|h|))dh \\
	&=\frac{1}{2}\int_{\R^d}\DDv (j_{0,\alpha}(|h|)-\sigma_{m,\alpha}(|h|))dh.
	\end{align*}
	Similarly,
	\begin{equation*}
	\Phi_{0,\alpha}(-\Delta)f(x)=\frac{1}{2}\int_{\R^d}\DDv j_{0,\alpha}(|h|)dh.
	\end{equation*}
	Consider the integral over $\R^d$ of $|\DDv ||\sigma_{m,\alpha}(|h|)|$ split over $B_{R(x)}(0)$ and its complement.
We have
	\begin{equation*}
	\int_{B_{R(x)}(0)}|\DDv ||\sigma_{m,\alpha}(|h|)|dh\le  L(x)\int_{B_{R(x)}}|h|^2
     (j_{m,\alpha}(|h|)+j_{0,\alpha}(|h|))dr<\infty
	\end{equation*}
	and
	\begin{equation*}
	\int_{B_{R(x)}^c(0)}|\DDv ||\sigma_{m,\alpha}(|h|)|dh\le 4\Norm{f}{\infty}
    \left(\nu_{m,\alpha}(B_{R(x)}^c(0))+\nu_{0,\alpha}(B_{R(x)}^c(0))\right)<\infty.
	\end{equation*}
	Hence
	\begin{align*}
	\Phi_{m,\alpha}(-\Delta)f(x)&=\frac{1}{2}\int_{\R^d}\DDv (j_{0,\alpha}(|h|)-\sigma_{m,\alpha}(|h|))dh\\
	&=\frac{1}{2}\int_{\R^d}\DDv j_{0,\alpha}(|h|)dh -\frac{1}{2}\int_{\R^d}\DDv \sigma_{m,\alpha}(|h|)dh\\
	&=\Phi_{0,\alpha}(-\Delta)f(x)-\frac{1}{2}\int_{\R^d}\DDv \sigma_{m,\alpha}(|h|)dh,
	\end{align*}
which shows \eqref{relativ}. Using \eqref{mumassless}-\eqref{mumassive}, we furthermore have
\begin{equation*}
\sigma_{m,\alpha}(r)=\frac{\alpha 2^{-d}}{\Gamma\left(1-\frac{\alpha}{2}\right)\pi^{\frac{d}{2}}}
\int_0^{\infty}t^{-1-\frac{d+\alpha}{2}}(1-e^{-m^{2/\alpha}t})e^{-\frac{r^2}{4t}}dt.
\end{equation*}
Noting that
\begin{equation*}
\frac{1-e^{-m^{2/\alpha} t}}{t} = \int_0^{m^{2/\alpha}}e^{-tz}dz
\end{equation*}
and using Fubini's theorem, we obtain
\begin{align*}
\sigma_{m,\alpha}(r)&=\frac{\alpha 2^{-d}}{\Gamma\left(1-\frac{\alpha}{2}\right)\pi^{\frac{d}{2}}}
\int_0^{m^{2/\alpha}}\int_0^{\infty}t^{-\frac{d+\alpha}{2}}e^{-\left(zt+\frac{r^2}{4t}\right)}dtdz\\
&=\frac{\alpha 2^{\frac{\alpha-d}{2}}r^{1-\frac{d+\alpha}{2}}}{\Gamma\left(1-\frac{\alpha}{2}\right)\pi^{\frac{d}{2}}}
\int_0^{m^{2/\alpha}}z^{\frac{d+\alpha}{4}-\frac{1}{2}}\frac{1}{2}\left(\frac{r}{2\sqrt{z}}\right)^{\frac{d+\alpha}{2}-1}
\int_0^{\infty}t^{-\frac{d+\alpha}{2}}e^{-\frac{2z}{2}\left(t+\frac{r^2}{4zt}\right)}dtdz.
\end{align*}
Using \cite[(8.432.7)]{GR}, we get
\begin{equation*}
\frac{1}{2}\left(\frac{r}{2\sqrt{z}}\right)^{\frac{d+\alpha}{2}-1}\int_0^{\infty}t^{-\frac{d+\alpha}{2}}
e^{-\frac{2z}{2}\left(t+\frac{r^2}{4zt}\right)}dt=K_{\frac{d+\alpha}{2}-1}\left(r\sqrt{z}\right),
\end{equation*}
which by the change of variable $w=r\sqrt{z}$ gives
\begin{align*}
\sigma_{m,\alpha}(r)&=\frac{\alpha 2^{1-\frac{d-\alpha}{2}}r^{-d-\alpha}}{\Gamma\left(1-\frac{\alpha}{2}\right)\pi^{\frac{d}{2}}}
\int_0^{m^{1/\alpha} r}w^{\frac{d+\alpha}{2}}K_{\frac{d+\alpha}{2}-1}(w)dw.
\end{align*}
Finally, by using \cite[(5.52)]{GR} we obtain
\begin{equation*}
\int_0^{m^{1/\alpha} r}w^{\frac{d+\alpha}{2}}K_{\frac{d+\alpha}{2}-1}(w)dw=
2^{\frac{\alpha+d-2}{2}}\Gamma\left(\frac{\alpha+d}{2}\right)-r^{\frac{d+\alpha}{2}}
m^{\frac{d+\alpha}{2\alpha}}K_{\frac{d+\alpha}{2}}(m^{1/\alpha} r),
\end{equation*}
which yields \eqref{sigma}.
\end{proof}

For later use, we show as a consequence of the above that $\sigma_{m,\alpha}$ is monotone in a neighbourhood of infinity.
\begin{cor}
\label{sigmadecrease}
	There exists $R>0$ such that $\sigma_{m,\alpha}$ is decreasing in $(R,\infty)$.
\end{cor}
\begin{proof}
	We start from the second representation in \eqref{sigma}
	and notice that $\sigma_{m,\alpha} \in C^1(0,\infty)$. Consider the function
	\begin{equation*}
	\widetilde{\sigma}_{m,\alpha}(r)=\frac{1}{r^{d+\alpha}}\int_0^{m^{1/\alpha} r}w^{\frac{d+\alpha}{2}}K_{\frac{d+\alpha}{2}-1}(w)dw
	\end{equation*}
	and observe that
	\begin{equation*}
	\widetilde{\sigma}_{m,\alpha}'(r)
=-(d+\alpha)r^{-d-\alpha-1}\int_0^{m^{1/\alpha} r}w^{\frac{d+\alpha}{2}}
K_{\frac{d+\alpha}{2}-1}(w)dw+m^{\frac{d+\alpha+2}{\alpha}}r^{-d-\alpha+\frac{d+\alpha}{2}}K_{\frac{d+\alpha}{2}-1}(m^{1/\alpha} r).
	\end{equation*}
	By formula \cite[(5.52)]{GR} we obtain
	\begin{align*}
	\widetilde{\sigma}_{m,\alpha}'(r)&=-(d+\alpha)r^{-d-\alpha-1}2^{\frac{\alpha+d-2}{2}}
\Gamma\left(\frac{\alpha+d}{2}\right)\\
&\qquad +m^{\frac{d+\alpha}{2\alpha}}r^{-1-\frac{d+\alpha}{2}}K_{\frac{d+\alpha}{2}}(m^{1/\alpha} r)
+m^{\frac{d+\alpha+2}{2\alpha}}r^{-\frac{d+\alpha}{2}}K_{\frac{d+\alpha}{2}-1}(m^{1/\alpha} r).
	\end{align*}
Notice that as $r \to \infty$, by \eqref{bessel3asymp}
	\begin{align*}
	m^{\frac{d+\alpha}{2\alpha}}r^{-1-\frac{d+\alpha}{2}}K_{\frac{d+\alpha}{2}}(m^{1/\alpha} r) &\sim \sqrt{\frac{\pi}{2}} \, m^{\frac{d+\alpha-2}{2\alpha}}r^{-\frac{3}{2}-\frac{d+\alpha}{2}}e^{-m^{1/\alpha} r}\\
	m^{\frac{d+\alpha+2}{\alpha}}r^{-d-\alpha+\frac{d+\alpha}{2}}K_{\frac{d+\alpha}{2}-1}(m^{1/\alpha} r) &\sim \sqrt{\frac{\pi}{2}} \, m^{\frac{d+\alpha}{2\alpha}}r^{-\frac{1}{2}-\frac{d+\alpha}{2}}e^{-m^{1/\alpha} r},
	\end{align*}
	hence there exists $R>0$ such that for every $r>R$
	\begin{equation*}
	m^{\frac{d+\alpha}{2\alpha}}r^{-1-\frac{d+\alpha}{2}}K_{\frac{d+\alpha}{2}}(m^{1/\alpha} r)+
m^{\frac{d+\alpha+2}{\alpha}}r^{-d-\alpha+\frac{d+\alpha}{2}}K_{\frac{d+\alpha}{2}-1}(m^{1/\alpha} r)\le \frac{d+\alpha}{2}r^{-d-\alpha-1}2^{\frac{\alpha+d-2}{2}}\Gamma\Big(\frac{\alpha+d}{2}\Big)
	\end{equation*}
holds.	In particular, for $r>R$ we have
	\begin{align*}
	\widetilde{\sigma}_{m,\alpha}'(r)&\le -\frac{d+\alpha}{2}r^{-d-\alpha-1}2^{\frac{\alpha+d-2}{2}}
\Gamma\left(\frac{\alpha+d}{2}\right)<0,
	\end{align*}
	which shows the claim.
\end{proof}


Define the measure
\begin{equation*}
\Sigma_{m,\alpha}(A)=\int_{A}\sigma_{m,\alpha}(x)dx
\end{equation*}
on $\R^d$, which is finite, positive and gives full mass $\Sigma_{m,\alpha}(\R^d)=m$, see \cite[Lem. 2]{Ryz}.
We also introduce the measure
\begin{equation*}
\Sigma_{m,\alpha}^{(2)}(A)=\int_{A}|x|^2\sigma_{m,\alpha}(x)dx.
\end{equation*}
In the cited reference it has been shown that
\begin{equation}\label{Ryzbound}
\sigma_{m,\alpha}(r) \le C(\alpha,d)\frac{m^{2/\alpha}}{r^{\alpha+d-2}},
\end{equation}
with an $m$-dependent bound. In our cases below we need a more precise estimate on $\sigma_{m,\alpha}$. Using
the second representation in \eqref{sigma} and the estimate \eqref{bessel3asymp}, we obtain
\begin{equation}
\label{sigmasymp}
\sigma_{m,\alpha}(r)\sim \frac{\alpha 2^{\alpha}\Gamma\left(\frac{\alpha+d}{2}\right)}
{\pi^{\frac{d}{2}}\Gamma\left(1-\frac{\alpha}{2}\right)} \, \frac{1}{r^{d+\alpha}}.
\end{equation}
as $r \to \infty$.

\begin{rmk}
{\rm
It is interesting to note that the large $r$ asymptotic behaviour of $\sigma_{m,\alpha}(r)$ can be derived
from the bound \eqref{Ryzbound} in the $m \downarrow 0$ limit while keeping $r$ fixed. We obtain the
following:
\begin{enumerate}
\item
if $d+\alpha>2$, then
\begin{equation*}\label{sigmam1}
\sigma_{m,\alpha}(r)\sim \frac{\alpha\Gamma\left(\frac{d+\alpha}{2}\right)2^{\alpha-1}m^{\frac{2}{\alpha}}}
{\pi^{\frac{d}{2}}\Gamma\left(1-\frac{\alpha}{2}\right)(d+\alpha-2)r^{d+\alpha-2}}
\end{equation*}

\item if $d=1$, then
\begin{equation*}
|\sigma_{m,1}(r)|\sim
\left\{
\begin{array}{lcl}
 \frac{(2\gamma_{\rm {\tiny EM}}-1)m^2}{8\pi} &\mbox{if}& \alpha = 1 \vspace{0.1cm} \\ \vspace{0.2cm}
\frac{\alpha \sin\left(\frac{\pi\alpha}{2}\right)\Gamma(\alpha)
m^{\frac{1}{\alpha}+1}}{2(1-\alpha)\pi^{\frac{3}{2}}} &\mbox{if}& 0 < \alpha < 1,
\end{array}
\right.
\end{equation*}
where $\gamma_{\rm {\tiny EM}}$ is the Euler-Mascheroni constant.
\end{enumerate}
This will be further explored and its proof will be presented elsewhere.
}
\end{rmk}

We also have the following asymptotic results, whose proof we leave to the reader.
\begin{prop}\label{corSigma}
Let $p>1$. The following properties hold in the $R \to \infty$ limit:
\begin{eqnarray*}
 \Sigma_{m,\alpha}(B_R^c(0))&\sim& \frac{d\omega_d 2^{\alpha}\Gamma\left(\frac{\alpha+d}{2}\right)}
{\Gamma\left(1-\frac{\alpha}{2}\right)\pi^{\frac{d}{2}}}R^{-\alpha} \\
 \Sigma_{m,\alpha}^{(2)}(B_R(0))&\sim& \frac{d\omega_d\alpha 2^{\alpha}\Gamma\left(\frac{\alpha+d}{2}\right)}
{(2-\alpha)\Gamma\left(1-\frac{\alpha}{2}\right)\pi^{\frac{d}{2}}}R^{2-\alpha} \\
\Norm{\sigma_{m,\alpha}}{L^p(B_R^c(0))}&\sim& \left(\frac{d\omega_d}{(p-1)d+p\alpha}\right)^{\frac{1}{p}}\frac{\alpha 2^{\alpha}\Gamma\left(\frac{\alpha+d}{2}\right)}{\Gamma\left(1-\frac{\alpha}{2}\right)\pi^{\frac{d}{2}}}R^{-qd-\alpha},
\end{eqnarray*}
where $q$ is the H\"older-conjugate exponent of $p$.
\end{prop}
\begin{proof}
The proof is similar to Proposition \ref{light} and we skip it.
\end{proof}

For every function $f \in \cZ(x)$ we define the operator $G_{m,\alpha}$ as
\begin{equation*}
G_{m,\alpha}f(x)=\frac{1}{2}\int_{\R^d}\DDv \sigma_{m,\alpha}(|h|)dh,
\end{equation*}
and thus by Proposition \ref{rigRyznar} we have
\begin{equation}
\label{withG}
L_{m,\alpha}f(x)=L_{0,\alpha}f(x)-G_{m,\alpha}f(x).
\end{equation}
Since $\Sigma_{m,\alpha}$ is a finite measure, we can extend $G_{m,\alpha}$ to every function $f \in L^\infty(\R^d)$.
Moreover, note that Remark \ref{rmkimport} also applies for the operator $G_{m,\alpha}$. The operator $G_{m,\alpha}$
is easy to control.
\begin{lem}
\label{Glem}
	Let $f \in L^\infty(\R^d)$. If $d+\alpha>2$, then
	\begin{equation*}
	\Norm{G_{m,\alpha}f}{\infty}\le 2m\Norm{f}{\infty}.
	\end{equation*}
\end{lem}
\begin{proof}
We have directly by the definition	
\begin{equation*}
|G_{m,\alpha}f(x)|\le \frac{1}{2}\int_{\R^d}|D_hf(x)|\sigma_{m,\alpha}(|h|)dh
\le 2\Norm{f}{\infty}\int_{\R^d}\sigma_{m,\alpha}(|h|)dh=2m\Norm{f}{\infty}.
\end{equation*}
\end{proof}

\section{Boundedness and continuity of the potentials}
\subsection{Boundedness}
Recall the notations \eqref{defDv} and \eqref{defJ}. Using the representation of $\Phi(-\Delta)\varphi$ given in Corollary
\ref{cor3}, from \eqref{Veq} we have the expression
\begin{equation}
\label{Vformula}
V(x)=\frac{1}{2\varphi(x)}\int_{\R^d}\Dv j(|h|)dh
\end{equation}
for the potential, which we will extensively use in what follows. To keep the notations and statements simple, we note that
the functions $L_\varphi, R_\varphi$ appearing in the statements below are understood to be the functions appearing in the
definition of the space $\cZ(\R^d)$ applied to $f=\varphi$.

\begin{thm}\label{thm1}
Let Assumption \ref{eveq} hold with $\varphi \in \mathcal{Z}(\R^d)$.
Then
	\begin{align}\label{eqbound}
	\begin{split}
	V(x)\le \frac{2\Norm{\varphi}{\infty}}
{\varphi(x)}\,\nu(B_{R_\varphi(x)}^c(0))+\frac{L_\varphi(x)}{2\varphi(x)}\J(R_\varphi(x)) < \infty, \quad x \in \R^d.
	\end{split}
	\end{align}
\end{thm}
\begin{proof}
	We have by \eqref{Vformula}
	\begin{align*}
	|V(x)|\le\frac{1}{2\varphi(x)}\left(\int_{B^c_{R_\varphi(x)}} +\int_{B_{R_\varphi(x)}}\right)|\Dv |j(|h|)dh =
    I_1(x) + I_2(x).
	\end{align*}
	The first term gives $I_1(x)\le 4\Norm{\varphi}{\infty}\nu(B_{R_\varphi(x)}^c(0))<\infty$.
	For the second term we have $I_2(x)\le L_\varphi(x)\J(R_\varphi(x))< \infty$, by the fact that $\varphi \in
\mathcal{Z}(\R^d)$ and Proposition \ref{prop3}.
\end{proof}

\begin{cor}\label{corbound}\label{cor4}
Consider $\Omega \subseteq \R^d$ and let Assumption \ref{eveq} be satisfied.
If $\varphi \in \mathcal{Z}_{\rm b}(\Omega)$, then $V \in L^\infty(\Omega)$.
If $\varphi \in \mathcal{Z}_{\rm b,loc}(\Omega)$, then $V \in L^\infty_{\rm {loc}}(\Omega)$.
\end{cor}
\begin{proof}
Part (1) results immediately by noting that \eqref{eqbound} is now independent of $x \in \Omega$. For (2) note that since
for any fixed compact set $K \subseteq \Omega$ the bound given in \eqref{eqbound} is again independent of $x \in K$, the
result follows.
\end{proof}
\begin{rmk}
\rm{
Estimate \eqref{eqbound} also holds for every $x \in \Omega$ when $\R^d$ is replaced by a proper subset $\Omega$. Also,
Theorem \ref{thm1} and Corollary \ref{corbound} continue to hold if $\varphi$ is strictly negative. Furthermore, the
also hold if $\varphi$ does not have a definite sign on any set $\Omega \subseteq \R^d$ which does not contain the set
$\{\varphi(x)=0\}$ of zeroes.
Moreover, Theorem \ref{thm1} and Corollary \ref{cor4} hold also if $\varphi \in \cZ^{\beta}(\R^d)$ (or, respectively,
$\cZ^\beta_{\rm b}(\R^d)$, $\cZ^\beta_{\rm b,loc}(\R^d)$) for a modulus of continuity $\beta: \R^+\to \R^+$ such that
$\beta \in L^1_{\rm rad}(\R^d,\nu)$, by replacing $\cJ(R)$ with $\cJ_\beta(R)$ as given by \eqref{Jbeta}.}
\end{rmk}
\subsection{Continuity}
\begin{thm}\label{thm2}
Let $\Omega \subset \R^d$ be an open set. If Assumption \ref{eveq} holds with $\varphi \in \mathcal{Z}_{\rm b,loc}
(\Omega)\cap C(\Omega)$, then $V \in C(\Omega)$.
\end{thm}
\begin{proof}
Since $\varphi >0$ and $\varphi \in C(\Omega)$, it suffices to show the continuity of the function
$\widetilde{V}(x):=-\Phi(-\Delta)(x)$. Fix $R<R_\varphi$ such that $B_R(x)\subset \Omega$. Consider $\delta x \in \R^d$
such that $|\delta x|<\frac{R}{4}$ and write $\wV(x+\delta x)-\wV(x)$ by using Remark \ref{rmkimport} (4). We have with
$\widehat B_R(x,x+\delta x)=B_R(x)\cup B_R(x+\delta x)$
	\begin{align*}
	\wV(x+\delta x)-\wV(x)&=\int_{\R^d\setminus B_{R}(x+\delta x)}(\varphi(y)-\varphi(x+\delta x))j(|x+\delta x-y|)dy\\
	&\quad -\int_{\R^d\setminus B_{R}(x)}(\varphi(y)-\varphi(x))j(|x-y|)dy
	+\int_{B_R(0)}({\rm D}_h \varphi(x+\delta x)- \Dv) j(|h|)dh \\
    &= \int_{\widehat B^c_{R}(x,x+\delta x)}(\varphi(y)-\varphi(x+\delta x))j(|x+\delta x-y|)dy\\
	&\qquad -\int_{\widehat B^c_{R}(x,x+\delta x)}(\varphi(y)-\varphi(x))j(|x-y|)dy\\
	&\qquad +\int_{B_{R}(x)\setminus B_{R}(x+\delta x)}(\varphi(y)-\varphi(x+\delta x))j(|x+\delta x-y|)dy\\
	&\qquad -\int_{B_{R}(x+\delta x)\setminus B_{R}(x)}(\varphi(y)-\varphi(x))j(|x-y|)dy \\
	&\qquad +\int_{B_R(0)}({\rm D}_h \varphi(x+\delta x)- \Dv) j(|h|)dh.
	\end{align*}
By adding and subtracting the term $\varphi(x)\int_{\widehat B^c_R(x,x+\delta x)}j(|x+\delta x-y|)dy$ marked by the
underlining, we get
	\begin{align*}
	\wV(x+\delta x)-\wV(x)
\label{thm2pass1}
    &=\int_{\widehat B_{R}^c(x,x+\delta x)}\varphi(y)(j(|x+\delta x-y|)-j(|x-y|))dy\\
	&\qquad +(\underline{\varphi(x)}-\varphi(x+\delta x))\int_{\widehat B_{R}^c(x,x+\delta x)}j(|x+\delta x-y|)dy\\
	&\qquad +\varphi(x)\int_{\widehat B_{R}^c(x,x+\delta x)}(j(|x-y|)- \underline{j(|x+\delta x-y|)})dy\\
	&\qquad +\int_{B_{R}(x)\setminus B_{R}(x+\delta x)}(\varphi(y)-\varphi(x+\delta x))j(|x+\delta x-y|)dy\\
	&\qquad -\int_{B_{R}(x+\delta x)\setminus B_{R}(x)}(\varphi(y)-\varphi(x))j(|x-y|)dy\\
	&\qquad +\int_{B_R(0)}({\rm D}_h \varphi(x+\delta x)- \Dv) j(|h|)dh.
	\end{align*}
	In the fourth integral we make the change of variable $y=x+z$ giving
	\begin{equation*}
	\int_{B_{R}(x)\setminus B_{R}(x+\delta x)}(\varphi(y)-\varphi(x+\delta x))j(|x+\delta x-y|)dy =
\int_{B_{R}(0)\setminus B_{R}(\delta x)}(\varphi(x+z)-\varphi(x+\delta x))j(|z-\delta x|)dz,
	\end{equation*}
    and $y=x+\delta x-z$ in the fifth integral to obtain
	\begin{equation*}
	\int_{B_{R}(x+\delta x)\setminus B_{R}(x)}(\varphi(y)-\varphi(x))j(|x-y|)dy=
\int_{B_{R}(0)\setminus B_{R}(-\delta x)}(\varphi(x+\delta x-z)-\varphi(x))j(|z-\delta x|)dz.
	\end{equation*}
	After taking absolute value, we finally get
	\begin{align*}
	\begin{split}
	|\wV(x+\delta x)-\wV(x)|&\le \int_{\widehat B_{R}^c(x,x+\delta x)}\varphi(y)|j(|x+\delta x-y|)-j(|x-y|)|dy\\
	&\qquad +|\varphi(x)-\varphi(x+\delta x)|\int_{\widehat B_{R}^c(x,x+\delta x)}j(|x+\delta x-y|)dy\\
	&\qquad +\varphi(x)\int_{\widehat B_{R}^c(x,x+\delta x)}|j(|x-y|)-j(|x+\delta x-y|)|dy\\
	&\qquad +\int_{B_{R}(0)\setminus B_{R}(\delta x)}|\varphi(x+z)-\varphi(x+\delta x)-\varphi(x+\delta x-z)+\varphi(x)|j(|z-\delta x|)dz\\
	&\qquad +\int_{B_R(0)}|{\rm D}_h \varphi(x+\delta x)- \Dv| j(|h|)dh\\
	&=I_1(\delta x)+I_2(\delta x)+I_3(\delta x)+I_4(\delta x)+I_5(\delta x).
	\end{split}
	\end{align*}
	Consider the first integral. Since $B_{R/2}(x)\subset \widehat B_R(x,x+\delta x)$, we have
	\begin{align*}
	\begin{split}
	I_1(\delta x)
	\le\Norm{\varphi}{\infty}\int_{B^c_{R/2}(x)}|j(|x+\delta x-y|)-j(|x-y|)|dy.
	\end{split}
	\end{align*}
Since	
\begin{equation*}
	|x+\delta x-y|\ge |x-y|-|\delta x| \ge |x-y|-\frac{R}{4},
	\end{equation*}
and in $B_{R/2}(x)$ we have that $|x-y|\ge \frac{R}{2}$ 
giving
	$|x+\delta x-y|\le \frac{|x-y|}{2}$.
	As $j$ is decreasing, we have
	\begin{equation*}
	|j(|x+\delta x-y)-j(|x-y|)|\le 2 j\left(\frac{|x-y|}{2}\right).
	\end{equation*}
Thus	
	\begin{align*}
	\int_{B^c_{R/2}(x)}j\left(\frac{|x-y|}{2}\right)dy
    =2^d\int_{B^c_{R/4}(0)}j\left(|w|\right)dw=2^d\nu\left(B_{R/4}^c(0)\right)<\infty
	\end{align*}
by the change of variable $w=\frac{y-x}{2}$. Thus by the continuity of $j$	
and dominated convergence
	\begin{equation*}
	\lim_{\delta x \to 0}\int_{B^c_{R/2}(x)}|j(|x+\delta x-y|)-j(|x-y|)|dy=0,
	\end{equation*}
	giving $\lim_{\delta x \to 0}I_1(\delta x)=0$.
Considering $I_2(x)$, we have by using the same estimate as before
	\begin{equation*}
	I_2(\delta x)\le |\varphi(x)-\varphi(x+\delta x)|\int_{B^c_{R/2}(x)}j\left(\frac{|x-y|}{2}\right)dy
	\end{equation*}
	thus by continuity of $\varphi$, we have $\lim_{\delta x \to 0}I_2(\delta x)=0$.
	Concerning $I_3(\delta x)$, we have similarly
	\begin{equation*}
	I_3(\delta x)\le \varphi(x)\int_{B^c_{R/2}(x)}|j(|x-y|)-j(|x+\delta x-y)|dy,
	\end{equation*}
	and $\lim_{\delta x \to 0}I_3(\delta x)=0$ as above.
	Next consider $I_4(\delta x)$. We have
	\begin{equation*}
	I_4(\delta x) \le 4\Norm{\varphi}{\infty}j(R)|B_R(0)\setminus B_R(\delta x)|
	\end{equation*}
	where the bars denote Lebesgue measure, giving $\lim_{\delta x \to 0}I_4(\delta x)=0$.
	Finally, consider $I_5(\delta x)$. To use dominated convergence, notice that $x+\delta x \in
    \overline{B}_{R/2}(x)$. Since for $K=\overline{B}_{R/2}(x)$ we have that
	$|{\rm D}_h \varphi(x+\delta x)- \Dv|\le 2L(K)|h|^2$ by assumption, where $L(K)=\sup_{x \in K}L_\varphi(x)$,
    and
	$\int_{B_R(0)}|h|^2j(|h|)dh=\J(R)<\infty$,
	using continuity of $\varphi$ it follows that
	$\lim_{\delta x \to 0}I_5(\delta x)=0$.
	Hence $\lim_{\delta x \to 0}|\wV(x+\delta x)-\wV(x)|=0$ and thus $\wV$ is continuous, and so is $V$.
\end{proof}

\noindent
{\begin{rmk}
{\rm
Choosing $\Omega=\R^d$ and $\varphi \in \mathcal{Z}_{\rm b}(\R^d)$, we have that $\varphi \in C(\R^d)$, thus Theorem \ref{thm2}
actually implies Corollary \ref{cor4} in this case. Theorem \ref{thm2} continues to hold if $\varphi \in \cZ_{\rm b}^\beta(\Omega)
\cap C(\Omega)$ with $\beta\in L^1_{\rm rad}(\R^d,\nu)$.}
\end{rmk}}

\section{Decay properties of the potentials at infinity}
\subsection{Conditions on decay to zero}
A first point of interest about decay properties is when does $V$ tend to zero at all as $|x| \to \infty$. Recall the notation
\eqref{ZCspace} and its simplified form for $\beta(r)=r^2$.
\begin{thm}\label{thm3}
Let Assumption \ref{eveq} hold with $\varphi \in \cZ_{C}(\R^d)$ for some constant $C \in (0,1)$. Then
	\begin{equation*}
	|V(x)|\le \frac{1}{2} \frac{L_{\varphi}(x)}{\varphi(x)}\J(C|x|)+2\Norm{\varphi}{\infty}\frac{\nu(B_{C|x|}^c(0))}{\varphi(x)},
\quad x \in B_{M_\varphi}^c(0).
	\end{equation*}
In particular, if
	$$
\lim_{|x| \to \infty}\frac{L_\varphi(x)\J(C|x|)}{\varphi(x)}=0 \quad \mbox{and} \quad
	\lim_{|x| \to \infty}\frac{\nu(B_{C|x|}^c(0))}{\varphi(x)}=0,
	$$
then $\lim_{|x| \to \infty}V(x)=0$.
\end{thm}
\begin{proof}
Since $\varphi \in \cZ_{\rm b}(B_{M_\varphi}^c(0))$, using \eqref{Vformula} for $x \in B_{M_\varphi}^c(0)$,
we write
	\begin{align*}
	|V(x)|& \le \frac{1}{2\varphi(x)}\int_{\R^d}|\Dv |j(|h|)dh \\
    &\le \frac{1}{2\varphi(x)}\int_{B_{C|x|}(0)}|\Dv |j(|h|)dh\\
	&\quad +\frac{1}{2\varphi(x)}\int_{B_{C|x|}^c(0)}\varphi(x+h)j(|h|)dh
	+\int_{B_{C|x|}^c(0)}j(|h|)dh
    +\frac{1}{2\varphi(x)}\int_{B_{C|x|}^c(0)}\varphi(x-h)j(|h|)dh.
	\end{align*}
	Replacing $h$ by $-h$ in the fourth integral gives
	\begin{align*}
	|V(x)|&\le \frac{1}{2\varphi(x)}\int_{B_{C|x|}(0)}|\Dv |j(|h|)dh\\
	&\quad +\frac{1}{\varphi(x)}\int_{B_{C|x|}^c(0)}\varphi(x+h)j(|h|)dh+\int_{B_{C|x|}^c(0)}j(|h|)dh\\
	&=I_1(x)+I_2(x)+I_3(x).
	\end{align*}
	The first integral can be estimated as
	\begin{align*}
	I_1(x)&\le \frac{L_\varphi(x)}{2\varphi(x)}\int_{B_{C|x|}(0)}|h|^2j(|h|)dh
	=\frac{1}{2}\frac{L_\varphi(x)}{\varphi(x)}\J(C|x|).
	\end{align*}
	For the second we get
	\begin{equation*}
	I_2(x)\le \frac{\Norm{\varphi}{\infty}}{\varphi(x)} \, \nu(B_{C|x|}^c(0)),
	\end{equation*}
	while $I_3(x)=\nu(B_{C|x|}^c(0))$. In sum,
	\begin{align*}
	|V(x)|&\le\frac{1}{2}\frac{L_\varphi(x)}{\varphi(x)}\J(C|x|)+
\left(\frac{\Norm{\varphi}{\infty}}{\varphi(x)}+1\right)\nu(B_{C|x|}^c(0))\\
	&\le \frac{1}{2}\frac{L_\varphi(x)}{\varphi(x)}\J(C|x|)+
\frac{2\Norm{\varphi}{\infty}}{\varphi(x)} \, \nu(B_{C|x|}^c(0)).
	\end{align*}
\end{proof}

\begin{thm}
\label{thm4-5}
Let Assumption \ref{eveq} hold with $\varphi \in \cZ_C(\R^d)\cap L^p(\R^d)$ for some
$C \in (0,1)$ and $p \geq 1$. Then for all $x \in B_{M_\varphi}^c(0)$ we have
	\begin{align*}
|V(x)| & \le \frac{1}{2} \frac{L_\varphi(x)}{\varphi(x)}\J(C|x|)+
\Norm{\varphi}{1}\frac{j(C|x|)}{\varphi(x)}+\nu(B_{C|x|}^c(0)),  \quad \mbox{if \ $p = 1$}, \\
\vspace{0.1cm}
|V(x)| & \le \frac{1}{2} \frac{L_\varphi(x)}{\varphi(x)}\J(C|x|)+
\Norm{\varphi}{p}\frac{\Norm{j}{L^q(B_{C|x|}^c(0))}}{\varphi(x)} +\nu(B_{C|x|}^c(0)), \quad \mbox{if \ $p >1$,}
	\end{align*}
	where $q$ is the H\"older conjugate of $p$. In particular, if
\begin{eqnarray*}
&& \lim_{|x| \to \infty}\frac{L_\varphi(x) \J(C|x|)}{\varphi(x)}=0 \quad \mbox{and} \quad
\lim_{|x| \to \infty}\frac{j(C|x|)}{\varphi(x)}=0, \quad \mbox{for \ $p = 1$,} \\
\vspace{0.1cm}
&& \lim_{|x| \to \infty}\frac{L_\varphi(x)\J(C|x|)}{\varphi(x)}=0 \quad \mbox{and} \quad
\lim_{|x| \to \infty}\frac{\Norm{j}{L^q(B_{C|x|}^c(0))}}{\varphi(x)}=0, \quad \mbox{for \ $p >1$,}
\end{eqnarray*}
then $\lim_{|x| \to \infty}V(x)=0$.
\end{thm}
\begin{proof}
Following through the proof of Theorem \ref{thm3}, we only need to change the estimate on $I_2$. Choose $p=1$.
Since  $\varphi \in L^1(\R^d)$ and $j$ is decreasing, we have
\begin{equation*}
I_2(x)\le \frac{\Norm{\varphi}{1}}{\varphi(x)} j(C|x|).
\end{equation*}
Similarly, when $p>1$, using that $h \mapsto j(|h|)$ belongs to $L^q(B_{C|x|}^c(0))$ we have by H\"older
inequality
\begin{equation*}
I_2(x)\le \frac{\Norm{\varphi}{p}}{\varphi(x)}\, \Norm{j}{L^q(B_{C|x|}^c(0))}.
\end{equation*}
\end{proof}

\begin{rmk}
{\rm
\hspace{100cm}
\begin{enumerate}
\item
Theorems \ref{thm3}-\ref{thm4-5} continue to hold if $\varphi \in \cZ_C^\beta(\R^d)$ for a modulus of continuity
$\beta\in L^1_{\rm rad}(\R^d,\nu)$, when $\cJ$ is replaced by $\cJ_\beta$.
\vspace{0.1cm}
\item
Also, Theorem \ref{thm3} holds if instead of strict positivity of $\varphi$ we require the set of zeroes $\{x \in \R^d:
\varphi(x)=0\}$ to be bounded.
\vspace{0.1cm}
\item
It is an interesting question whether the potential may be compactly supported. Note that if $\supp\varphi$ is compact,
then \eqref{Vformula} can hold only inside the support of $\varphi$. However, we also know that $\Phi(-\Delta)\varphi$
is zero outside $\supp\varphi$. Moreover, if we consider $x \not \in \supp\varphi$ and $r_1,r_2>0$ such that $B_{r_1}(x)
\subset(\supp\varphi)^c$ and $\supp\varphi \subset B_{r_2}(0)\subset B_{r_1}^c(x)$, we have that
\begin{align*}
0 & =\Phi(-\Delta)\varphi(x)=\frac{1}{2}\int_{B_{r_1}^c(0)}\Dv j(|h|)dh \\
& \qquad =\int_{B_{r_1}^c(x)}\varphi(y)j(|x-y|)dy=\int_{B_{r_2}(0)}\varphi(y)j(|x-y|)dy.
\end{align*}
Since $j>0$, a conclusion is that $\varphi$ must change sign inside $B_{r_2}(0)$. In particular, this means that $\varphi$
can be at most rotationally antisymmetric. However, this is also impossible by the fact that $j$ is strictly decreasing.
\end{enumerate}
}
\end{rmk}

\subsection{Decay rates for regularly varying L\'evy intensities}
Now we turn to deriving actual decay rates of $V$ for various choices of the operator $\Phi(-\Delta)$.

\begin{thm}\label{thm6}
Let Assumption \ref{eveq} hold with $\varphi \in \cZ_{C_1}(\R^d)$ for some $C_1 \in (0,1)$. Also, suppose that
there exists $\alpha \in (0,2)$ and a function $\ell$ slowly varying at zero such that $\Phi(u)\sim
u^{\frac{\alpha}{2}}\ell(u)$ as $u \downarrow 0$. If there exist
\begin{enumerate}
	\item $C_2,C_3,M_1>0$ and $\kappa \in \left(0,\frac{d+\alpha}{2}\right)$ such that
	$C_2 |x|^{-2\kappa}\le \varphi(x)\le C_3 |x|^{-2\kappa}$ for $x \in B_{M_1}^{c}(0)$, \vspace{0.1cm}
	\item $M_2,C_4>0$ such that $L_\varphi(x)\le C_4|x|^{-(2\kappa+2)}$ for  $x \in B_{M_2}^c(0)$,
\end{enumerate}
then
\begin{equation*}
V(x) =
\left\{
\begin{array}{lcl}
O\big(|x|^{2\kappa-d-\alpha}\well(|x|^2)\big)  &\mbox{if}& \kappa \in \big(\frac{d}{2},\frac{d+\alpha}{2}\big) \vspace{0.1cm} \\
O\big(|x|^{-\alpha}\well(|x|^2)\big) &\mbox{if}& \kappa \in \big(0,\frac{d}{2}\big) \vspace{0.1cm} \\
O\big(|x|^{-\alpha}\log(|x|)\well(|x|^2)\big) &\mbox{if}& \kappa=\frac{d}{2},
\end{array}
\right.
\end{equation*}
where $\well(r)=\ell(1/r)$.
\end{thm}
\begin{proof}

\medskip
\noindent
\emph{Case 1:} First consider $\kappa \in \left(\frac{d}{2},\frac{d+\alpha}{2}\right)$. Then it follows that
$\varphi \in L^1(\R^d)$ and we can make use of Theorem \ref{thm4-5}. We obtain for every $x \in B_{M_\varphi}^c(0)$
	\begin{equation}\label{thm7pass1}
	|V(x)|\le \frac{1}{2} \frac{L_\varphi(x)}{\varphi(x)}\J (C_1|x|)+
\Norm{\varphi}{1}\frac{j(C_1|x|)}{\varphi(x)}+\nu(B_{C_1|x|}^c(0)).
	\end{equation}
	By Corollary \ref{cor2} (1) there exist constants $C_5,M_3>0$ such that
	\begin{equation}\label{thm7pass2}
	\nu(B_{C_1|x|}^c(0))\le C_5|x|^{-\alpha}\well(|x|^2), \quad  x \in B_{M_3}^c(0),
	\end{equation}
	and by Proposition \ref{prop2} there exist constants $C_6,M_4>0$ such that
	\begin{equation}\label{thm7pass3}
	j(C_1|x|)\le C_6|x|^{-d-\alpha}\well(|x|^2), \quad  x \in B_{M_4}^c(0).
	\end{equation}
	Furthermore, Corollary \ref{cor3} (2) implies that there exist $C_7,M_5>0$ such that
	\begin{equation}\label{thm7pass4}
	\cJ(C_1|x|)\le C_7|x|^{2-\alpha}\well(|x|^2), \quad x \in B_{M_5}^c(0).
	\end{equation}
	Take $M=\max_{i=0,\dots,5}M_i$, where $M_0=M_\varphi$, and $|x|>M$. Since $|x|>M$ we have
	\begin{equation}\label{thm7pass5}
	\varphi(x)\ge \frac{C_2}{|x|^{2\kappa}} \quad \mbox{and} \quad L_\varphi(x)\le \frac{C_4}{|x|^{2\kappa+2}}.
	\end{equation}
	Applying the estimates \eqref{thm7pass2}-\eqref{thm7pass5} to \eqref{thm7pass1} gives
	\begin{equation*}
	|V(x)|\le \frac{1}{2}\frac{C_{7}C_4}{C_{2}}|x|^{-\alpha}\well(|x|^2)+\Norm{\varphi}{1}
\frac{C_6}{C_2}|x|^{2\kappa-\alpha-d}\well(|x|^2)+C_5|x|^{-\alpha}\well(|x|^2).
	\end{equation*}
	In particular, there exists a constant $C>0$ such that
	\begin{equation*}
	|V(x)|\le C|x|^{2\kappa-\alpha-d}\well(|x|^2).
	\end{equation*}

\medskip
\noindent
\emph{Case 2:}
	Consider $\kappa \in \left(0,\frac{d}{2}\right]$. Recall that by Corollary \ref{cor3} and
Remark \ref{rmkimport} (4) we have for $x \in B_{M_\varphi}^c(0)$
	\begin{align*}
	V(x)&=\frac{1}{\varphi(x)}\left(\int_{\R^d\setminus B_{C_1|x|}(x)}(\varphi(y)-\varphi(x))j(|x-y|)dy
	+\frac{1}{2}\int_{B_{C_1|x|}(0)}\Dv j(|h|)dh\right).
	\end{align*}
	Splitting up the integral we have
	\begin{align*}
	\begin{split}
	V(x)&=\frac{1}{\varphi(x)}\left(\int_{B_{|x|}^c(0)\setminus B_{C_1|x|}(x)}\varphi(y)j(|x-y|)dy\right.
	+\int_{B_{|x|}(0)\setminus B_{C_1|x|}(x)}\varphi(y)j(|x-y|)dy\\
	&\qquad -\varphi(x)\int_{B^c_{C_1|x|}(x)}j(|x-y|)dy
	\left.+\frac{1}{2}\int_{B_{C_1|x|}(0)}\Dv j(|h|)dh\right) \\
	&=\frac{1}{\varphi(x)}\left(\int_{B_{|x|}^c(0)\setminus B_{C_1|x|}(x)}\varphi(y)j(|x-y|)dy\right.
	+\int_{B_{|x|}(0)\setminus B_{C_1|x|}(x)}\varphi(y)j(|x-y|)dy\\
	&\left. \qquad +\frac{1}{2}\int_{B_{C_1|x|}(0)}\Dv j(|h|)dh\right)-\nu(B_{C_1|x|}^c(0)).
	\end{split}
	\end{align*}
	This gives
	\begin{align}\label{thm7pass6}
	\begin{split}
	|V(x)|
	&\le\frac{1}{\varphi(x)}\left(\int_{B_{|x|}^c(0)\setminus B_{C_1|x|}(x)}|\varphi(y)|j(|x-y|)dy
	+\int_{B_{|x|}(0)\setminus B_{C_1|x|}(x)}|\varphi(y)|j(|x-y|)dy \right.\\
	&\left. \qquad +\frac{L_\varphi(x)}{2}\J(C_1|x|)\right)+\nu(B_{C|x|}^c(0))\\
	&=:\frac{1}{\varphi(x)}I_1(x)+\frac{1}{\varphi(x)}I_2(x)+\frac{L_\varphi(x)}{2\varphi(x)}\cJ(C_1|x|)+\nu(B_{C_1|x|}^c(0)).
	\end{split}
	\end{align}
	Consider $I_1(x)$. Choose $p>\frac{d}{2\kappa}$ and $q$ such that $\frac{1}{p}+\frac{1}{q}=1$. By the H\"older inequality we have
	\begin{align*}
	I_1(x)&\le \left(\int_{B^c_{|x|}(0)\setminus B_{C_1|x|}(x)}\varphi^p(y)dy\right)^{\frac{1}{p}}
\left(\int_{B^c_{|x|}(0)\setminus B_{C_1|x|}(x)}j(|x-y|)^qdy\right)^{\frac{1}{q}} \\
	&\le \left(\int_{B^c_{|x|}}\varphi^p(y)dy\right)^{\frac{1}{p}}\Norm{j}{L^q(B_{C_1|x|}^c(0))}.
	\end{align*}
	By Corollary \ref{cor2} (3) there exists two constants $C_{8}>0$ and $M_6>0$ such that, for any $x \in B_{M_6}^c(0)$,
	\begin{equation*}
	\Norm{j}{L^q(B_{C_1|x|}^c(0))}\le C_{8}|x|^{-\frac{d}{p}-\alpha}\well(|x|^2)^q,
	\end{equation*}
	thus
	\begin{equation*}
	I_1(x)\le C_{8}\left(\int_{B^c_{|x|}(0)}\varphi^p(y)dy\right)^{\frac{1}{p}}|x|^{-\frac{d}{p}-\alpha}\well(|x|^2).
	\end{equation*}
	Let $M=\max_{i=0,\dots,6}M_i$ and consider $|x|>M$. Since $|x|>M$, we have that for $y \in B_{|x|}^c(0)$ $\varphi(y)
    \le \frac{C_3}{|y|^{2\kappa}}$, and so
	\begin{align*}
	I_1(x)&\le C_{3}C_{8}
\left(\int_{B^c_{|x|}(0)}\frac{1}{|y|^{2p\kappa}}dy\right)^{\frac{1}{p}}|x|^{-\frac{d}{p}-\alpha}\well(|x|^2)\\
	&= C_{3}C_{8}(d\omega_d)^{\frac{1}{p}}\left(\int_{|x|}^{\infty}r^{d-2p\kappa-1}dr\right)^{\frac{1}{p}}|x|^{-\frac{d}{p}-\alpha}\well(|x|^2)
	=\frac{C_{3}C_{11}(d\omega_d)^{\frac{1}{p}}}{(2p\kappa-d)^{\frac{1}{p}}}|x|^{-2\kappa-\alpha}\well(|x|^2)
	\end{align*}
	Finally, by using that $\varphi(x)\ge \frac{C_2}{|x|^{2\kappa}}$ due to $|x|>M$, we get
	\begin{equation}\label{estI1decay1}
	\frac{1}{\varphi(x)}I_1(x)\le \frac{C_{3}C_{8}(d\omega_d)}{C_2(2p\kappa-d)^{\frac{1}{p}}}|x|^{-\alpha}\well(|x|^2):=C_9|x|^{-\alpha}\well(|x|^2).
	\end{equation}
	Applying estimates \eqref{thm7pass2}, \eqref{thm7pass4}-\eqref{thm7pass5} and \eqref{estI1decay1} to \eqref{thm7pass6}, we get
	\begin{equation}\label{thm7pass8}
	|V(x)|\le\frac{1}{\varphi(x)}I_2(x)+C_{10}|x|^{-\alpha}\well(|x|^2),
	\end{equation}
	where
	\begin{equation*}
	C_{10}=C_9+\frac{C_7C_4}{2C_2}+C_5.
	\end{equation*}
	Consider next $I_2(x)$. We have
	\begin{equation*}
	I_2(x)\le j(C_1|x|)\int_{B_{|x|}(0)}\varphi(y)dy.
	\end{equation*}
	Since $|x|>M$, we have by \eqref{thm7pass3},
	\begin{align*}
	I_2(x)&\le C_6|x|^{-d-\alpha}\well(|x|^2)\int_{B_{|x|}(0)}\varphi(y)dy\\
	&=C_6|x|^{-d-\alpha}\well(|x|^2)\left(\int_{B_M(0)}\varphi(y)dy+\int_{B_{|x|}(0)\setminus B_M(0)}\varphi(y)\right)\\
	&\le C_6|x|^{-d-\alpha}\well(|x|^2)\left(\Norm{\varphi}{\infty}M^d\omega_d+C_3d\omega_d\int_{M}^{|x|}
r^{d-2\kappa-1}dr\right).
	\end{align*}
	Due to $d-2\kappa-1\ge -1$ we have $\lim_{|x| \to \infty}\int_{M}^{|x|}r^{d-2\kappa-1}dr=\infty$, thus there
exists a constant $C_{11}>0$ such that
	\begin{equation*}\label{thmpass9}
	I_2(x)\le C_{11}|x|^{-d-\alpha}\well(|x|^2)\int_{M}^{|x|}r^{d-2\kappa-1}dr.
	\end{equation*}
	We distinguish two cases. If $\kappa \in \left(0,\frac{d}{2}\right)$, we have that $d-2\kappa-1>-1$ and
	$\int_{M}^{|x|}r^{d-2\kappa-1}dr \le\frac{1}{d-2\kappa}|x|^{d-2\kappa}$.
	Hence
	\begin{equation*}
	I_2(x)\le \frac{C_{11}}{d-2\kappa}|x|^{-2\kappa-\alpha}\well(|x|^2)
	\end{equation*}
	and thus
	\begin{equation*}
	\frac{1}{\varphi(x)}I_2(x)\le \frac{C_{11}}{(d-2\kappa)C_2}|x|^{-\alpha}\well(|x|^2).
	\end{equation*}
	Thus for this range of $\kappa$ we have by \eqref{thm7pass8}
	\begin{equation*}
	|V(x)|\le C_{12}|x|^{-\alpha}\well(|x|^2),
	\end{equation*}
	where $C_{12}=\frac{C_{11}}{(d-2\kappa)C_2}+C_{10}$, and $V(x)=O(|x|^{-\alpha}\well(|x|^2))$.
For the remaining case $\kappa=\frac{d}{2}$ we have, by taking $M>1$, that
	$\int_{M}^{|x|}r^{d-2\kappa-1}dr=\log|x|-\log M \le \log |x|$,
	thus given that $2\kappa=d$, again by \eqref{thm7pass8} we get
	\begin{equation*}
	\frac{1}{\varphi(x)}I_2(x)\le \frac{C_{11}}{C_2}|x|^{-\alpha}\well(|x|^2) \log|x|.
	\end{equation*}
	and thus
	\begin{equation*}
	|V(x)|\le\left(\frac{C_{11}}{C_2}\log|x| +C_{10}\right)|x|^{-\alpha}\well(|x|^2),
	\end{equation*}
	In particular, there exists a constant $C_{13}>0$ such that
	\begin{equation*}
	\frac{C_{11}}{C_2}\log |x|+C_{10}\le C_{13}\log|x|
	\end{equation*}
	for $|x|>M$, and thus we have
	$|V(x)|\le C_{13}|x|^{-\alpha}\log(|x|)\well(|x|^2)$.
\end{proof}
{\begin{rmk}
\label{rmkRes}
\rm{
Note that for $\kappa<\frac{d}{4}$ the eigenfunction $\varphi$ does not belong to $L^2(\R^d)$, thus it is a resonance.
}
\end{rmk}}
To apply the above results, we need to show that for some $C_1 \in (0,1)$ the eigenfunction $\varphi$ belongs to
$\cZ_{C_1}(\R^d)$ and the assumption on $L_\varphi$ is verified. We give the following more general criterion.
\begin{prop}\label{shapeeigen}
	Let $f(x)=\rho(|x|)$ with a real-valued function $\rho \in C^2(\R^+)$  such that for large enough $r$ the following hold:
	(1) $\rho$ is decreasing,
	(2) $\rho'(r)$ is increasing,
	(3) $\rho''(r)$ is decreasing, and
	(4) $\rho(r)\sim C_\rho r^{-2\kappa}$ for $\kappa > 0$.
	Then for every $C_1 \in (0,1)$ we have $f \in \cZ_{C_1}(\R^d)$ and
	\begin{equation*}
		L_f(x)\sim \frac{C_\rho 4\kappa(2\kappa+1)d^2}{(1-C_1)^{2\kappa+2}}|x|^{-2\kappa-2}, \quad |x| \to \infty.
	\end{equation*}
\end{prop}
\begin{proof}
For simplicity we assume that $\rho$, $\rho''$ are decreasing and $\rho'$ is increasing for every $r$, it is straightforward
to adapt the proof for large $r$. Since $\rho(r)-\rho(0)=\int_0^r\rho'(t)dt$ and $\rho'(r)-\rho'(0)=\int_0^r\rho''(t)dt$, by the
monotone density theorem
$$
\rho'(r)\sim - \frac{C_\rho 2\kappa}{r^{2\kappa+1}} \quad \mbox{and} \quad \rho''(r)\sim \frac{C_\rho 2\kappa(2\kappa+1)}
{r^{2\kappa+2}}.
$$
By elementary analysis
	\begin{align*}
		|\DDv |&=|\langle\nabla f(\xi_+),h\rangle-\langle \nabla f(\xi_-), h \rangle|\\
		&=|\langle \nabla f(\xi_+)-\nabla f(\xi_-),h\rangle|
		\le|\nabla f(\xi_+)-\nabla f(\xi_-)||h|,
	\end{align*}
	where $\xi_\pm \in \vv{[x \pm h,x]}$, with the same notation of a segment as before.
	For every $i=1,\dots,d$,
	\begin{align*}
		\Big|\pd{f}{x_i}(\xi_+)-\pd{f}{x_i}(\xi_-)\Big|
		&=\Big|\langle\nabla\pd{f}{x_i}(\xi_i), \xi_+-\xi_-\rangle\Big|
		\le \Big|\nabla\pd{f}{x_i}(\xi_i)\Big||\xi_+-\xi_-|.
	\end{align*}
	Since $|\xi_+-\xi_-|\le 2 |h|$ and $\vv{[\xi_+,\xi_-]} \subseteq \vv{[x+h,x-h]}$, we see that
	\begin{align*}
		|\nabla f(\xi_+)-\nabla f (\xi_-)|&=\left(\sum_{i=1}^{n}\left(\pd{f}{x_i}(\xi_+)-
		\pd{f}{x_i}(\xi_-)\right)^2\right)^{1/2}\\
		&\le\left(\sum_{i=1}^{n}\Big|\nabla\pd{f}{x_i}(\xi_i)\Big|^2|\xi_+-\xi_-|^2\right)^{1/2}
		\le 2|h| \left(\sum_{i=1}^{n}\Big|\nabla\pd{f}{x_i}(\xi_i)\Big|^2\right)^{1/2},
	\end{align*}
	and thus
	\begin{equation*}
		|\DDv |\le 2 \left(\sum_{i,j=1}^{d}\left(\frac{\partial^2 f}
		{\partial x_i \partial x_j}(\xi_i)\right)^2\right)^{1/2}|h|^2.
	\end{equation*}
	Furthermore, since for $i \not = j$
	\begin{equation*}
		\frac{\partial^2 f}{\partial x_i\partial x_j}(x)=-\frac{x_i x_j}{|x|^2}\rho'(|x|)+
		\frac{x_i x_j}{|x|^3}\rho''(|x|)=\frac{x_ix_j}{|x|^2}\left(\rho''(|x|)-\frac{\rho'(|x|)}{|x|}\right),
	\end{equation*}
	and given that $\rho$ is decreasing, i.e., $\rho' \le 0$, we get
	\begin{equation*}
		\Big|\frac{\partial^2 f}{\partial x_i\partial x_j}(\xi_i)\Big|\le \rho''((1-C_1)|x|),
	\end{equation*}
	as $\xi_i \in \vv{[x+h,x-h]} \subset B_{(1-C_1)|x|}^c(0)$ and $\rho''$ is decreasing.
	Next consider $i=j$. We have
	\begin{align*}
		\frac{\partial^2 f}{\partial x_i^2}(x)
		&=\frac{|x|^2-x_i^2}{|x|^3}\rho'(|x|)+\frac{x_i^2}{|x|^2}\rho''(|x|).
	\end{align*}
	Note that $|x|^2-x_i^2\ge 0$ and $\rho' \le 0$. Using again that $\rho''$ is decreasing and $\xi_i \in
	B_{(1-C_1)|x|}^c(0)$, we get
	\begin{equation*}
		\Big|\frac{\partial^2 f}{\partial x_i^2}(\xi_i)\Big|\le \rho''((1-C_1)|x|).
	\end{equation*}
	Combining the above estimates, we finally obtain
	\begin{equation*}
		|\DDv |\le 2 d^2\rho''((1-C_1)|x|)|h|^2.
	\end{equation*}
\end{proof}
\begin{rmk}
\rm{Theorem \ref{thm6} remains valid if $\varphi \in \cZ_{C_1}^\beta(\R^d)$ for $\beta \in L^1_{\rm rad}(\R^d,\nu)$
such that $\beta(r)\le C_\beta r^{\gamma}$ for some $\gamma \in (0,2]$ and $r>M_\beta$, and $L_\varphi(x)\le C_L
|x|^{-(2\kappa+\gamma)}$ for $|x|>M_\beta$, with appropriate constants $C_\beta, C_L, M_\beta$.}
\end{rmk}

\subsection{Decay rates for exponentially light L\'evy intensities}
\begin{thm}\label{rateexp}
Let Assumption \ref{eveq} hold with $\varphi \in \cZ_{C_1}(\R^d)$ for some $C_1 \in (0,1)$. Suppose that there exist
constants $C_\mu,\eta>0$ and $\alpha \in (0,2]$ such that $\mu(t) \sim C_\mu t^{-1-\frac{\alpha}{2}}e^{-\eta t}$ as
$t \to \infty$. If, furthermore, there exist
\begin{enumerate}
	\item $C_2,C_3,M_1>0$ and $\kappa \in \left(0,\frac{d+\alpha}{2}\right)$ such that
	$C_2 |x|^{-2\kappa}\le \varphi(x)\le C_3 |x|^{-2\kappa}$ for $x \in B_{M_1}^{c}(0)$, \vspace{0.1cm}
	\item $M_2,C_4>0$ such that $L_\varphi(x)\le C_4|x|^{-(2\kappa+2)}$ for  $x \in B_{M_2}^c(0)$,
\end{enumerate}
then $V(x)=O(|x|^{-2})$.
\end{thm}
\begin{proof}
Since the assumptions of Theorem \ref{thm3} are verified, for $x \in B_{M_\varphi}^c(0)$ we can write
\begin{equation}\label{est0}
|V(x)|\le \frac{1}{2}\frac{L_\varphi(x)}{\varphi(x)}\cJ(C_1|x|)+
2\Norm{\varphi}{\infty}\frac{\nu(B_{C_1|x|}^c(0))}{\varphi(x)}.
\end{equation}
By Corollary \ref{corexp} there exist $C_5,C_6>0$ and $M_3>0$ such that for every $x \in B_{M_3}^c(0)$
\begin{equation}\label{est1}
\nu(B_{C_1|x|}^c(0))\le C_5|x|^{\frac{d-\alpha-4}{2}}e^{-\sqrt{\eta}C_1|x|} \quad \text{and} \quad \cJ(C_1|x|)\le C_6.
\end{equation}
Fix $M=\max_{i=0,\dots,3}M_i$, where $M_0=M_{\varphi}$. Thus we have for $|x|>M$,
\begin{equation}\label{est2}
\varphi(x) \ge C_2|x|^{-2\kappa} \quad \text{and} \quad L_\varphi(x)\le C_4|x|^{-(2\kappa+2)}.
\end{equation}
Using \eqref{est1}-\eqref{est2} in the estimate \eqref{est0} we obtain
\begin{equation*}
|V(x)|\le \frac{C_4C_6}{2C_2}|x|^{-2}+2\Norm{\varphi}{\infty}\frac{C_5}{C_2}
|x|^{\frac{d-\alpha-4+4\kappa}{2}}e^{-\sqrt{\eta}C_1|x|}
\end{equation*}
which implies that there exists a constant $C_7>0$ such that $|V(x)|\le C_{7}|x|^{-2}$, for every $x \in B_M^c(0)$.

	\end{proof}
\begin{rmk}
\label{shortrange1}
\rm{
The assumptions of Theorem \ref{rateexp} are verified, for instance, for the massive relativistic Schr\"odinger operator
$H=L_{m,\alpha}+V$ and for radial eigenfunctions that satisfy the conditions of Proposition \ref{shapeeigen}. Moreover,
Theorem \ref{rateexp} continues to hold if $\varphi \in \cZ_{C_1}^\beta(\R^d)$ for some $\beta \in L^1_{\rm rad}(\R^d,\nu)$.
Also, we note that if $L_\varphi(x)\le C_4|x|^{-(2\kappa+\gamma)}$ for an exponent $\gamma>0$, then $V(x) = O(|x|^{-\gamma})$.
Also, as in Remark \ref{rmkRes}, if $\kappa<d/4$, then $\varphi$ is a resonance as it does not belong to $L^2(\R^d)$.
}
\end{rmk}

Now we introduce a measure of excess between the positive and negative parts in the \emph{local component} of the estimates,
i.e., in a neighbourhood of zero. Under a condition on this excess we can derive also a lower bound.
\begin{thm}\label{rateexp2}
Let Assumption \ref{eveq} hold with $\varphi \in \cZ_{C_1}(\R^d)$ for some $C_1 \in (0,1)$, and suppose that there exist
$C_\mu,\eta>0$ and $\alpha \in (0,2]$ such that $\mu(t) \sim C_\mu t^{-1-\frac{\alpha}{2}}e^{-\eta t}$ as $t \to \infty$.
Furthermore, assume that
\begin{enumerate}
\item
there exist $C_2,C_3,M_1>0$ and $\kappa \in \left(0,\frac{d+\alpha}{2}\right)$ such that
$C_2 |x|^{-2\kappa}\le \varphi(x)\le C_3 |x|^{-2\kappa}$ for $x \in B_{M_1}^{c}(0)$;
\vspace{0.1cm}
\item
$L_\varphi(x)\le C_4|x|^{-(2\kappa+2)}$ with constant $C_4>0$, for every $x \in B_{M_\varphi}^c(0)$;
\vspace{0.1cm}
\item
there exist a function $f:\R^d \to \R$ and constants $M_2>0$ and $\omega \ge 2$ such that
\begin{equation*}
|\Dv|\ge f(x)|h|^\omega, \quad h \in B_{C_1|x|}(0), \; x \in B_{M_2}^c(0);
\end{equation*}
\item
$f(x) \ge C_5|x|^{-(2\kappa+2)}$ with $C_5>0$, for every $x \in B_{M_2}^c(0)$.
\end{enumerate}
Define
\begin{align*}
	H_L^{\pm} & :=C_4\lim_{|x| \to \infty}\int_{B_{C_1|x|}(0)}|h|^2 j(|h|)1_{\{\pm\Dv \ge 0\}}dh \\
	H_f^{\pm} & :=C_5\lim_{|x| \to \infty}\int_{B_{C_1|x|}(0)}|h|^\omega j(|h|)1_{\{\pm\Dv \ge 0\}}dh,
\end{align*}
whenever they exist. If $\max\{H_f^+-H_L^-,H_f^--H_L^+\}>0$, then there exists $M >0$ such that
\begin{equation*}
|V(x)| \asymp \frac{1}{|x|^{2}}, \quad x \in B_M^c(0).
\end{equation*}
\end{thm}
\begin{proof}
By Theorem \ref{rateexp} we already know that there exist $C_V^+,M^+ > 0$ such that $|V(x)|\le C_V^+|x|^{-2}$
for every $x \in B_{M^+}^c(0)$. Thus we only need to show the lower bound. We have
	\begin{align}\label{estV}
	\begin{split}
	|V(x)|&\ge \frac{1}{2\varphi(x)}\Big|\int_{B_{C_1|x|}(0)}\Dv j(|h|)dh\Big|
-\left(\frac{1}{\varphi(x)}\int_{B_{C_1|x|}}\varphi(x+h)j(|h|)dh+\nu(B_{C_1|x|}^c(0))\right)\\
	&\ge \frac{1}{2\varphi(x)}\Big|\int_{B_{C_1|x|}(0)}\Dv j(|h|)dh\Big|
-2\Norm{\varphi}{\infty}\frac{\nu(B_{C_1|x|}^c(0))}{\varphi(x)}
	\end{split}
	\end{align}
By Corollary \ref{corexp} and assumption (1) above there exist $C_6,M_3>0$ such that
	\begin{equation*}
		C_2|x|^{-2\kappa} \le \varphi(x)\le C_3|x|^{-2\kappa} \quad \text{and}
\quad \nu(B_{C_1|x|}^c(0))\le C_6|x|^{\frac{d-\alpha-4}{2}}e^{-\sqrt{\eta}C_1|x|},
	\end{equation*}
for every $x \in B_{M_3}^c(0)$, thus we have
\begin{equation}\label{somest}
	|V(x)|\ge \frac{1}{2C_3}|x|^{2\kappa}\Big|\int_{B_{C_1|x|}(0)}\Dv j(|h|)dh\Big|-
\frac{2\Norm{\varphi}{\infty}C_6}{C_2}|x|^{\frac{d-\alpha-4+4\kappa}{2}}e^{-\sqrt{\eta}C_1|x|}.
\end{equation}
On the other hand, with obvious notations of positive and negative parts, writing for a shorthand $I_k^+ =
\int_{B_{C_1|x|}(0)}|h|^k j(|h|)1_{\{\Dv \ge 0\}}dh$ and $I_k^- = \int_{B_{C_1|x|}(0)}|h|^k j(|h|)
1_{\{\Dv \le 0\}}dh$ for $k \in \{\omega, 2\}$, and denoting $M_4=\max\{M_\varphi,M_2\}$, for $x \in B_{M_4}^c(0)$
we obtain

\begin{align}\label{estabsval}
	\begin{split}
\Big|\int_{B_{C_1|x|}(0)}\Dv j(|h|)dh\Big|
& \ge
\max\left\{\int_{B_{C_1|x|}(0)}(\Dv )^+j(|h|)dh-
\int_{B_{C_1|x|}(0)}(\Dv )^-j(|h|)dh,\right.\\
&\qquad \qquad \left.\int_{B_{C_1|x|}(0)}(\Dv )^-j(|h|)dh-\int_{B_{C_1|x|}(0)}(\Dv )^+j(|h|)dh\right\}\\
& \qquad \ge
\max\left\{I^+_\omega f(x) - I^-_2 L_\varphi(x), I^-_\omega f(x) - I^+_2 L_\varphi(x) \right\}\\
&\qquad \ge
|x|^{-(2\kappa+2)}\max\left\{C_5I^+_\omega-C_4I^-_2,C_5I^-_\omega -C_4I^+_2\right\}.
\end{split}
\end{align}
Applying this estimate to \eqref{somest} and multiplying by $|x|^{2}$ we obtain
\begin{align*}
|V(x)||x|^2
&\ge
\frac{1}{2C_3} \max\left\{C_5I^+_\omega-C_4I^-_2,C_5I^-_\omega -C_4I^+_2\right\}
-\frac{2\Norm{\varphi}{\infty}C_6}{C_2}|x|^{\frac{d-\alpha+4\kappa}{2}}e^{-\sqrt{\eta}C_1|x|}.
\end{align*}
Setting $H^\pm =\max\{H_f^+-H_L^-,H_f^--H_L^+\}>0$ we get $\liminf_{|x| \to \infty}|V(x)||x|^2\ge H^\pm >0$. Thus
there exist constants $C_V^-,M^->0$ such that $|V(x)||x|^2\ge C_V^-$, for every $x \in B_{M^-}^c(0)$. Taking $M
=\max\{M^-,M^+\}$  completes the proof.
\end{proof}

\newpage
\begin{rmk}
\label{shortrange2}
\rm{
\hspace{100cm}
\begin{enumerate}
\item
For eigenfunctions satisfying the assumptions of Proposition \ref{shapeeigen}, we have some sufficient conditions
for the additional assumptions (3)-(4) of Theorem \ref{rateexp2}. Indeed, for $\omega=2$ both are implied if the
Hessian matrix $D^2\varphi(x)$ is positive definite for $|x|$ large enough and $\lambda(x)\sim |x|^{-(2\kappa+2)}$
as $|x| \to \infty$, where $\lambda(x)=\min_{z \in B_{(1-C_1)|x|}^c(0)\cap B_{(1+C_1)|x|}(0)}\lambda_{\rm min}(z)$
and $\lambda_{\rm min}(x)$ is the lowest eigenvalue of $D^2\varphi(x)$.

\item
A sufficient condition to guarantee that $\int_{B_{C_1}|x|}\Dv j(|h|)dh>0$ is the following:  (1) there exists $R>0$
such that $\varphi$ is convex in $B_R^c(0)$, (2) for every $M>R$ there exist $x \in B_M^c(0)$ and $h \in B_{C_1|x|}(0)$
such that $\Dv \not = 0$. Indeed, observe that if $\varphi$ is convex in $B_R^c(0)$, and we fix $x \in B_R^c(0)$ such
that $(1-C_1)|x|>R$ and $v \in \partial B_1(0)$, the function $(-C_1|x|,C_1|x|) \ni t \mapsto g_{x,v}(t)=\varphi(x+tv)$
is convex. In particular, we may choose $v=\frac{h}{|h|}$ for $0 \neq h \in B_{C_1|x|}(0)$ and use the fact that ratios
of increments of convex functions are increasing, giving
\begin{equation*}
\frac{g_{x,v}(|h|)-g_{x,v}(0)}{|h|}-\frac{g_{x,v}(0)-g_{x,v}(-|h|)}{|h|} \ge 0,
\end{equation*}
which on multiplication by $|h|>0$ implies $\Dv \ge 0$. The second condition is needed to ensure that $\int_{B_{C_1}|x|}
\Dv j(|h|)dh\not =0$.
\item
Moreover, Theorem \ref{rateexp2} holds also if $\varphi \in \cZ_{C_1}^{\beta_1}(\R^d)$ for some
$\beta_1 \in L^1_{\rm rad}(\R^d,\nu)$, and
\begin{equation*}
	|\Dv|\ge f(x)\beta_2(|h|).
\end{equation*}
This is possible since with $\beta_2 \le \beta_1$ we have $\beta_2 \in L^1_{\rm rad}(\R^d,\nu)$), and by a suitable
change of the definitions of $H_f^{\pm}$ and $H_L^{\pm}$.
\item
Also, Theorem \ref{rateexp2} remains true if we replace $2\kappa+2$ in assumptions (2) and (4) by $2\kappa+\gamma$,
$\gamma > 0$. In this case we obtain
\begin{equation*}
|V(x)| \asymp \frac{1}{|x|^{\gamma}},
\end{equation*}
i.e., $V$ can be any polynomially short range.
\item
If $\Dv \ge 0$ for sufficiently large $|x|$ and $h \in B_{C_1|x|}^c(0)$, then $f$ and $L_\varphi$ need not have the
same order of decay for a meaningful result. Indeed, if we replace $2\kappa+2$ only in assumption (4) by $2\kappa+\gamma$,
then using $\Dv \ge 0$ we can conclude that
	\begin{equation*}
	C_-|x|^{-\gamma}\le|V(x)|\le C_+|x|^{-2}
	\end{equation*}
	with $\gamma \ge 2$.
We also note that if $\varphi \in \cZ_{C_1}^{\beta_1}(\R^d)$ for some $\beta_1 \in L^1_{\rm rad}(\R^d,\nu)$, moreover
$\Dv\ge f(x)\beta_1(|h|)$, $L_\varphi(x)\sim |x|^{-(2\kappa+\gamma_1)}$ and $f(x) \sim |x|^{-(2\kappa+\gamma_2)}$, then
$\gamma_2 \ge \gamma_1$, and thus
\begin{equation*}
C_-|x|^{-\gamma_2}\le|V(x)|\le C_+|x|^{-\gamma_1}.
\end{equation*}
\end{enumerate}

}
\end{rmk}
\begin{cor}\label{corsign1}
Let the assumptions of Theorem \ref{rateexp2} hold, and suppose $\varphi \in C(B^c_R(0))$ for $R>M$. Then no sign
change of $V$ occurs in $B_R^c(0)$.
\end{cor}
\begin{proof}
By Theorems \ref{rateexp}-\ref{rateexp2} we know that
$|V(x)| \asymp |x|^{-2}$ for $|x|>M$.
If $\varphi \in \cZ_{C_1}(\R^d)$, then $\varphi \in \cZ_{\rm b, loc}(B_R^c(0))$. Thus by Theorem \ref{thm2}
we have $V \in C(B_R^c(0))$, and then
the two-sided estimate implies that $V(x)\not = 0$ for every $x \in B_R^c(0)$.
\end{proof}
We have seen from Theorem \ref{rateexp} that in the case of an exponentially light L\'evy intensity the local part
in the representation of the potential plays a crucial role due to the joint action of the exponentially light
tails of $\nu$ outside balls and the boundedness of $\cJ$. The rate of decay is, as shown in Theorem \ref{rateexp},
completely determined by the ratio between a zero-energy eigenfunction $\varphi \in \cZ_{C}(\R^d)$ and $L_\varphi$
which, whenever $\varphi \in C^2(\R^d)$, is proportional to any norm of its Hessian $D^2\varphi$. One may ask in
what circumstances does the non-balancing condition $H^\pm>0$ hold for a specific class of functions. The following
result provides an answer.
\begin{prop}\label{propDv}
Let $f \in C^2(\R^d) \cap L^\infty(\R^d)$ with positive definite Hessian $D^2f(x)$ for $x \in B_R^c(0)$, for some
$R>0$. Then $f \in \cZ_{C}(\R^d)$ for every $C \in (0,1)$, and there exists $M(C)>0$ such that for every $x \in
B_{M(C)}^c(0)$, we have $\DDv>0$ for all $j \in B_{C|x|}^c(0)$.
\end{prop}
\begin{proof}
Fix $C \in (0,1)$ and consider the set $A_{C}(x)=\overline{B_{(1-C)|x|}^c(0)\cap B_{(1+C)|x|}(0)}$. Fix $h \in
B_{C|x|}(0)$. Then
\begin{equation*}
(1-C)|x| \le |x|-|h|\le |x+h| \le |x|+|h| \le (1+C)|x|
\end{equation*}
and similarly for $x-h$. Thus we have that $x,x+h,x-h \in A_{C}(x)$. Consider the Taylor expansion of $f$ with
remainders in $x\pm h$
\begin{align*}
f(x+h)&=f(x)+\langle \nabla f(x), h \rangle + \frac{1}{2}\langle D^2f(\xi_1(h))h,h\rangle\\
f(x-h)&=f(x)-\langle \nabla f(x), h \rangle + \frac{1}{2}\langle D^2f(\xi_2(h))h,h\rangle
\end{align*}
where $\xi_1(h),\xi_2(h)\in A_{C}(x)$. Adding the two relations we obtain
\begin{equation}\label{eqDDv}
\DDv=\frac{1}{2}\langle (D^2f(\xi_1(h))+D^2f(\xi_2(h)))h,h\rangle
\end{equation}
Consider now $\lambda_{\rm max}(x)$ to be the highest eigenvalue of $D^2f(x)$ and $\lambda(x)=
\max_{z \in A_{C}(x)}\lambda_{\rm max}(z)$. If $|x|>\frac{R}{1-C}$, we have that $A_{C}(x) \subset
B_R^c(0)$, thus $D^2f(x)$ is positive definite in $A_{C}(x)$ and $\lambda(x)>0$. Thus we easily
have
\begin{equation*}
|\DDv|\le\lambda(x)|h|^2
\end{equation*}
and then $f \in \cZ_{C}(\R^d)$. Moreover, since $D^2f(x)$ is positive definite in $A_{C}(x)$, by \eqref{eqDDv}
we have that $\DDv>0$ for every $x \in B_{M(C)}^c(0)$, where $M(C)=R/(1-C)$, and $h \in B_{C|x|}(0)$.
\end{proof}
\begin{rmk}\label{rmknonoB}
\rm{
Notice that if $\Dv>0$ or $\Dv<0$ for large enough $|x|$, the excess condition $H>0$ in Theorem
\ref{rateexp2} is unnecessary. Indeed, then we can simply estimate like
\begin{equation*}
\Big|\int_{B_{C_1|x|}(0)}\Dv j(|h|)dh\Big|=\int_{B_{C_1|x|}(0)}|\Dv| j(|h|)dh\ge f(x)\int_{B_{C_1|x|}(0)}
|h|^\omega j(|h|)dh,
\end{equation*}
to complete the proof.
}
\end{rmk}
In particular, from these observations we obtain the following result.
\begin{cor}\label{corC21}
Let Assumption \ref{eveq} be satisfied with $\varphi \in C^2(\R^d)$ and $\mu(t)\sim C_\mu
t^{-1-\frac{\alpha}{2}}e^{-\eta t}$ as $t \to \infty$ for some $\eta,C_\mu>0$ and $\alpha \in (0,2]$. Furthermore,
suppose that also the following hold:
\begin{enumerate}
\item
The assumptions of Proposition \ref{shapeeigen} hold.
\item
$D^2\varphi(x)$ is positive definite for $x \in B_R^c(0)$ and some $R>0$.
\item
Let $A_C(x)$ be defined as in Proposition \ref{propDv} and $\lambda_{\rm min}(x)$ be the lowest eigenvalue
of $D^2\varphi(x)$. Consider $\lambda(x)=\min_{z \in A_C(x)}\lambda_{\rm min}(z)$, and assume that $\lambda(x)
\ge C_\lambda |x|^{-(2\kappa+2)}$, for $x \in B_R^c(0)$.
\end{enumerate}
Then there exists $M_V >0$ such that
\begin{equation*}
|V(x)| \asymp \frac{1}{|x|^{2}}, \quad x \in B_{M_V}^c(0).
\end{equation*}
\end{cor}
\begin{proof}
By Proposition \ref{propDv} we know that $\varphi \in \cZ_{C}(\R^d)$. Moreover, combining Proposition
\ref{shapeeigen} together with the assumptions above we see that the hypotheses of Theorem \ref{rateexp2}
are verified, except possibly the excess condition $H^\pm>0$. Finally, by Proposition \ref{propDv} we
know that $\Dv>0$ thus, by Remark \ref{rmknonoB} it follows that the claim is true even without the
excess condition $H^\pm>0$.
\end{proof}
For exponentially light L\'evy intensities we have another interesting case in which the potential has a
decay.
\begin{thm}\label{expdecaype}
Let Assumption \ref{eveq} hold with $\varphi \in \cZ_{C_1}(\R^d)$ for some $C_1 \in (0,1)$, and $\mu(t)\sim C_\mu
t^{-1-\frac{\alpha}{2}}e^{-\eta t}$ as $t \to \infty$ for some $\eta,C_\mu>0$ and $\alpha \in (0,2]$. If there
exist
\begin{enumerate}
\item
$C_2,C_3,M_1,\eta_\varphi>0$, $\gamma \in (0,1)$ and $\delta \in \R$ such that
\begin{equation*}
C_2|x|^{\delta}e^{-\eta_\varphi |x|^\gamma}\le \varphi(x)\le C_3|x|^{\delta}e^{-\eta_\varphi|x|^{\gamma}},
\quad  x \in B_{M_1}^c(0);
\end{equation*}
\item
$C_4,M_3>0$ such that
\begin{equation*}
L_{\varphi}(x)\le C_4|x|^{\delta+2(\gamma-1)}e^{-\eta_\varphi|x|^\gamma}, \quad x \in B_{M_3}^c(0),
\end{equation*}
\end{enumerate}
then $V(x)=O(|x|^{-2(1-\gamma)})$.
\end{thm}
\begin{proof}
Recall that $\cJ(|x|)$ is uniformly bounded by a constant $C_5>0$, see Corollary \ref{corexp}.
Moreover, by the same corollary there exist $C_6,M_4>0$ such that
	\begin{equation*}
		\nu(B^c_{C_1|x|}(0))\le C_6 |x|^{\frac{d-\alpha-4}{2}}e^{-\sqrt{\eta}C_1|x|}, \quad x \in B_{M_4}^c(0).
	\end{equation*}
	Take $M=\max_{i=0,\dots,4}M_i$, where $M_0=M_{\varphi}$ and $|x|>M$. We get
	\begin{align}
	\begin{split}\label{estgamma}
		|V(x)|&\le \frac{1}{2\varphi(x)}\int_{B_{C_1|x|}(0)}|\Dv |j(|h|)dh
		+2\Norm{\varphi}{\infty}\frac{\nu(B_{C_1|x|}^c(0))}{\varphi(x)}\\
		&\le \frac{C_5L_\varphi(x)}{2\varphi(x)}+2\Norm{\varphi}{\infty}\frac{\nu(B_{C|x|}^c(0))}{\varphi(x)}\\
		&\le C_7|x|^{2(\gamma-1)} + C_8|x|^{-\delta-1-\alpha}e^{\eta_\varphi|x|^\gamma-C_1\sqrt{\eta}|x|},
	\end{split}
	\end{align}
	Since $\gamma \in (0,1)$, we thus have $|V(x)|\le C_8|x|^{2(\gamma-1)}$.
\end{proof}
\begin{thm}\label{rateexp3}
Let Assumption \ref{eveq} hold with $\varphi \in \cZ_{C_1}(\R^d)$ for some $C_1 \in (0,1)$, and
$\mu(t)\sim C_{\mu}t^{-1-\frac{\alpha}{2}}e^{-\eta_\mu t}$ as $t \to \infty$ for some $\eta_\mu,
C_\mu>0$ and $\alpha \in (0,2]$. Suppose, furthermore, that
\begin{enumerate}
\item
there exist $C_2,C_3,M_1,\eta_\varphi>0$, $\gamma \in (0,1)$ and $\delta \in \R$ such that
\begin{equation*}
C_2|x|^\delta e^{-\eta_\varphi |x|^{\gamma}}\le \varphi(x)\le |x|^{\delta}C_3e^{-\eta_\varphi|x|^\gamma},
\quad  x \in B_{M_1}^c(0);
\end{equation*}
\item
$L_\varphi(x)\le C_4 |x|^{\delta+2(\gamma-1)}e^{-\eta_\varphi|x|^\gamma}$ for $C_4>0$ and every $x
\in B_{M_\varphi}^c(0)$;
\item
there exist a function $f:\R^d \to \R$ and constants $M_2>0$ and $\omega \ge 2$ such that
\begin{equation*}
|\Dv|\ge f(x)|h|^\omega, \quad h \in B_{C_1|x|}(0), \; x \in B_{M_2}^c(0);
\end{equation*}
\item
$f(x)\ge C_5|x|^{\delta+2(\gamma-1)}e^{-\eta_\varphi|x|^\gamma}$ for $C_5>0$ and every $x \in B_{M_2}^c(0)$.
\end{enumerate}
Define $H_L^{\pm}, H_f^{\pm}$ and $H^\pm = \max\{H_f^+-H_L^-,H_f^--H_L^+\}$ as in Theorem \ref{rateexp2},
whenever they exist.
If the excess condition $H^\pm >0$ holds, then there exists $M>0$ such that
$$
|V(x)| \asymp |x|^{-2(1-\gamma)}, \quad x \in B_M^c(0).
$$
\end{thm}
\begin{proof}
By Theorem \ref{expdecaype} there exist $C_V^+,M^+$ such that $|V(x)||x|^{2(1-\gamma)}\le C_V^+$ for all
$x \in B_{M^+}^c(0)$. To obtain the the lower bound recall estimate \eqref{estV}. By Corollary \ref{corexp}
and assumption (1) it follows that there exist $C_6,M_3>0$ such that
\begin{equation*}
C_2|x|^\delta e^{-\eta_\varphi|x|^\gamma}\le \varphi(x)\le C_3|x|^\delta e^{-\eta_\varphi|x|^\gamma} \quad
\mbox{and} \quad \nu(B_{C_1|x|}^c(0))\le C_6|x|^{\frac{d-\alpha-4}{2}}e^{-C_1\sqrt{\eta_\mu}|x|},
\end{equation*}
for every $x \in B_{M_3}^c(0)$. Hence by \eqref{estV} we have
\begin{equation}\label{estV3}
|V(x)|\ge \frac{1}{2C_3}|x|^{-\delta}e^{\eta_\varphi|x|^\gamma}\left|\int_{B_{C_1|x|}(0)}\Dv\varphi(x)j(|h|)dh\right|
-\frac{2C_6\Norm{\varphi}{\infty}}{C_2}|x|^{\frac{d-\alpha-4-2\delta}{2}}e^{\eta_\varphi|x|^\gamma-C_1\sqrt{\eta_\mu}|x|}
\end{equation}
On the other hand, we can use again an estimate of the type \eqref{estabsval}. As before, using the shorthand notations
$I_k^+ = \int_{B_{C_1|x|}(0)}|h|^k j(|h|)1_{\{\Dv \ge 0\}}dh$ and $I_k^- = \int_{B_{C_1|x|}(0)}
|h|^k j(|h|)1_{\{\Dv \le 0\}}dh$ for $k \in \{\omega, 2\}$, we have
\begin{align*}
\Big|\int_{B_{C_1|x|}(0)}\Dv j(|h|)dh\Big|
\ge |x|^{\delta+2(\gamma-1)} e^{-\eta_\varphi|x|^\gamma}\max \left\{C_5 I^+_\omega
-C_4 I^-_2, C_5I^-_\omega -C_4 I_2^+ \right\}.
\end{align*}
Combining the latter estimate with \eqref{estV3} and multiplying by $|x|^{2(1-\gamma)}$ we obtain
\begin{align*}
|V(x)||x|^{2(1-\gamma)}
&\ge
\frac{1}{2C_3} \max \left\{C_5 I^+_\omega -C_4 I^-_2, C_5I^-_\omega -C_4 I_2^+ \right\} \\
& \qquad -\frac{2C_6\Norm{\varphi}{\infty}}{C_2}|x|^{\frac{d-\alpha-4-2\delta+4(1-\gamma)}{2}}
e^{\eta_\varphi|x|^\gamma-C_1\sqrt{\eta_\mu}|x|}.
\end{align*}
Using that $\gamma<1$, we obtain
\begin{equation*}
\liminf_{|x| \to \infty}|V(x)||x|^{2(1-\gamma)}\ge \frac{H^\pm}{2C_3}>0
\end{equation*}
thus there exist $C_V^-,M^->0$ such that $|V(x)||x|^{2(1-\gamma)}\ge C_V^-$ for all $x \in B_{M^-}^c(0)$.
Setting $M=\max\{M^-,M^+\}$ completes the proof.
\end{proof}
\begin{rmk}
\rm{
Similarly to Theorem \ref{rateexp2}, in Theorem \ref{rateexp3} we do not need to require the excess condition
$H^\pm >0 $ if $\Dv$ has a definite sign. Furthermore, Theorems \ref{expdecaype} and \ref{rateexp3} continue to
hold if $\varphi \in \cZ_{C_1}^\beta(\R^d)$ for some $\beta \in L^1_{\rm rad}(\R^d)$ such that $\beta(r)\le C
r^\omega$ for large enough $r$, with suitable constants $C,\omega>0$, and (for Theorem \ref{rateexp3}) there is
a modulus of continuity $\beta_f(r)\le \beta(r)$, $\beta_f \in L^1_{\rm rad}(\R^d)$, such that $|\Dv|
\ge f(x) \beta_f(h)$.}
\end{rmk}

The following consequences are counterparts of Corollaries \ref{corsign1}-\ref{corC21}. Since the proofs are
similar, we skip them.
\begin{cor}
\label{corsign2}
Let the assumptions of Theorem \ref{rateexp3} hold and $\varphi \in C(B_R^c(0))$. Then $V$ does not change
sign in $B_R^c(0)$.
\end{cor}
\begin{cor}
Let Assumption \ref{eveq} hold with $\varphi \in C^2(\R^d)$, and $\mu(t)\sim C_\mu t^{-1-\frac{\alpha}{2}}
e^{-\eta_\mu t}$ as $t \to \infty$ for some $\eta_\mu,C_\mu>0$ and $\alpha \in (0,2]$. Furthermore, suppose
that the following properties hold:
\begin{enumerate}
\item
There exist $M_1,\eta_\varphi>0$, $\delta \in \R$ and $\gamma \in (0,1)$ such that
\begin{equation*}
\varphi(x) \asymp |x|^{\delta}C_2e^{-\eta_\varphi|x|^\gamma}, \quad  x \in B_{M_1}^c(0).
\end{equation*}
\item
$D^2\varphi(x)$ is positive definite for $x \in B_R^c(0)$ and some $R>0$.
\item
Let $A_{C}(x)$ be defined as in Proposition \ref{propDv} for any $C\in (0,1)$ and $\lambda_{\rm min}(x)$ and
$\lambda_{\rm max}(x)$ be the lowest and highest eigenvalues of $D^2\varphi(x)$. Consider $\lambda^-(x)=
\min_{z \in A_C(x)}\lambda_{\rm min}(z)$ and $\lambda^+(x)=\max_{z \in A_C(x)}\lambda_{\rm max}(z)$, and assume
that
$C_\lambda^- |x|^{-\delta+2(\gamma-1)}e^{-\eta_\varphi|x|^\gamma}\le \lambda^-(x)\le \lambda^+(x)
\le C_\lambda^- |x|^{-\delta+2(\gamma-1)}e^{-\eta_\varphi|x|^\gamma}$, for $x \in B_R^c(0)$.
\end{enumerate}
Then there exists $M >0$ such that
\begin{equation*}
|V(x)| \asymp \frac{1}{|x|^{2(1-\gamma)}}, \quad x \in B_{M}^c(0).
\end{equation*}
\end{cor}

Consider the massive relativistic operator $L_{m,\alpha}$ as it was defined in Section \ref{maasivemassless},
and the expression \eqref{withG}. By \eqref{sigmasymp} and Proposition \ref{corSigma} we see that the effect of
the massive part $j_{m,\alpha}$ vanishes asymptotically, and we can easily show that Theorems \ref{thm3}-\ref{thm6}
continue to hold also for the quantity $\frac{1}{\varphi(x)}G_{m,\alpha}\varphi(x)$ instead of $V$. In particular,
we get the following counterpart of Theorem \ref{thm6}.
\begin{prop}
Let $f \in \cZ_{C_1}(\R^d)$ and suppose that there exist
\begin{enumerate}
\item
$M>0$ and $\kappa \in \left(0,\frac{d+\alpha}{2}\right)$ such that $f(x) \asymp |x|^{-2\kappa}$ for $x \in B_{M}^c(0)$;

\vspace{0.1cm}
\item
$C_2>0$ such that $g(x)\le C_2|x|^{-(2\kappa+2)}$ for all $x \in B_{M_f}^c(0)$.
\end{enumerate}
If $d+\alpha>2$, then
\begin{equation*}
\frac{1}{f(x)} \, G_{m,\alpha}f(x)=\begin{cases} O(|x|^{2\kappa-d-\alpha})
& \kappa \in \left(\frac{d}{2},\frac{d+\alpha}{2}\right) \vspace{0.1cm}\\
O(|x|^{-\alpha}\log(|x|)) & \kappa=\frac{d}{2} \vspace{0.1cm} \\
O(|x|^{-\alpha}) & \kappa \in \left(0,\frac{d}{2}\right).
\end{cases}
\end{equation*}
\end{prop}
\noindent
As for Theorem \ref{thm6}, the proof in the case $\kappa \in \left[\frac{d}{2},\frac{d+\alpha}{2}\right)$ is
made by exploiting the fact that
\begin{equation*}\label{formulaG}
\left|\frac{1}{f(x)}G_{m,\alpha}f(x)\right| \le \frac{1}{f(x)}\int_{B_{C_1|x|}^c(0)}f(y)
\sigma_{m,\alpha}(|x-y|)dy+C|x|^{-2\kappa-\alpha}
\end{equation*}
for a suitable constant $C>0$. Since both terms $\frac{1}{\varphi(x)}L_{0,\alpha}\varphi(x)$ and
$\frac{1}{\varphi(x)}G_{m,\alpha}\varphi(x)$ have the same asymptotic behaviour, this reveals another aspect
of the operator which can be seen from the representation \eqref{withG}. We may expect that some cancellations
occur, and this indeed happens. For instance, from the above we know that if $\varphi$ is polynomially bounded
with exponent $-2\kappa$ with $\kappa<\frac{d}{2}$, then both terms are $O(|x|^{-\alpha})$, while by Theorem
\ref{rateexp} the potential $V$ is $O(|x|^{-2})$ and decays faster. The same happens if $\varphi$ is bounded
by $|x|^\delta e^{-\eta_\varphi|x|^\gamma}$ with $0<\gamma<1$. Indeed, both $\frac{1}{\varphi(x)}L_{0,\alpha}
\varphi(x)$ and $\frac{1}{\varphi(x)}G_{m,\alpha}\varphi(x)$ are even exploding at infinity, while we see by
Theorem \ref{expdecaype} that $V$ is at least polynomially decaying like $|x|^{-2(1-\gamma)}$.

\subsection{$L^p$-integrability}
As in the case of classical Schr\"odinger operators, the existence of embedded eigenvalues can also be related
with $L^p$-properties of the potentials, which makes this information relevant. Having established the decay
properties of the potentials, $L^p$ properties are now almost immediate. First we consider operators with
regularly varying L\'evy intensities.
\begin{thm}
\label{Lp1}
Let Assumption \ref{eveq} hold with $\varphi \in \cZ_{C_1}(\R^d)$ for some $C_1 \in (0,1)$, and furthermore
let the assumptions of Theorem \ref{thm6} hold. Define
\begin{equation}
\label{pstar}
p^*(\kappa,\alpha,d)=\begin{cases} \frac{d}{\alpha+d-2\kappa} & \kappa \in \left(\frac{d}{2},\frac{d+\alpha}{2}\right)
\vspace{0.2cm} \\
\frac{d}{\alpha} & \kappa \in \left(0,\frac{d}{2}\right]\end{cases}.
\end{equation}
Then for every $p>p^*(\kappa,\alpha,d)$ there exists a constant $M>0$ such that $V \in L^p(B_{M}^c(0))$. Moreover, if
$\varphi \in \cZ_{C_1}(\R^d)\cap \cZ_{\rm b,loc}(\R^d)$, then $V \in L^p(\R^d) $ for every $p > p^*(\kappa,\alpha,d)$
and $V \in L^\infty_{\rm{loc}}(\R^d)$.
\end{thm}
\begin{proof}
 That $V$ belongs to $L^\infty(B^c_M(0))$ for some $M$ follows from the fact that $V$ is decaying.
 Choose first $\kappa \in \left(\frac{d}{2},\frac{d+\alpha}{2}\right)$. By Theorem \ref{thm6} there exist constants $M_1,C_2>0$
 such that for every $x \in B_{M_1}^c(0)$ it follows that $|V(x)|\le C_2|x|^{2\kappa-d-\alpha}\well(|x|^2)$. Put
	$p^*(\kappa,\alpha,d)=d/(\alpha+d-2\kappa)$.
	Since the denominator is positive, we have $p(\alpha+d-2\kappa)-d>0$ for $p>p^*(\kappa,\alpha,d)$. Fix such a $p>p^*(\kappa,\alpha,d)$
and consider $0<p_1<p(\alpha+d-2\kappa)-d$. Since $\lim_{r \to \infty}r^{-p_1}\well(r^2)=0$, there exists $M_2>0$ such that for
all $r>M_2$ it follows that $\well(r^2)\le r^{p_1}$. Take $M=\max\{M_1,M_2\}$ and notice that
	\begin{align*}
	\int_{B_M^c(0)}|V(x)|^pdx&\le d\omega_d C_2^p\int_M^{\infty}r^{2p\kappa-p\alpha-(p-1)d-1}\well(r^2)dr\\
	&\le d\omega_d C_2^p\int_M^{\infty}r^{2p\kappa-p\alpha-(p-1)d-1+p_1}dr\\
	&=\frac{d\omega_d C_2^p}{(p(\alpha+d-2\kappa)-d-p_1)M^{(p(\alpha+d-2\kappa)-d-p_1)}}<\infty,
	\end{align*}
	thus $V \in L^p(B_M^c(0))$.
	
Next take $\kappa \in \left(0,\frac{d}{2}\right)$. By Theorem \ref{thm6} there exist $M_1,C_2>0$ such that for every
$x \in B_M^c(0)$ we have $|V(x)|\le C_2|x|^{-\alpha}\well(|x|^2)$. Write now $p^*(\alpha,d)=d/\alpha$; thus $p\alpha-d>0$
for $p>p^*(\alpha,d)$. Fix such a $p>p^*(\alpha,d)$ and take $0<p_1<p\alpha-d$. Since $\lim_{r \to \infty}r^{-p_1}\well(r^2)=0$,
it follows that there exists a constant $M_2>0$ such that for every $r>M_2$ we have $\well(r^2)\le r^{p_1}$. Taking $M=\max\{M_1,M_2\}$
we see that
	\begin{align*}
	\int_{B_M^c(0)}|V(x)|^pdx&\le d\omega_d C_1^p\int_M^{\infty}r^{-p\alpha+d-1}\well(r^2)dr\\
	&\le d\omega_d C_1^p\int_M^{\infty}r^{-p\alpha+d-1+p_1}dr
	=\frac{d\omega_d C_1^p}{(p\alpha-d-p_1)M^{p\alpha-d-p_1}}<\infty,
	\end{align*}
thus $V \in L^p(B_M^c(0))$. The same can be shown for $\kappa=\frac{d}{2}$ on observing that $\log(r)\well(r^2)$ is
again a slowly varying function at infinity. Finally, the claim $V \in L^p(\R^d) $ for every $p > p^*(\kappa,\alpha,d)$
and $V \in L^\infty_{\rm{loc}}(\R^d)$ follows by Corollary \ref{corbound}.
\end{proof}

Next we consider the case of operators with exponentially light L\'evy intensities.
\begin{thm}
\label{Lp2}
Let Assumption \ref{eveq} hold with $\varphi \in \cZ_{C_1}(\R^d)$ for some $C_1 \in (0,1)$, and
define $p^*(\kappa,\alpha,d)$ as in \eqref{pstar}.
\begin{enumerate}
\item
If the assumptions of Theorem \ref{rateexp} hold, then for every $p>\frac{d}{2}$ there exists a constant $M>0$
such that $V \in L^p(B_{M}^c(0))$.
\item
If the assumptions of Theorem \ref{expdecaype} hold, then for every $p>\frac{d}{2(1-\gamma)}$
there exists a constant $M>0$ such that $V \in L^p(B_{M}^c(0))$.
\end{enumerate}
Furthermore, in either case above, $V \in L^\infty_{\rm{loc}}(\R^d)$, and if $\varphi \in \cZ_{C_1}(\R^d)
\cap \cZ_{\rm b,loc}(\R^d)$, then $V \in L^p(\R^d)$ for every $p>\frac{d}{2}$.
\end{thm}
\begin{proof}
The last claim is implied in either case by Corollary \ref{corbound}.
Consider (1). By Theorem \ref{rateexp} there exist $C_2,M>0$ such that $V(x)\le C_2|x|^{-2}$ for every
$x \in B_{M_2}^c(0)$. Thus for $p>\frac{d}{2}$,
\begin{equation*}
\int_{B_M^c(0)}|V(x)|^pdx\le C_2^p \int_{B_M^c(0)}|x|^{-2p}dx=\frac{C_2^pd\omega_dM^{d-2p}}{2p-d}<\infty,
\end{equation*}
hence $V \in L^p(B_M^c(0))$. To obtain (2), note again that $V \in L^\infty(B_M^c(0))$ for some $M>0$ since
it is decaying. By Theorem \ref{expdecaype} we there exist $M,C_2>0$ such that $V(x)\le C_2|x|^{-2(1-\gamma)}$
for all $x \in B_M^c(0)$. Thus
\begin{equation*}
\int_{B_M^c(0)}|V(x)|^pdx\le C_2^p\int_{B_M^c(0)}|x|^{-2(1-\gamma)p}dx=\frac{C_2^pd\omega_d
M^{d-2(1-\gamma)p}}{2(1-\gamma)p-d}<\infty,
\end{equation*}
hence $V \in L^p(B_M^c(0))$.
\end{proof}

\begin{rmk}
{\rm
For the potentials given in \eqref{eq:motiv_ex} some integrability properties are summarized in the table below.
\begin{center}
\renewcommand{\arraystretch}{1.6}
\begin{tabular}{ccc}
& $V_{\kappa,\alpha} \in L^1(\R^d)$ & $V_{\kappa,\alpha} \in L^d(\R^d)$  \\ \hline
$\kappa \in (0,\tfrac{\delta}{2}) \setminus \{ \tfrac{\delta-\alpha}{2} \}$  & $\alpha > d$ & $\alpha > 1$ \\ \hline
$\kappa = \tfrac{\delta-\alpha}{2}$  & $\alpha > \tfrac{d}{2}$  & $\alpha > \tfrac{1}{2}$ \\ \hline
$\kappa=\tfrac{\delta}{2}$  & $\alpha > d$ & $\alpha > 1$ \\ \hline
$\kappa \in (\tfrac{\delta}{2},\tfrac{\delta+\alpha}{2})$ & $\kappa < l + \tfrac{\alpha}{2}$ &
$\kappa < \tfrac{\delta-1+\alpha}{2}$ \\ \hline
\end{tabular}
\end{center}
}
\end{rmk}

\bigskip

\section{Non-decaying potentials and no-go consequences}
\subsection{Non-decaying potentials for regularly varying L\'evy intensities}
In the previous section we have established criteria under which the potentials decay at infinity, and
derived their decay rates. Next we show that under given conditions the potentials may not decay at all.
The interest in this type of results is the implication that zero-energy eigenfunctions with specific
(too rapid) decays cannot exist. First we consider operators with regularly varying L\'evy intensities.

\begin{thm}
\label{nogo1}
Let Assumption \ref{eveq} hold with $\varphi \in \cZ_{C_1}(\R^d)$ for some $C_1 \in (0,1)$, and
suppose that there exist $\alpha \in (0,2)$ and a slowly varying function $\ell$ at zero such that
$\Phi(u)\sim u^{\frac{\alpha}{2}}\ell(u)$ as $u \downarrow 0$. If there exist
	\begin{enumerate}
		\item $C_2,C_3,M_1,\eta,\beta>0$ and $\delta \in \R$ such that
		\begin{equation*}
		C_2|x|^\delta e^{-\eta |x|^\beta}\le \varphi(x)\le |x|^{\delta}C_3e^{-\eta|x|^{\beta}},
\quad  x \in B_{M_1}^c(0),
		\end{equation*}
		\item $M_2,C_4>0$ such that
		 \begin{equation*}
		 L_\varphi(x)\le C_4 |x|^{\delta+2(\beta-1)}e^{-m|x|^\beta},
		 \end{equation*}
	\end{enumerate}
	then $V(x) \to \infty$ as $|x|\to\infty$.
\end{thm}
\begin{proof}
	Splitting up \eqref{Vformula} we write
	\begin{align*}
	V(x) 
	&=\frac{1}{2\varphi(x)}\int_{B_{C_1|x|}(0)}\Dv j(|h|)dh
	+\frac{1}{\varphi(x)}\int_{B_{C_1|x|}^c(x)}\varphi(y)j(|x-y|)dy-\nu_{0,\alpha}(B_{C_1|x|}^c(0))\\
	&\ge -\frac{1}{2\varphi(x)}\int_{B_{C_1|x|}(0)}|\Dv |j(|h|)dh
	+\frac{1}{\varphi(x)}\int_{B_{C_1|x|}^c(x)}\varphi(y)j(|x-y|)dy-\nu(B_{C_1|x|}^c(0))\\
	&\ge -\frac{L_\varphi(x)}{2\varphi(x)}\cJ(C_1|x|)
	+\frac{1}{\varphi(x)}\int_{B_{C_1|x|}^c(x)}\varphi(y)j(|x-y|)dy-\nu(B_{C_1|x|}^c(0)).
	\end{align*}
	By Corollary \ref{cor2} there exist $C_5,C_6,M_3,M_4 > 0$ such that for $x \in B_{M_3}^c(0)$
	\begin{equation*}
	\cJ(C_1|x|)\ge C_5|x|^{2-\alpha} \quad \mbox{and} \quad \nu(B_{C_1|x|}^c(0)) \ge C_6|x|^{-\alpha}.
	\end{equation*}
	Combining this with the asymptotic control on $f$ and $\varphi$, we obtain for large enough $|x|$
	\begin{equation}\label{firstlob}
	V(x)\ge -C_5|x|^{2\beta-\alpha}+\frac{1}{\varphi(x)}\int_{B_{C_1|x|}^c(x)}\varphi(y)j(|x-y|)dy-C_6|x|^{-\alpha}.
	\end{equation}
	Notice that for values $y \in B_{(1-C_1)|x|}(0)$ we have $|x-y|\ge|x|-|y|>|x|-(1-C_1)|x|=C_1|x|$, then
$B_{(1-C_1)|x|}(0)
	\subset B_{C_1|x|}^c(x)$. Thus
	\begin{align*}
	\int_{B_{C_1|x|}^c(x)}\varphi(y)j(|x-y|)dy
	& \ge \int_{B_{(1-C_1)|x|}(0)}\varphi(y)j(|x-y|)dy
	\end{align*}
	Set $M_4=\frac{M_1}{1-C_1}$ and suppose $|x|>M_4$. Then
	\begin{equation*}
	\varphi(y)\ge C_2|y|^{\delta}e^{-\eta|y|^\beta}
	\end{equation*}
	and
	\begin{align*}
	\int_{B_{(1-C_1)|x|}(0)}\varphi(y)j(|x-y|)dy& \ge C_2\int_{B_{(1-C_1)|x|}(0)}|y|^{\delta}e^{-\eta|y|^\beta}
j(|x-y|)dy.
	\end{align*}
Choosing a large enough constant $M_5 > 0$, the function $r \mapsto r^\delta e^{-\eta r^\beta}$ is decreasing for
$r \geq (1-C_1)M_5$. Taking $|x|>M_5$ gives
\begin{align*}
C_2\int_{B_{(1-C_1)|x|}(0)}|y|^{\delta}e^{-\eta|y|^\beta}j(|x-y|)dy
\ge C_7|x|^{\delta}e^{-\eta (1-C_1)^\beta|x|^\beta}\int_{B_{(1-C_1)|x|}(0)}j(|x-y|)dy,
\end{align*}
with $C_7=C_2(1-C_1)^\delta$. Since $j$ is radial and decreasing, and $|x-y|\le |x|+|y|<(2-C)|x|$, it follows that
\begin{equation*}
C_7|x|^{\delta}e^{-\eta (1-C_1)^\beta|x|^\beta}\int_{B_{(1-C_1)|x|}(0)}j(|x-y|)dy
\ge C_8|x|^{\delta+d}e^{-\eta (1-C_1)^\beta|x|^\beta}j((2-C_1)|x|),
\end{equation*}
where $C_8=C_7(1-C_1)^d\omega_d$. Also, by Proposition \ref{prop2} there exist $C_9,M_6>0$ such that $j((2-C_1)|x|)
\ge C_9|x|^{-d-\alpha}$ for $|x|>M_6$, hence for sufficiently large $|x|$ we get
\begin{equation*}
C_8|x|^{\delta+d}e^{-\eta (1-C_1)^\beta|x|^\beta}j((2-C_1)|x|)\ge C_{10}e^{-\eta(1-C_1)^\beta|x|^\beta}|x|^{-\delta-\alpha},
\end{equation*}
with $C_{10}=C_8C_9$. Using \eqref{firstlob} it then follows that
\begin{equation}
V(x)\ge
-C_5|x|^{2\beta-\alpha}+\frac{C_{10}e^{-\eta(1-C_1)^\beta|x|^\beta}|x|^{-\delta-\alpha}}{\varphi(x)}-C_6|x|^{-\alpha}.
\end{equation}
Set $M=\max_{i=0,\dots,6}M_i$, where $M_0=M_\varphi$. For $|x|>M$ we can use the upper bound on $\varphi$ to conclude
that for large enough $|x|$
\begin{align*}
V(x)\ge -C_5|x|^{2\beta-\alpha}+C_{11}e^{\eta(1-(1-C_1)^\beta)|x|^\beta}|x|^{-\alpha}-C_6|x|^{-\alpha}
\ge C_{12}|x|^{-\alpha}e^{\eta(1-(1-C)^\beta)|x|^\beta},
\end{align*}
with $C_{11},C_{12}>0$.
\end{proof}

\subsection{Non-decaying potentials for exponentially light L\'evy intensities}
Next our aim is to consider the same problem for operators with exponentially light L\'evy intensities,
such as the massive relativistic operator. Due to the light tails, the control in terms of dilated balls
becomes more delicate and first we need the following auxiliary result.

\begin{lem}\label{genC1}
Let $f \in \cZ_{C_1}(\R^d)$ for some $C_1 \in (0,1)$, and suppose that there exist a constant $C_{g,1}>0$
and a decreasing function $g:\R^+\to\R^+$ such that $L_f(x)\le C_{g,1}g(|x|)$ for every $x\in B_{M_f}^c(0)$.
Then $f \in \cZ_{C}(\R^d)$ for all $C \in (0,1)$, and for every such $C$, using the notations $L_{f,C}$ and
$M_{f,C}$ for the function and constant corresponding to $\cZ_{C}(\R^d)$, there exist constants $C_{g,2},
C_{g,3} > 0$ such that $L_{f,C}(x)\le C_{g,2}g(C_{g,3}|x|)$ for all $x \in B_{M_{f,C}}^c(0)$.
\end{lem}
\begin{proof}
The proof of this lemma is presented in the Appendix.
\end{proof}
\begin{rmk}
\rm{
We note that if there is no need to keep the given bound on $L_f$, there are simpler alternatives.
Indeed, if $C>C_1$ and $h \in B_{C|x|}(0)\setminus B_{C_1|x|}(0)$, then $|h|^2>C_1^2|x|^2$ and
\begin{equation*}
\frac{|\DDv|}{|h|^2}\le \frac{4\Norm{f}{\infty}}{C_1^2|x|^2}
\end{equation*}
hence we may use
$L_{f,C}(x)=\max\{L_f(x),\frac{4\Norm{f}{\infty}}{C_1^2|x|^2}\}$.
Moreover, if $f$ is radially symmetric and decreasing, we have
\begin{equation*}
\frac{|\DDv|}{|h|^2}\le \frac{4f((1-C)x)}{C_1^2|x|^2}
\end{equation*}
and we can use
$L_{f,C}(x)=\max\{L_f(x),\frac{4f((1-C)x)}{C_1^2|x|^2}\}$.
}
\end{rmk}

Now we can state our results for this type of operators.
\begin{thm}\label{nogogen}
Let Assumption \ref{eveq} hold with $\varphi \in \cZ_{C_1}(\R^d)$ for some $C_1 \in (0,1)$, and
$\mu(t)\sim C_{\mu}t^{-1-\frac{\alpha}{2}}e^{-\eta_\mu t}$ as $t \to \infty$ for some $\eta_\mu,
C_\mu>0$ and $\alpha \in (0,2]$. Furthermore, suppose that there exist
	\begin{enumerate}
		\item $C_2,C_3,M_1>0$, $\delta \in \R$ and $\eta_\varphi \in (0,\sqrt{\eta_\mu})$ such that
		\begin{equation*}
		C_2|x|^\delta e^{-\eta_\varphi |x|}\le \varphi(x)\le |x|^{\delta}C_3e^{-\eta_\varphi|x|},
\quad  x \in B_{M_1}^c(0),
		\end{equation*}
		\item $M_2,C_4>0$ such that $L_\varphi(x)\le C_4 |x|^{\delta}e^{-\eta_\varphi|x|}$.
	\end{enumerate}
If $C_1 \in \Big(\frac{\eta_\varphi}{\sqrt{\eta_\mu}},1\Big)$, then there exist $C_V,M>0$ such
that $|V(x)|\le C_V$ for every $x \in B_M^c(0)$.
\end{thm}
\begin{proof}
By using estimate \eqref{estgamma} with $\gamma=1$ we have
	\begin{equation*}
	|V(x)|\le C_5+C_6|x|^{-\delta-1-\alpha}e^{-(C_1\sqrt{\eta_\mu}-\eta_\varphi)|x|}, \quad x\in B_M^c(0),
	\end{equation*}
with suitable constants $M, C_5,C_6>0$.	Since $C_1\sqrt{\eta_\mu}-\eta_\varphi>0$, we find $C_V>0$ such
that $|V(x)|\le C_V$ for all $x \in B_M^c(0)$.
\end{proof}

\begin{thm}\label{thmBound}
Let Assumption \ref{eveq} hold with $\varphi \in \cZ_{C_1}(\R^d)$ for some $C_1 \in (0,1)$, and
$\mu(t)\sim C_{\mu}t^{-1-\frac{\alpha}{2}}e^{-\eta_\mu t}$ as $t \to \infty$ for some $\eta_\mu,
C_\mu>0$ and $\alpha \in (0,2]$. Suppose that
\begin{enumerate}
\item
there exist $C_2,C_3,M_1>0$, $\delta \in \R$ and $\eta_\varphi \in (0,\sqrt{\eta_\mu})$
such that
\begin{equation*}
C_2|x|^\delta e^{-\eta_\varphi |x|}\le \varphi(x)\le |x|^{\delta}C_3e^{-\eta_\varphi|x|},
\quad  x \in B_{M_1}^c(0),
\end{equation*}
\item
there exists $C_1 \in (0,1)$ such that $\eta_\varphi<C_1\sqrt{\eta_\mu}$ and $\varphi \in
\cZ_{C_1}(\R^d)$;
\item
$L_\varphi(x)\le C_4 |x|^{\delta}e^{-\eta_\varphi|x|}$ with $C_4>0$, for all $x \in B_{M_\varphi}^c(0)$;
\item
there exist a function $f:\R^d \to \R$ and constants $M_2>0$ and $\omega \ge 2$ such that
\begin{equation*}
|\Dv|\ge f(x)|h|^\omega, \quad h \in B_{C_1|x|}(0), \; x \in B_{M_2}^c(0);
\end{equation*}
\item
$f(x)\ge C_5|x|^{\delta}e^{-\eta_\varphi|x|}$ with $C_5>0$, for every $x \in B_{M_2}^c(0)$.
\end{enumerate}
Define $H_L^{\pm}, H_f^{\pm}$ and $H^\pm = \max\{H_f^+-H_L^-,H_f^--H_L^+\}$ as in Theorem \ref{rateexp2},
whenever they exist. If $H^\pm >0$, then there exist $C_V^{\pm},M>0$ such that $C_V^- \le |V(x)|\le C_V^+$
for every for any $x \in B_M^c(0)$.
\end{thm}
\begin{proof}
The argument goes in the spirit of the proofs of Theorems \ref{rateexp2} and \ref{rateexp3}, and we use the
notations introduced there. Theorem \ref{nogogen} gives the bound $|V(x)|\le C_V^+$, $x \in B_{M^+}^c(0)$,
with suitable constants $C_V^+,M^+$. By Corollary \ref{corexp} and assumption (1) above we have
\begin{equation*}
C_2|x|^\delta e^{-\eta_\varphi|x|}\le \varphi(x)\le C_3|x|^\delta e^{-\eta_\varphi|x|} \quad
\mbox{and} \quad \nu(B_{C_1|x|}^c(0))\le C_6|x|^{\frac{d-\alpha-4}{2}}e^{-C_1\sqrt{\eta_\mu}|x|},
\end{equation*}
for all $x \in B_{M_3}^c(0)$, with appropriate constants $C_6,M_3>0$. A counterpart of \eqref{estabsval}
yields
\begin{align*}
\Big|\int_{B_{C_1|x|}(0)}\Dv j(|h|)dh\Big| \ge |x|^{\delta}e^{-\eta_\varphi|x|}
\max\left\{C_5I^+_\omega -C_4 I^-_2, C_5 I^-_\omega-C_4 I^+_2\right\}.
\end{align*}
Hence by \eqref{estV} we get
\begin{align*}\label{estV2}
|V(x)|
&\ge \frac{1}{2C_3}|x|^{-\delta}e^{\eta_\varphi|x|}\Big|\int_{B_{C_1|x|}(0)}\Dv\varphi(x)j(|h|)dh\Big|-
\frac{2C_6\Norm{\varphi}{\infty}}{C_2}|x|^{\frac{d-\alpha-4-2\delta}{2}}e^{-(C_1\sqrt{\eta_\mu}-\eta_\varphi)|x|}\\
& \ge
\frac{1}{2C_3}\max\left\{C_5I^+_\omega -C_4 I^-_2, C_5 I^-_\omega-C_4 I^+_2\right\} -\frac{2C_6\Norm{\varphi}{\infty}}{C_2}|x|^{\frac{d-\alpha-4-2\delta}{2}}e^{-(C_1\sqrt{\eta_\mu}-\eta_\varphi)|x|},
\end{align*}
thus, using that $C_1\sqrt{\eta_\mu}-\eta_\varphi>0$, we obtain
$\liminf_{|x| \to \infty}|V(x)|\ge \frac{H^\pm}{2C_3}>0$.
This allows to choose $C_V^-,M^->0$ such that $|V(x)|\ge C_V^-$ for all $x \in B_{M^-}^c(0)$.
\end{proof}
\begin{rmk}
\rm{
As for Proposition \ref{rateexp2}, also in this case the excess condition $H^\pm>0$ is unnecessary if
$\Dv$ has a definite sign.
}
\end{rmk}
The following are again counterparts of Corollaries \ref{corsign1}-\ref{corC21}, and we leave the proofs to the reader.
\begin{cor}
\label{corsign3}
Let the assumptions of Theorem \ref{thmBound} hold and $\varphi \in C(B_R^c(0))$. Then $V(x)$ does not change
sign in $B_R^c(0)$.
\end{cor}

\begin{cor}
Let Assumption \ref{eveq} hold with $\varphi \in C^2(\R^d)$, and $\mu(t)\sim C_\mu t^{-1-\frac{\alpha}{2}}e^{-\eta_\mu t}$
as $t \to \infty$ for some $\eta_\mu,C_\mu>0$ and $\alpha \in (0,2]$. Suppose the following properties hold:
\begin{enumerate}
\item
There exist $M_1>0$, $\delta \in \R$, and $\eta_\varphi \in (0,\sqrt{\eta_\mu})$ such that
\begin{equation*}
\varphi(x) \asymp |x|^{\delta}e^{-\eta_\varphi|x|}, \quad  x \in B_{M_1}^c(0).
\end{equation*}
\item
$D^2\varphi(x)$ is positive definite for $x \in B_R^c(0)$ and some $R>0$.
\item
Let $A_{C}(x)$ be defined as in Proposition \ref{propDv} for $C>\frac{\eta_\varphi}{\eta_\mu}$, and denote by
$\lambda_{\rm min}(x)$ and $\lambda_{\rm max}(x)$ the lowest and highest eigenvalues of $D^2\varphi(x)$, respectively.
Let $\lambda^-(x)=\min_{z \in A_C(x)}\lambda_{\rm min}(z)$, $\lambda^+(x)=\max_{z \in A_C(x)}\lambda_{\rm max}(z)$, and
assume
\begin{equation*}
C_\lambda^- |x|^{-\delta}e^{-\eta_\varphi|x|}\le \lambda^-(x)\le \lambda^+(x)\le
C_\lambda^- |x|^{-\delta}e^{-\eta_\varphi|x|}, \quad x \in B_R^c(0).
\end{equation*}
\end{enumerate}
Then there exists $M_V >0$  such that
\begin{equation*}
|V(x)| \asymp \frac{1}{|x|^{2}}, \quad x \in B_{M_V}^c(0).
\end{equation*}
\end{cor}

Finally we show a severe case of no decay also for exponentially light L\'evy intensities.
\begin{thm}\label{nogogen2}
Let Assumption \ref{eveq} hold with $\varphi \in \cZ_{C_1}(\R^d)$ for some $C_1 \in (0,1)$, and $\mu(t)\sim
C_{\mu}t^{-1-\frac{\alpha}{2}}e^{-\eta_\mu t}$ as $t \to \infty$ for some $\eta_\mu, C_\mu>0$ and $\alpha
\in (0,2]$. Assume, moreover, that there exist
\begin{enumerate}
\item
$C_2,C_3,M_1>0$, $\delta \in \R$, $\gamma \ge 1$ and $\eta_\varphi > \eta^*(\gamma)$, where
\begin{equation*}
\eta^*(\gamma)=\begin{cases} \eta_\mu & \gamma=1\\
0 & \gamma>1,
\end{cases}
\end{equation*}
such that
$C_2|x|^\delta e^{-\eta_\varphi |x|^\gamma}\le \varphi(x)\le |x|^{\delta}C_3e^{-\eta_\varphi|x|^\gamma}$
for all $x \in B_{M_1}^c(0)$;
\vspace{0.1cm}
\item
$M_2,C_4>0$ such that
$L_\varphi(x)\le C_4 |x|^{\delta+2(\gamma-1)}e^{-\eta_\varphi|x|^\gamma}$.
\end{enumerate}
If $C_1 \in \big(2-\frac{\eta_\varphi}{\sqrt{2\eta_\mu}},1\big)$, then $V(x)\to \infty$ as $|x| \to \infty$.
\end{thm}
\begin{proof}
We have
\begin{equation*}
V(x)\ge \frac{1}{\varphi(x)}\int_{B_{C_1|x|}^c(0)}\varphi(x+h)j(|h|)dh-\nu(B_{C_1|x|}^c(0))
-\frac{1}{2\varphi(x)}\int_{B_{C_1|x|}(0)}|\Dv |j(|h|)dh.
\end{equation*}
For $x$ such that $(1-C_1)|x| > 1$ the first integral can be bounded as
\begin{align*}
\int_{B_{C_1|x|}^c(0)}\varphi(x+h)j(|h|)dh&\ge \int_{B_{(1-C_1)|x|}(0)}\varphi(y)j(|x-y|)dh\\&\ge j((2-C_1)|x|)
\int_{B_{(1-C_1)|x|}(0)}\varphi(y)dh
\ge j((2-C_1)|x|)\Norm{\varphi}{L^1(B_1(0))},
\end{align*}
leading to
\begin{align*}
V(x)&\ge \frac{j((2-C_1)|x|)}{\varphi(x)}\Norm{\varphi}{L^1(B_1(0))}-C_5\frac{L_\varphi(x)}{\varphi(x)}
\ge C_7|x|^{-\frac{\alpha+d}{2}-1}e^{\eta_\varphi |x|^\gamma-(2-C_1)\sqrt{\eta_\mu}|x|}-C_6|x|^{2(\gamma-1)}.
\end{align*}
For $\gamma>1$ we clearly have
\begin{align*}
V(x)\ge C_8|x|^{-\delta-\frac{\alpha+d}{2}-1}e^{\eta_\varphi |x|^\gamma-(2-C_1)\sqrt{\eta_\mu}|x|},
\end{align*}
thus $V$ is growing to infinity. For $\gamma=1$,
\begin{align*}
V(x)\ge C_7|x|^{-\delta-\frac{\alpha+d}{2}-1}e^{(\eta_\varphi-(2-C_1)\sqrt{\eta_\mu})|x|}-C_6.
\end{align*}
To secure $\eta_\varphi-(2-C_1)\sqrt{\eta_\mu}>0$,
taking into account that $\eta_\varphi>\sqrt{\eta_\mu}$,  	
we can choose $C_1 \in \big(2-\frac{\eta_\varphi}{\sqrt{\eta_\mu}},1\big)$ to obtain
\begin{align*}
V(x)\ge C_7|x|^{-\delta-\frac{\alpha+d}{2}-1}e^{(\eta_\varphi-(2-C_1)\sqrt{\eta_\mu})|x|},
\end{align*}
thus $V(x)$ is again increasing to infinity as $|x|\to\infty$.
\end{proof}
\begin{rmk}
\rm{
If in Theorem \ref{nogogen} $C_1 \not \in \big(\frac{\eta_\varphi}{\sqrt{\eta_\mu}},1\big)$, we can use Lemma
\ref{genC1} to ensure that $C_1$ can be appropriately chosen. However, the bound for $L_\varphi$ needs to be
verified since $\eta_\varphi$ is then replaced by $(1-C_1)^N\eta_\varphi$ for some $N \in \N$. A similar
requirement applies also for Theorem \ref{nogogen2}.
}
\end{rmk}

\section{Sign of the potentials at infinity}
\subsection{Sign for regularly varying L\'evy intensities}

As discussed in the Introduction, apart from the decay properties at infinity, the sign at infinity is one of the
most important features of potentials generating zero eigenvalues. In this section we obtain conditions under which
decaying potentials have a definite sign at infinity.

\begin{thm}\label{positive1}
Let Assumption \ref{eveq} hold with $\varphi \in \cZ_{C_1}(\R^d)$ for some $C_1 \in (0,1)$, and suppose that
there exists $\alpha \in (0,2)$ and a function $\ell$ slowly varying at zero such that $\Phi(u)\sim
u^{\alpha/2}\ell(u)$ as $u \downarrow 0$. If there exist
\begin{enumerate}
\item
a decreasing function $\rho:\R^+ \to \R^+$ such that $\varphi(x)=\rho(|x|)$ for $x \in B_{M_\varphi}^c(0)$;
\item
$C_2,C_3,M_1>0$ and $\kappa\in \left(\frac{d}{2},\frac{d+\alpha}{2}\right)$ such that $C_2|x|^{-2\kappa}\le
\varphi(x)\le C_3|x|^{-2\kappa}$ for $x \in B_{M_2}^c(0)$;
\item
$C_4,M_2>0$ such that $L_\varphi(x)\le C_4|x|^{-(2\kappa+2)}$ for $x \in B_{M_2}^c(0)$,
\end{enumerate}
then there exists $M>0$ such that $V(x)>0$ for every $x \in B_{M}^c(0)$.
\end{thm}
\begin{proof}
Define $\widehat{M}=\max\{M_1,M_2,M_3,M_4\}$. By Proposition \ref{prop2} there exist constants $C_5,C_6>0$
such that
\begin{equation*}
j((1+C_1)|x|)\ge C_5|x|^{-d-\alpha}\well(|x|^2) \quad \mbox{and} \quad j(C_1|x|)\le C_6|x|^{-d-\alpha}\well(|x|^2),
\end{equation*}
for $x \in B_{\widehat{M}}^c(0)$.
	By using Remark \ref{rmkimport} (4) we have that
	\begin{align*}
	\varphi(x)V(x)&=\int_{B^c_{C_1|x|}(x)}(\varphi(y)-\varphi(x))j(|x-y|)dy
	+\frac{1}{2} \int_{B_{C_1|x|}(0)}\Dv j(|h|)dh \\
	&= \left(\int_{B_{|x|}(0) \setminus B_{C_1|x|}(x)}+\int_{B_{|x|}^c(0) \setminus B_{C_1|x|}(x)}\right)
	(\varphi(y)-\varphi(x))j(|x-y|)dy
    \\ & \qquad +\frac{1}{2}\int_{B_{C_1|x|}(0)}\Dv j(|h|)dh.
	\end{align*}
	Radial symmetry and the fact that $\rho$ is decreasing give $\varphi(y)-\varphi(x)\ge 0$
	for $y \in B_{|x|}(0)$. Since $C_1<1$, we have that $B_{(1-C_1)|x|}(0)\subset B_{|x|}(0)
	\setminus B_{C_1|x|}(x)$, and thus
	\begin{align}\label{thm8pass1}
	\varphi(x)V(x)
	&\geq
	\left(\int_{B_{(1-C_1)|x|}(0)} + \int_{B_{|x|}^c(0) \setminus B_{C_1|x|}(x)}\right)(\varphi(y)-\varphi(x))j(|x-y|)dy
	\nonumber\\
	& \qquad + \frac{1}{2}\int_{B_{C_1|x|}(0)}\Dv j(|h|)dh \nonumber \\
	\begin{split}
	&=\int_{B_{(1-C_1)|x|}(0)}\varphi(y)j(|x-y|)dy
	-\varphi(x)\int_{B_{(1-C_1)|x|}(0)}j(|x-y|)dy\\
	&\qquad +\int_{B_{|x|}^c(0) \setminus B_{C_1|x|}(x)}\varphi(y)j(|x-y|)dy
	-\varphi(x)\int_{B_{|x|}^c(0) \setminus B_{C_1|x|}(x)}j(|x-y|)dy\\
	&\qquad +\frac{1}{2}\int_{B_{C_1|x|}(0)}\Dv j(|h|)dh\\
	&=I_1(x)-I_2(x)+I_3(x)-I_4(x)+I_5(x).
	\end{split}
	\end{align}
	Since for $y \in B_{(1-C_1)|x|}(0)$ we have $j(|x-y|)\le j(C_1|x|)$, it follows that
	\begin{equation}\label{passI2}
	I_2(x)\le \varphi(x)j(C_1|x|)|x|^d(1-C_1)^d\omega_d \leq C_7|x|^{-2\kappa-\alpha}\well(|x|^2),
	\end{equation}
	where the second bound is due to $|x|>\widehat{M}$, with $C_7=C_6C_3(1-C_1)^d\omega_d$.
	
	To estimate $I_3(x)$ we define the ring
	\begin{equation*}
	A(|x|)=B_{(2+C_1)|x|}(0)\setminus B_{(1+C_1)|x|}(0) \subset B_{|x|}^c(0)\setminus B_{C_1|x|}(x),
	\end{equation*}
	so that by positivity of $\varphi$ we get
	\begin{equation*}
	I_3(x)\ge \int_{A(|x|)}\varphi(y)j(|x-y|)dy.
	\end{equation*}
	For $y \in A(|x|)$ we have that $j(|x-y|)\ge j((1+C_1)|x|)$ and, since $\rho$ is decreasing, we
furthermore have $\varphi(y)\ge \varphi((2+C_1)|x|)$. Hence
	\begin{equation*}\label{passI3}
	I_3(x) \ge \varphi((2+C_1)|x|)j((1+C_1)|x|)|x|^d\big((2+C_1)^d-(1+C_1)^d\big)
\geq C_8|x|^{-2\kappa-\alpha}\well(|x|^2),
	\end{equation*}
	where the second bound again follows from $|x|>\widehat{M}$ and
$C_8=C_2(2+C_1)^{-2\kappa}C_5\big((2+C_1)^{d}-(1+C_1)^d\big)$.
	Also, by Corollary \ref{cor2} $(1)$ there exists a constant $C_9>0$ such that
	\begin{equation*}
	\nu(B_{C_1|x|}^c(0))\le C_9|x|^{-\alpha}\well(|x|^2), \quad x \in B_{\widehat{M}}^c(0).
	\end{equation*}
	Thus for the fourth integral we obtain
	\begin{equation*}\label{passI4}
	I_4(x)\le \varphi(x)\nu(B_{C_1|x|}^c(0)) \le C_{10}|x|^{-2\kappa-\alpha}\well(|x|^2),
	\end{equation*}
	for $|x|>\widehat{M}$ with $C_{10}=C_9C_3$.
	By Corollary \ref{cor2} $(3)$ there exists also a constant $C_{11}>0$ such that
	\begin{equation*}
	\J(C_1|x|)\le C_{11}|x|^{2-\alpha}\well(|x|^2)
	\end{equation*}
	which gives for the fifth integral
	\begin{equation}\label{passI5}
	|I_5(x)|\le \frac{L_\varphi(x)}{2} \J(C_1|x|) \leq  C_{12}|x|^{-2\kappa-\alpha}\well(|x|^2),
	\end{equation}
	for $|x|>\widehat{M}$ and with $C_{12}=\frac{C_4C_{11}}{2}$.
	Finally, consider $I_1(x)$. For $y \in B_{(1-C_1)|x|}(0)$ we have $j(|x-y|)\ge j(C_1|x|)$, thus
	\begin{equation*}\label{passI11}
	I_1(x)\ge j(C_1|x|)\int_{B_{(1-C_1)|x|}(0)}\varphi(y)dy
\ge C_5|x|^{-d-\alpha}\well(|x|^2)\int_{B_{(1-C_1)|x|}(0)}\varphi(y)dy, \quad |x|>\widehat{M}.
	\end{equation*}
	We only have to evaluate the inner integral in the lower bound of $I_1(x)$. The following two cases
occur.
	
	\medskip
	\noindent
	\emph{Case 1:} Let $\kappa>\frac{d}{2}$. By positivity of $\varphi \in L^1(B_{(1-C_1)|x|}(0))$, on setting
$\widehat{M}>\frac{1}{1-C_1}$ we get
	\begin{equation*}
	\int_{B_{(1-C_1)|x|}(0)}\varphi(y)dy\ge \Norm{\varphi}{L^1(B_1(0))}, \quad |x|>\widehat{M},
	\end{equation*}
	and then
	\begin{equation}\label{passI12}
	I_1(x)
	\ge C_{13}|x|^{-d-\alpha}\well(|x|^2)
	\end{equation}
	where $C_{13}=C_5\Norm{\varphi}{L^1(B_1(0))}$.
	Thus by applying estimates \eqref{passI2}-\eqref{passI5} and \eqref{passI12} to \eqref{thm8pass1},
	\begin{align*}
	\varphi(x)V(x)
	&\ge C_{13}|x|^{-d-\alpha}
	\well(|x|^2)+\Big(C_8-C_7+C_{10}-C_{12}\Big)|x|^{-2\kappa-\alpha}\well(|x|^2)\\
	&=|x|^{-d-\alpha}\well(|x|^2)(C_{13}+C_{14}|x|^{d-2\kappa}),
	\end{align*}
	where $C_{14}=C_8-C_7+C_{10}-C_{12}$. Since now $d-2\kappa<0$ and $C_{13}>0$, we can chose $M>\widehat{M}$
large enough to have $C_{13}+C_{14}M^{d-2\kappa}>0$. Thus in particular
	\begin{equation}\label{lowest1}
	\varphi(x)V(x)\ge |x|^{-d-\alpha}\well(|x|^2)(C_{13}+C_{14}M^{d-2\kappa})>0, \quad x \in B_{M}^c(0).
	\end{equation}

	\medskip
	\noindent
	\emph{Case 2:} Let $\kappa=\frac{d}{2}$ and write $\widehat{M}>\frac{M_1}{(1-C_1)}$. By positivity of $\varphi$,
for $|x|>\widehat{M}$ we get
	\begin{equation*}
	\int_{B_{(1-C_1)|x|}(0)}\varphi(y)dy
	\ge \int_{B_{(1-C_1)|x|}(0)\setminus B_{M_1}(0)}\varphi(y)dy
	\end{equation*}
	Since $(1-C_1)|x|>M_1$, we furthermore have
	\begin{equation*}
	\int_{B_{(1-C_1)|x|}(0)\setminus B_{\widetilde{M}}(0)}\varphi(y)dy
	\ge C_2d\omega_d\int_{M_1}^{(1-C_1)|x|}\frac{dr}{r}=C_2d\omega_d\log\frac{(1-C_1)|x|}{M_1}
	\ge C_{15}\log|x|,
	\end{equation*}
	for a suitable constant $C_{15}>0$.
	Combining this estimate with \eqref{passI11}, we obtain
	\begin{equation}\label{passI13}
	I_1(x)\ge C_{16}|x|^{-2\kappa-\alpha}\log(|x|)\well(|x|^2),
	\end{equation}
	using that $2\kappa=d$ and setting $C_{16}=C_{15}C_5$.
	Again, combining estimates \eqref{passI2}-\eqref{passI5} and \eqref{passI13} with \eqref{thm8pass1},
we arrive at
	\begin{equation}\label{est}
	\varphi(x)V(x)\ge |x|^{-2\kappa-\alpha}\well(|x|^2)(C_{15}\log(|x|)+C_{14})
	\end{equation}
	where $C_{14}$ was defined in the previous case. Since $C_{15}>0$ we can chose $M>\widehat{M}$ such that
$C_{15}\log(M)+C_{14}>0$, and hence
	\begin{equation}\label{est}
	\varphi(x)V(x)\ge |x|^{-2\kappa-\alpha}\well(|x|^2)(C_{15}\log(M)+C_{14})>0, \quad x \in B_{M}^c(0).
	\end{equation}
\end{proof}
\begin{rmk}
{\rm
Combining estimate \eqref{lowest1} with Theorem \ref{thm6}, we obtain under the assumptions of Theorem
\ref{positive1} for $\kappa>\frac{d}{2}$ that, for a suitable $M>0$
\begin{equation*}
V(x) \asymp \frac{1}{|x|^{d+\alpha-2\kappa}}, \quad x \in B_M^c(0).
\end{equation*}
Similarly, combining estimate \eqref{est} with Theorem \ref{thm6}, we obtain under the assumptions of
Theorem \ref{positive1} for $\kappa=\frac{d}{2}$ that, for a suitable $M>0$
\begin{equation*}
V(x) \asymp \frac{\log|x|}{|x|^{\alpha}}, \quad x \in B_M^c(0).
\end{equation*}
Also, we note that Theorem \ref{positive1} continues to hold if $\varphi \in \cZ_{C_1}^\beta(\R^d)$ for
a modulus of continuity $\beta \in L^1_{\rm rad}(\R^d,\nu)$ satisfying $\beta(r)\le Cr^\omega$ for large
enough $r$, and $L_\varphi \le C_4|x|^{-(2\kappa+\omega)}$.
}
\end{rmk}

Next we consider the sign of the potential $V$ assuming $\varphi(x) \sim C|x|^{2\kappa}$, with $\kappa$ below
the critical exponent $\frac{d}{2}$. First we need the following lemma.
\begin{lem}\label{lemest}
	Let $g \in L^\infty(\R^d)$ such that $g(x)=\rho(|x|)$ with a function $\rho:\R^+ \to \R$ satisfying
	$\rho(r)\sim C_\rho r^{-2\kappa}$ as $r \to \infty$, with $\kappa \in \left(0,\frac{d}{2}\right)$. Then
	\begin{enumerate}
		\item as $R \to \infty$, we have
		\begin{equation*}
		\int_{B_R(0)}g(x)dx \sim \frac{d\omega_d C_\rho R^{d-2\kappa}}{d-2\kappa};
		\end{equation*}
		\item for every $p>\frac{d}{2\kappa}$,
		\begin{equation*}
		\int_{B_R^c(0)}g^p(x)dx \sim \frac{d\omega_d C_\rho^p R^{d-2p\kappa}}{2p\kappa-d}
		\end{equation*}
		holds as $R \to \infty$.
	\end{enumerate}
\end{lem}
\begin{proof}
	Pick $\varepsilon \in (0,1)$. There exists $M>0$ such that for every $r>M$
	\begin{equation}\label{eqest}
	\Big|\frac{\rho(r)}{C_\rho r^{-2\kappa}}-1\Big|\le \varepsilon.
	\end{equation}
	With $R>M$ we have
	$\int_{B_{R}(0)}g(x)dx = d\omega_d\left(\int_0^{M}+\int_M^{R}\right) \rho(r)r^{d-1}dr$.
	Writing $C < \infty$ for the first term and using \eqref{eqest}, we have
with suitable $\varepsilon > 0$ that
\begin{equation*}
	C+d\omega_d(1-\varepsilon)C_\rho\int_M^{R}r^{d-2\kappa-1}dr\le \int_{B_{R}(0)}g(x)dx
	\le C+d\omega_d(1+\varepsilon)C_\rho\int_M^{R}r^{d-2\kappa-1}dr.
	\end{equation*}
	Integrating, dividing by the right-hand side displayed in (1), and taking limits proves the
	first claim.
	Next consider $p>\frac{d}{2\kappa}$ and note that
	\begin{equation*}
	\int_{B_R(0)^c}g^p(x)dx=
	d\omega_d \int_{R}^{\infty}\rho^p(r)r^{d-1}dr=d\omega_dC^p_\rho
	\int_{R}^{\infty}\frac{\rho^p(r)}{C^p_\rho r^{-2p\kappa}}r^{d-2p\kappa-1}dr.
	\end{equation*}
	By similar steps as above, (2) also follows.
\end{proof}

We will discuss the main condition of our next theorem after the proof of Theorem \ref{V-}.
\begin{thm}
\label{V+}
Let Assumption \ref{eveq} hold with $\varphi \in \cZ_{C_1}(\R^d)$ for any $C_1 \in (0,1)$, and suppose
that there exist $\alpha \in (0,2)$ and a function $\ell$ slowly varying at zero such that $\Phi(u)\sim
u^{\alpha/2}\ell(u)$ as $u \downarrow 0$. Assume that there exist $C_\varphi>0$, $\kappa \in
\left(\frac{d-1}{2},\frac{d}{2}\right)$, and a decreasing function $\rho:\R^+ \to \R^+$ such that
$\varphi(x)=\rho(|x|)$ with $\rho(r)\sim C_\varphi r^{-2\kappa}$. Moreover, suppose that for every $C_1
\in (0,1)$
$$		
L_\varphi(x) \sim \frac{C_\varphi 4 \kappa(2\kappa+1)d^2}{(1-C_1)^{2\kappa+2}|x|^{2\kappa+2}}
$$
holds. Define the function
	\begin{align}
	\label{posH}
	H_{+}(t)&=\frac{(1-t)^{d-2\kappa}(1+t)^{2\kappa+\alpha}(2\kappa+\alpha+1)^{2\kappa+\alpha}
		+(2\kappa+\alpha)^{2\kappa+\alpha}(2-t)^{d+\alpha}(d-2\kappa)}{(2-t)^{d+\alpha}(1+t)^{2\kappa+\alpha}
		(d-2\kappa)(2\kappa+\alpha+1)^{2\kappa+\alpha}}
	\\& \qquad -\frac{\alpha(1-t)^{d+2\kappa+2}(2-\alpha)
		+dt^d(2-\alpha)(1-t)^{2\kappa+2}+2\kappa(2\kappa+1)d^3\alpha t^{d+2}}
	{d\alpha(2-\alpha)t^{d+\alpha}(1-t)^{2\kappa+2}} \nonumber
	\end{align}
	for $t \in (0,1)$. If
\begin{equation}\label{Kplus}
	K_+(d,\alpha,\kappa):=\max_{t \in (0,1)}H_+(t)>0,
	\end{equation}
	then there exists $R>0$ such that $V(x)>0$ for every $x \in B_{R}^c(0)$.
\end{thm}
\begin{proof}
	We split $\varphi(x)V(x)$ as done in \eqref{thm8pass1} above, and fix $\varepsilon \in (0,1)$.
	As before, we start with $I_2(x)$. We have
	\begin{equation*}
	I_2(x)\le \varphi(x)j(C_1|x|)|x|^d(1-C_1)^d\omega_d
	\end{equation*}
	Using the asymptotics of $\varphi$ in the assumption, and Proposition \ref{prop2} combined with $\well(C_1^2|x|^2)
	\le (1+\varepsilon)\well(|x|^2)$ for $|x|$ large enough, we get
	\begin{equation*}
	I_2(x)\le \frac{(1-C_1)^d}{d C_1^{d+\alpha}}(1+\varepsilon)^3 E(x),
	\end{equation*}
	where
	\begin{equation*}
	E(x):=\frac{C_\varphi\alpha d \omega_d\Gamma\left(\frac{d+\alpha}{2}\right)}{2^{2-\alpha}
		\pi^{\frac{d}{2}}\Gamma\left(1-\frac{\alpha}{2}\right)}|x|^{-2\kappa-\alpha}\well(|x|^2).
	\end{equation*}
	Consider $I_4(x)$. As before, we have
	\begin{equation*}
	I_4(x)\le \varphi(x)\nu(B_{C_1|x|}^c(0))
	\end{equation*}
	thus by making use of Corollary \ref{cor2} (1) we obtain
	\begin{equation*}
	I_4(x)\le \frac{1}{\alpha C_1^\alpha}(1+\varepsilon)^3E(x).
	\end{equation*}
	Consider $I_1(x)$. We have
	\begin{equation*}
	I_1(x)\ge j((2-C_1)|x|)\int_{B_{(1-C_1)|x|}(0)}\varphi(y)dy,
	\end{equation*}
	thus we can use Proposition \ref{prop2} and Lemma \ref{lemest} (1) to obtain
	\begin{equation*}
	I_1(x)\ge \frac{(1-C_1)^{d-2\kappa}}{(2-C_1)^{d+\alpha}(d-2\kappa)}(1-\varepsilon)^3E(x).
	\end{equation*}
	To estimate $I_3(x)$, define for every $\delta>0$
	\begin{equation*}
	A(\delta)=B_{(C_1+1+\delta)|x|}(0)\setminus B_{(C_1+1)|x|}(0)
	\end{equation*}
	and observe that
	\begin{align*}
	I_3(x)&\ge\int_{A(\delta)}\varphi(y)j(|x-y|)dy \ge j((C_1+\delta)|x|)
	\int_{A(\delta)}\varphi(y)dy\\
	&=j((C_1+\delta)|x|)\left(\int_{B_{(C_1+1+\delta)|x|}(0)}\varphi(y)dy
	-\int_{B_{(C_1+1)|x|}(0)}\varphi(y)dy\right).
	\end{align*}
	We may suppose $|x|$ to be large enough to have the previous estimates on $\int_{B_{(C_1+1)|x|}(0)}\varphi(y)dy$
	and $j((C_1+\delta)|x|)$ carry over, thus by Proposition \ref{prop2}  and  Lemma \ref{lemest} (1)
	\begin{align*}
	I_3(x)
	&\ge \frac{(1+C_1+\delta)^{d-2\kappa}(1-\varepsilon)-(1+C_1)^{d-2\kappa}(1+\varepsilon)}
	{(C_1+\delta)^{d+\alpha}(d-2\kappa)}(1-\varepsilon)^2E(x) \\
	& \ge \frac{\big((1+C_1+\delta)^{d-2\kappa}-(1+C_1)^{d-2\kappa}\big)(1-\varepsilon)-
		2\varepsilon(1+C_1)^{d-2\kappa}}{(1+C_1+\delta)^{d+\alpha}(d-2\kappa)}(1-\varepsilon)^2E(x).
	\end{align*}
	Since $0<d-2\kappa<1$, the function $r^{d-2\kappa}$ is concave and thus
	\begin{equation*}
	(1+C_1+\delta)^{d-2\kappa}-(1+C_1)^{d-2\kappa}\ge (d-2\kappa)(1+C_1+\delta)^{d-2\kappa-1}\delta,
	\end{equation*}
	implying
	\begin{equation*}
	I_3(x)\ge\left( \frac{\delta}{(1+C_1+\delta)^{2\kappa+\alpha+1}}(1-\varepsilon)^3-\frac{2\varepsilon(1+C_1)^{d-2\kappa}}
	{(1+C_1+\delta)^{d+\alpha}(d-2\kappa)}(1-\varepsilon)^2\right)E(x).
	\end{equation*}
	Consider the function
	\begin{equation*}
	g(t)=\frac{t}{(1+C_1+t)^{2\kappa+\alpha+1}}
	\end{equation*}
	defined for $t \ge 0$. Observe that $\lim_{t \to \infty}g(t)=g(0)=0$, $g \in C^1$, and
	\begin{equation*}
	g'(t)=\frac{1+C_1-(2\kappa+\alpha)t}{(1+C_1+t)^{2\kappa+\alpha+2}}.
	\end{equation*}
	In particular, we have $g'(t^*)=0$ only for
	\begin{equation*}
	t^*=\frac{1+C_1}{2\kappa+\alpha},
	\end{equation*}
	giving the global maximum
	\begin{equation*}
	g(t^*)=\left(\frac{2\kappa+\alpha}{(1+C_1)(2\kappa+\alpha+1)}\right)^{2\kappa+\alpha}.
	\end{equation*}
	Choosing $\delta=t^*$ we have
	\begin{equation*}	I_3(x)\ge\left(\left(\frac{2\kappa+\alpha}{(1+C_1)(2\kappa+\alpha+1)}\right)^{2\kappa+\alpha}
	(1-\varepsilon)^3-\frac{2\varepsilon(2\kappa+\alpha)^{d+\alpha}(1-\varepsilon)^2}
	{(1+C_1)^{2\kappa+\alpha}(2\kappa+\alpha+1)^{d+\alpha}(d-2\kappa)}\right) E(x).
	\end{equation*}
	Then with $\varepsilon<1-\frac{1}{\sqrt{2}}$ we conclude that
	\begin{equation*}
	I_3(x)\ge\left(\left(\frac{2\kappa+\alpha}{(1+C_1)(2\kappa+\alpha+1)}\right)^{2\kappa+\alpha}
	(1-\varepsilon)^3-\frac{\varepsilon(2\kappa+\alpha)^{d+\alpha}}{(1+C_1)^{2\kappa+\alpha}(2\kappa+\alpha+1)^{d+\alpha}
		(d-2\kappa)}\right)E(x).
	\end{equation*}
	Consider now $I_5(x)$. We have that
	\begin{equation*}
	I_5(x)\ge -\frac{L_\varphi(x)}{2}\cJ(C_1|x|).
	\end{equation*}
	By using Corollary \ref{cor2} (2) and the asymptotics of $f(x)$, we have
	\begin{equation*}
	I_5(x)\ge -\frac{2\kappa(2\kappa+1)d^2C_1^{2-\alpha}}{(2-\alpha)(1-C_1)^{2\kappa+2}}E(x)(1+\varepsilon)^3.
	\end{equation*}
	Combining all these estimates we obtain
	\begin{align*}
	& \hspace{-0.4cm} \varphi(x)V(x) \\
	&\ge \left(\frac{(1-C_1)^{d-2\kappa}(1+C_1)^{2\kappa+\alpha}(2\kappa+\alpha+1)^{2\kappa+\alpha}+
		(2\kappa+\alpha)^{2\kappa+\alpha}(2-C_1)^{d+\alpha}(d-2\kappa)}{(2-C_1)^{d+\alpha}(1+C_1)^{2\kappa+\alpha}
		(d-2\kappa)(2\kappa+\alpha+1)^{2\kappa+\alpha}}(1-\varepsilon)^3\right.\\
&\left. \qquad -\frac{\alpha(1-C_1)^{d+2\kappa+2}
		(2-\alpha)+dC_1^d(2-\alpha)(1-C_1)^{2\kappa+2}+2\kappa(2\kappa+1)d^3C_1^{d+2}}{d\alpha(2-\alpha)C_1^{d+\alpha}
		(1-C_1)^{2\kappa+2}}(1+\varepsilon)^3\right.\\
&\left. \qquad -\frac{\varepsilon(2\kappa+\alpha)^{d+\alpha}}
	{(1+C_1)^{2\kappa+\alpha}(2\kappa+\alpha+1)^{d+\alpha}(d-2\kappa)}\right)E(x).
	\end{align*}
	Consider the polynomials
	\begin{equation*}
	P(\varepsilon)=3+3\varepsilon+\varepsilon^2 \quad \mbox{and} \quad Q(\varepsilon)=3-3\varepsilon+\varepsilon^2,
	\end{equation*}
	and notice that
	$(1+\varepsilon)^3=1+\varepsilon P(\varepsilon)$ and $(1-\varepsilon)^3=1-\varepsilon Q(\varepsilon)$.
	Let $P_M=\max_{\varepsilon \in [0,1]}|P(\varepsilon)|$, $Q_M=\max_{\varepsilon \in [0,1]}|Q(\varepsilon)|$
	and $M=\max\{P_M,Q_M\}$. Then we have
	$(1+\varepsilon)^3 \le 1+M \varepsilon$ and $(1-\varepsilon)^3 \ge 1-M\varepsilon$.
	We obtain
	\begin{align*}
	\varphi(x)V(x)
	&\ge
	\left(H_+(C_1)+\varepsilon\Big(MH_+(C_1)-\frac{(2\kappa+\alpha)^{d+\alpha}}{(1+C_1)^{2\kappa+\alpha}
		(2\kappa+\alpha+1)^{d+\alpha}(d-2\kappa)}\Big)\right)E(x),
	\end{align*}
	with $H_+(t)$ given by \eqref{posH}. Recall the notation \eqref{Kplus}. Since we are free to choose
    $C_1 \in (0,1)$, we pick
	$C_1 \in \arg\max_{t \in [0,1]}H_+(t)$ to obtain
	\begin{equation*}
	\varphi(x)V(x)
	\ge
	\left(K_+(d,\alpha,\kappa)+\varepsilon\Big(MK_+(d,\alpha,\kappa)-\frac{(2\kappa+\alpha)^{d+\alpha}}
	{(1+C_1)^{2\kappa+\alpha}(2\kappa+\alpha+1)^{d+\alpha}(d-2\kappa)}\Big)\right)E(x).
	\end{equation*}
	In case	the factor multiplying $\varepsilon$ is positive, the proof is complete. Otherwise, we can
    chose
	\begin{equation*}
	2\varepsilon <\frac{K_+(d,\alpha,\kappa)}{\frac{(2\kappa+\alpha)^{d+\alpha}}
		{(1+C_1)^{2\kappa+\alpha}(2\kappa+\alpha+1)^{d+\alpha}(d-2\kappa)}-MK_+(d,\alpha,\kappa)}
	\end{equation*}
	and obtain
	\begin{equation*}
	\varphi(x)V(x)>\frac{K_+(d,\alpha,\kappa)}{2}>0,
	\end{equation*}
	for $|x|$ large enough.
\end{proof}

A counterpart of Theorem \ref{V+} for potentials negative at infinity is the following.
\begin{thm}
\label{V-}
Let Assumption \ref{eveq} hold with $\varphi \in \cZ_{C_1}(\R^d)$ for every $C_1 \in (0,1)$, and suppose
that there exist $\alpha \in (0,2)$ and a function $\ell$ slowly varying at zero such that $\Phi(u)\sim
u^{\frac{\alpha}{2}}\ell(u)$ as $u \downarrow 0$. Assume that there exist $C_\varphi>0$, $\kappa \in
\left(0,\frac{d}{2}\right)$, and a decreasing function $\rho: \R^+ \to \R^+$, $\varphi(x)=\rho(|x|)$,
such that $\rho(r)\sim C_\varphi r^{-2\kappa}$. Moreover, suppose that for every $C_1 \in (0,1)$
$$		
L_\varphi(x) \sim \frac{C_\varphi 4 \kappa(2\kappa+1)d^2}{(1-C_1)^{2\kappa+2}|x|^{2\kappa+2}}
$$
holds. Define the function
	\begin{align}
	\label{negH}
	H_-(t;d,\alpha,\kappa,\eta)&=\frac{1}{(d-2\kappa)t^{d+\alpha}}+
	\frac{2\kappa(2\kappa+1)d^2t^{2-\alpha}}{(2-\alpha)(1-t)^{2\kappa+2}}\\
	&\;\; +\left(\frac{2\kappa d+(d+\eta)\alpha}{(d+\eta-2\kappa)\eta}\right)^{\frac{2\kappa}{d+\eta}}
	\frac{d+\eta-2\kappa}{2d\kappa+(d+\eta)\alpha}\frac{1}{t^{\frac{2d\kappa+(d+\eta)\alpha}{d+\eta}}}
	-\left(\frac{1}{\alpha (1+C_1)^{\alpha}}+\frac{(1-t)^d}{d(2-t)^{d+\alpha}}\right) \nonumber
	\end{align}
	for $t \in (0,1)$ and $\eta>0$. If
	\begin{equation}\label{Kminus}
	K_-(d,\alpha,\kappa,\eta):=\min_{t \in (0,1)}H_-(t;d,\alpha,\kappa,\eta)<0,
	\end{equation}
then there exists $R>0$ such that $V(x)<0$ for every $x \in B_{R}^c(0)$.
\end{thm}
\begin{proof}
	Using Remark \ref{rmkimport} (4) we write
	\begin{align*}
	\varphi(x)V(x)&=\int_{B_{|x|}(0) \setminus B_{C_1|x|}(x)}\varphi(y)j(|x-y|)dy
	-\varphi(x)\int_{B_{|x|}(0)\setminus B_{C_1|x|}(x)}j(|x-y|)dy
	\\&\qquad +\int_{B_{|x|}^c(0) \setminus B_{C_1|x|}(x)}\varphi(y)j(|x-y|)dy
	-\varphi(x)\int_{B_{|x|}^c(0)\setminus B_{C_1|x|}(x)}j(|x-y|)dy
	\\&\qquad +\frac{1}{2}\int_{B_{C_1|x|}(0)}\Dv j(|h|)dh\\
	&=I_1(x)-I_2(x)+I_3(x)-I_4(x)+I_5(x).
	\end{align*}
	First, we have
	\begin{align*}
	I_2(x) & \ge \varphi(x)\int_{B_{(1-C_1)|x|}(0)}j(|x-y|)dy\ge\varphi(x)j((2-C_1)|x|)(1-C_1)^d|x|^d\omega_d
	\end{align*}
	Write
	\begin{equation*}
	E(x):=\frac{C_\varphi\alpha d \omega_d\Gamma\left(\frac{d+\alpha}{2}\right)}{2^{2-\alpha}\pi^{\frac{d}{2}}
		\Gamma\left(1-\frac{\alpha}{2}\right)}|x|^{-2\kappa-\alpha}\well(|x|^2),
	\end{equation*}
	and use the asymptotics of $\varphi$ together with Proposition \ref{prop2} giving
	\begin{equation*}
	I_2(x)\ge \frac{(1-C_1)^d}{d(2-C_1)^{d+\alpha}}(1-\varepsilon)^3E(x).
	\end{equation*}
	Next,
	\begin{equation*}
	I_4(x)\ge \varphi(x)\nu(B_{(1+C_1)|x|}^c(0)) \ge \frac{1}{\alpha (1+C_1)^{\alpha}}(1-\varepsilon)^3E(x),
	\end{equation*}
	by using the asymptotics of $\varphi$ and Corollary \ref{cor2} (1).
	Consider now $I_1(x)$. We have
	\begin{equation*}
	I_1(x)\le j(C_1|x|)\int_{B_{|x|}(0)}\varphi(y)dy \le \frac{1}{(d-2\kappa)C_1^{d+\alpha}}(1+\varepsilon)^3E(x),
	\end{equation*}
	where the second bound is implied by Proposition \ref{prop2} and Lemma \ref{lemest} (1).
	Consider $I_3(x)$. Fix $p=\frac{d+\eta}{2\kappa}$ for some $\eta>0$, and estimate by
	\begin{equation*}
	I_3(x)\le \left(\int_{B_{|x|}^c(0)}\varphi^p(y)dy\right)^{\frac{1}{p}}
	\left(\int_{B_{C_1|x|}^c(x)}j^q(|x-y|)dy\right)^{\frac{1}{q}}
	\end{equation*}
	By using Corollary \ref{cor2} (3) and Lemma \ref{lemest} (2) we have
	\begin{align*}
	I_3(x) &\le \frac{(1+\varepsilon)^{3}}{(2p\kappa-d)^{\frac{1}{p}}
		((q-1)d+q\alpha)^{\frac{1}{q}}C_1^{\frac{d}{p}+\alpha}}\\
	&\qquad =\left(\frac{d+p\alpha}{(p-1)(2p\kappa-d)}\right)^{\frac{1}{p}}\frac{p-1}{d+p\alpha}
	C_1^{-\frac{d}{p}-\alpha}(1+\varepsilon)^{3}E(x)\\
	&\qquad =\left(\frac{2\kappa d+(d+\eta)\alpha}{(d+\eta-2\kappa)\eta}\right)^{\frac{2\kappa}{d+\eta}}
	\frac{d+\eta-2\kappa}{2d\kappa+(d+\eta)\alpha}C_1^{-\frac{2d\kappa}{d+\eta}-\alpha}(1+\varepsilon)^{3}E(x).
	\end{align*}
	Finally consider $I_5(x)$. We have
	\begin{equation*}
	I_5(x)\le \frac{L_\varphi(x)}{2}\cJ(C_1|x|).
	\end{equation*}
	By Corollary \ref{cor2} (2) and the asymptotics of $L_\varphi(x)$ we have
	\begin{equation*}
	I_5(x)\le \frac{2\kappa(2\kappa+1)d^2C_1^{2-\alpha}}{(2-\alpha)(1-C_1)^{2\kappa+2}}(1+\varepsilon)^3E(x).
	\end{equation*}
	In sum,
	\begin{align*}
	\varphi(x)V(x)
	&\le
	\left(\Big(\frac{1}{(d-2\kappa)C_1^{d+\alpha}}+\frac{2\kappa(2\kappa+1)d^2C_1^{2-\alpha}}{(2-\alpha)
		(1-C_1)^{2\kappa+2}}\right. \\
	& \qquad
	\left. + \left(\frac{2\kappa d+(d+\eta)\alpha}
	{(d+\eta-2\kappa)\eta}\right)^{\frac{2\kappa}{d+\eta}}\frac{d+\eta-2\kappa}
	{2d\kappa+(d+\eta)\alpha}C_1^{-\frac{2d\kappa}{d+\eta}-\alpha}\Big)(1+\varepsilon)^{3}\right.\\
	&\qquad \left.
	- \Big(\frac{1}{\alpha (1+C_1)^{\alpha}}+\frac{(1-C_1)^d}{d(2-C_1)^{d+\alpha}}\Big)(1-\varepsilon)^3\right)E(x).
	\end{align*}
	Define $P(\varepsilon)=3+3\varepsilon+\varepsilon^2$ and $Q(\varepsilon)=3-
	3\varepsilon+\varepsilon^2$  in such a way that
	$(1+\varepsilon)^3=1+\varepsilon P(\varepsilon)$ and $(1+\varepsilon)^3=1+\varepsilon Q(\varepsilon)$.
	Write $P_M=\max_{\varepsilon \in [0,1]}|P(\varepsilon)|$, $Q_M=\max_{\varepsilon \in [0,1]}|Q(\varepsilon)|$,
    and $M=\max\{P_M,Q_M\}$, leading to $(1+\varepsilon)^3\le 1+M\varepsilon$, and $(1-\varepsilon)^3\ge
	1-M\varepsilon$, and finally thus to
	\begin{align*}
	\varphi(x)V(x)\le
	\left(H_-(C_1;d,\kappa,\alpha)+M\varepsilon H_-(C_1;d,\kappa,\alpha)+
    2M\varepsilon\Big(\frac{1}{\alpha (1+C_1)^{\alpha}}
	+\frac{(1-C_1)^d}{d(2-C_1)^{d+\alpha}}\Big)\right)E(x),
	\end{align*}
	with the function $H_-(t)$, $t \in (0,1)$, as given by \eqref{negH}.
Using the notation \eqref{Kminus} and taking $C_1 \in \arg\min_{t \in (0,1)}H_-(t;d,\alpha,\kappa,\eta)$ gives
	\begin{equation*}
	\varphi(x)V(x)
	\le
	\left(K_-(d,\alpha,\kappa,\eta) +M\varepsilon\Big(K_-(d,\alpha,\kappa,\eta)+
	\frac{2}{\alpha (1+C_1)^{\alpha}}+\frac{2(1-C_1)^d}{d(2-C_1)^{d+\alpha}}\Big)\right)E(x).
	\end{equation*}
	If
	\begin{equation*}
	K_-(d,\alpha,\kappa,\eta)+\frac{2}{\alpha (1+C_1)^{\alpha}}+\frac{2(1-C_1)^d}{d(2-C_1)^{d+\alpha}}<0,
	\end{equation*}
	then the proof is complete. Otherwise we can choose
	\begin{equation*}
	\varepsilon<-\frac{K_-(d,\alpha,\kappa,\eta)}{2\left(K_-(d,\alpha,\kappa,\eta)
		+\frac{2}{\alpha (1+C_1)^{\alpha}}+\frac{2(1-C_1)^d}{d(2-C_1)^{d+\alpha}}\right)}
	\end{equation*}
	to obtain $\varphi(x)V(x)\le \frac{1}{2}K_-(d,\alpha,\kappa,\eta)<0$.
\end{proof}

The following gives sufficient conditions to check the main conditions \eqref{Kplus} and \eqref{Kminus}
in the previous two theorems more directly in terms of the parameters.
\begin{prop}
\label{suffi+-}
The following hold:
\begin{enumerate}
\item
For fixed $d \in \N$ and $\alpha \in (0,2)$ there exists $\kappa^*(d,\alpha)>\frac{d-1}{2}$ such that
$K_+(d,\alpha,\kappa)>0$ for every $\kappa \in \left(\kappa^*,\frac{d}{2}\right)$, where $K_+$ is defined
in \eqref{Kplus}.

\vspace{0.1cm}
\item
For fixed $d \in \N$, $\eta>0$ and $\kappa \in \left(0,\frac{d}{2}\right)$ there exists $\alpha^* \in (0,2)$
such that $K_-(d,\alpha,\kappa,\eta)<0$ for every $\alpha \in (0,\alpha^*)$, where $K_-$ is defined in
\eqref{Kminus}.
\end{enumerate}
\end{prop}
\begin{proof}
Consider (1) and the function
\begin{equation*}
G_+(\kappa):=H_+\left(\frac{1}{2}\right)=\frac{2^{2\kappa+\alpha}}{3^{d+\alpha}(d-2\kappa)}+
\frac{(2\kappa+\alpha)^{2\kappa+\alpha}2^{2\kappa+\alpha}}{3^{2\kappa+\alpha}(2\kappa+\alpha+1)^{2\kappa+\alpha}}
-\frac{(\alpha+d)2^\alpha}{d\alpha}-\frac{2^{1+2\kappa+\alpha}\kappa(2\kappa+1)d^2}{\alpha(2-\alpha)}.
\end{equation*}
Since $G_+$ is a continuous function of $\kappa$ and $\lim_{\kappa \to d/2}G_+(\kappa)= \infty$, there
exists $\kappa^*>\frac{d-1}{2}$ such that for all $\kappa \in \left(\kappa^*,\frac{d}{2}\right)$ we have
$G_+(\kappa)>0$, which implies $K_+(d,\alpha,\kappa)>0$. Next consider (2) and define similarly
\begin{align*}
G_-(\alpha) :=H_-\left(\frac{1}{2}\right) & =\frac{2^{d+\alpha}}{d-2\kappa}+\frac{2\kappa(2\kappa+1)d^2
2^{2\kappa+2}}{(2-\alpha)2^{2-\alpha}}\\
&\quad +\left(\frac{2\kappa d+(d+\eta)\alpha}{(d+\eta-2\kappa)\eta}\right)^{\frac{2\kappa}{d+\eta}}
\frac{d+\eta-2\kappa}{2d\kappa+(d+\eta)\alpha}2^{\frac{2d\kappa}{d+\eta}+\alpha}-\frac{2^\alpha}
{\alpha 3^{\alpha}}-\frac{2^{\alpha}}{d3^{d+\alpha}}.
\end{align*}
Since $G_-$ is a continuous function of $\alpha$ and $\lim_{\alpha \downarrow 0}G_-(\alpha)=-\infty$, there exists
$\alpha^*>0$ such that for every $\alpha \in (0,\alpha^*)$ we have $G_-(\alpha)>0$, implying $K_-(d,\alpha,\kappa)<0$.
\end{proof}
{\begin{rmk}
\rm{
Since we have resonances as $\kappa<d/4$ and $G_-(\alpha)$ is a decreasing function with respect to $\kappa$, we
find that there exists $\alpha^*(d)$ such that for every $\alpha <\alpha^*(d)$ any resonance induces a potential
$V$ that is negative at infinity.
}
\end{rmk}}

\subsection{Sign for exponentially light L\'evy intensities}
We have seen that for operators with exponentially light L\'evy intensities, the dominating part of the upper/lower
bound is given by the integral of the centered difference with $h$ near zero. Hence a first way to show positivity
of the potential is to require positivity of the integrand in the local part.
\begin{thm}\label{signexp1}
Let Assumption \ref{eveq} hold with $\varphi \in \cZ_{C_1}(\R^d)$ for some $C^1 \in (0,1)$, and suppose that there
exist $C_\mu,\eta_\mu>0$ and $\alpha \in (0,2]$ such that $\mu(t)\sim C_\mu t^{-1-\frac{\alpha}{2}}e^{-\eta_\mu t}$
as $t \to \infty$. Furthermore, suppose that $\Dv>0$ for all $x \in B_{M_\varphi}^c(0)$ and $h \in B_{C_1|x|}(x)$.
If there exist a function $f:\R^d \to \R^+$ and constants $M_1,C_2>0$, $\omega \ge 2$ and $\gamma \in \R$ such that
$f(x)\ge C_2|x|^\gamma\varphi(x)$ for all $x \in B_{M_1}^c(0)$ and $\Dv\ge f(x)|h|^\omega$ for all $h \in B_{C_1|x|}(0)$,
then there exists $M>0$ such that $V(x)>0$ for every $x \in B_M^c(0)$.
\end{thm}
\begin{proof}
Using positivity of $\varphi$, we write
\begin{align*}
V(x)
&= \frac{1}{2\varphi(x)}\int_{B_{C_1|x|}(0)}\Dv j(|h|)dh
+\frac{1}{\varphi(x)}\int_{B_{C_1|x|}^c(0)}\varphi(x+h)j(|h|)dh-\nu(B_{C_1|x|}^c(0)) \\
& \ge \frac{1}{2\varphi(x)}\int_{B_{C_1|x|}(0)}\Dv j(|h|)dh-\nu(B_{C_1|x|}^c(0)).
\end{align*}
For sufficiently large $|x|$
\begin{equation*}\label{exp1extI1}
\frac{1}{2\varphi(x)}\int_{B_{C_1|x|}(0)}\Dv j(|h|)dh
\ge \frac{f(x)}{2\varphi(x)}\int_{B_{C_1|x|}(0)}|h|^\omega j(|h|)dh\ge \frac{C_2}{2}|x|^{\gamma}\cJ_\beta(1),
\end{equation*}
where $\beta(r)=r^\omega$, which is bounded by Corollary \ref{corexp}
and the fact that $\omega \ge 2$.  Moreover, by part (1) of the same corollary there exist $C_3,M_2>0$ such that
\begin{equation*}
\nu(B_{C_1|x|}^c(0))\le C_3|x|^{\frac{d-\alpha-4}{2}}e^{-C_1\sqrt{\eta_\mu}|x|}, \quad x \in B_{M_2}^c(0),
\end{equation*}
thus for large enough $|x|$ we get
\begin{equation*}
V(x)\ge|x|^{\gamma}\left(\frac{C_2}{2}-C_{3}|x|^{\frac{d-\alpha-4}{2}-\gamma}e^{-C_1\sqrt{\eta_\mu}|x|}\right).
\end{equation*}
Choose $M_3>0$ such that $r \mapsto r^{\frac{d-\alpha-4}{2}-\gamma}e^{-C_1\sqrt{\eta_\mu}r}$ is decreasing for
$r \geq M_3$ and
\begin{equation*}
\frac{C_2}{2}-C_{3}|x|^{\frac{d-\alpha-4}{2}-\gamma}e^{-C_1\sqrt{\eta_\mu}|x|}>\frac{C_2}{4},
\end{equation*}
and put $M_4=\frac{1}{C_1}$. Then by setting $M=\max_{i=1,\dots,4}M_i$ it follows that $V(x)>\frac{C_2}{4}|x|^{\gamma}>0$,
for all $|x|>M$.
\end{proof}
\begin{rmk}
\rm{
A typical instance when Theorem \ref{signexp1} holds is in case that $\varphi \in C^2(\R^d)$, $D^2\varphi(x)$ is
positive definite for large enough $|x|$, and $\lambda_-(x)\ge |x|^\gamma\varphi(x)$ for some $\gamma \in \R$ with
the ingredients $\lambda_-(x)=\min_{z \in A_C(x)}\lambda_{\rm min}(z)$, $A_C(x)$, and lowest eigenvalue
$\lambda_{\rm min}(x)$ of $D^2\varphi(x)$ used before in Proposition \ref{propDv}.
}
\end{rmk}

Instead of requiring the integrand to be positive as above, we may weaken it to the non-balancing condition used
before.
\begin{thm}
\label{signexp2}
Let Assumption \ref{eveq} hold with $\varphi \in \cZ_{C_1}(\R^d)$ for some $C_1 \in (0,1)$, and suppose there exist
$C_\mu,\eta>0$ and $\alpha \in (0,2]$ such that $\mu(t) \sim C_\mu t^{-1-\frac{\alpha}{2}}e^{-\eta t}$ as $t \to
\infty$. Furthermore, suppose that
\begin{enumerate}
\item
there exist $C_2,C_3,M_1,\eta_\varphi>0$, $\gamma\ge 0$, and $\delta \in \left(\delta_-(\gamma),\delta_+(\gamma)\right)$
such that $C_2 |x|^{\delta}e^{-\eta_\varphi|x|^\gamma}\le \varphi(x)\le C_3 |x|^{\delta}e^{-\eta_\varphi|x|^\gamma}$ for
$x \in B_{M_1}^{c}(0)$, where
\begin{equation*}
\delta_-(\gamma)=
\begin{cases}
-d-\alpha & \mbox{if \, $\gamma=0$}\\ -\infty & \mbox{if \, $\gamma \not = 0$}
\end{cases}
\quad \mbox{and} \quad
\delta_+(\gamma)=
\begin{cases} 0 & \mbox{if \, $\gamma=0$}\\ +\infty & \mbox{if \, $\gamma \neq 0$};
\end{cases}
\end{equation*} \vspace{0.1cm}
\item
$L_\varphi(x)\le C_4|x|^{\delta+2(\gamma-1)}e^{-\eta_\varphi|x|^\gamma}$ with $C_4>0$, for every $x \in
B_{M_\varphi}^c(0)$;
\item
there exist a function $f:\R^d \to \R$ and two constants $M_2>0$ and $\omega \ge 2$ such that
\begin{equation*}
|\Dv|\ge f(x)|h|^\omega, \quad  h \in B_{C_1|x|}(0), \; x \in B_{M_2}^c(0);
\end{equation*}
\item
$f(x) \ge C_5|x|^{\delta+2(\gamma-1)}e^{-\eta_\varphi|x|^\gamma}$ with $C_5>0$, for every $x \in B_{M_2}^c(0)$.
\end{enumerate}
Let $H_L^{\pm}, H_f^{\pm}$ be given as in Theorem \ref{rateexp2}, whenever they exist. The following hold:
\begin{enumerate}
\item[(i)]
If $H_f^+-H_L^->0$, then there exists $M_+>0$ such that $V(x)>0$ for every $x \in B_{M_+}^c(0)$.
\vspace{0.1cm}
\item [(ii)]
If $H_f^--H_L^+>0$ and one of the conditions
\vspace{0.1cm}
\begin{enumerate}
\item
$\gamma \in [0,1)$
\item
$\gamma=1$, $\eta_\varphi<\sqrt{\eta_\mu}$ and $C_1 \in \Big(\frac{\eta_\varphi}{\sqrt{\eta_\mu}},1\Big)$
\end{enumerate}
\vspace{0.1cm}
is satisfied, then there exists $M_->0$ such that $V(x)<0$ for every $x \in B_{M_-}^c(0)$.
\end{enumerate}
\end{thm}
\begin{proof}
Consider (i). By the assumptions we have for the positive and negative parts when $|x|$ is large enough,
	\begin{equation*}
	(\Dv)^+\ge f(x)|h|^\omega, \quad (\Dv)^-\le L_\varphi(x)|h|^\omega.
	\end{equation*}
Thus we have
\begin{align*}
V(x)
&\ge
\frac{1}{2\varphi(x)}\left(\int_{B_{C_1|x|}(0)}\Dv^+ j(|h|)dh- \int_{B_{C_1|x|}(0)}\Dv^- j(|h|)dh\right)
-\nu(B_{C_1|x|}^c(0)) \\
&\ge
\frac{f(x)}{2\varphi(x)}\int_{B_{C_1|x|}(0)}|h|^\omega j(|h|)1_{\{\Dv \ge 0\}}dh
-\frac{L_\varphi(x)}{2\varphi(x)}\int_{B_{C_1|x|}(0)}|h|^2 j(|h|)1_{\{-\Dv \ge 0\}}dh \\
& \qquad -\nu(B_{C_1|x|}^c(0))\\
&\ge
\frac{|x|^{2(\gamma-1)}}{2C_3}\left(C_5\int_{B_{C_1|x|}(0)}|h|^\omega j(|h|)1_{\{\Dv \ge 0\}}dh-
C_4\int_{B_{C_1|x|}(0)}|h|^2 j(|h|)1_{\{-\Dv \ge 0\}}dh\right)\\
&\qquad -C_6|x|^{\frac{d-\alpha-4}{2}}e^{-C_1\sqrt{\eta_\mu}|x|}
\end{align*}
for a suitable constant $C_6>0$, where we used Corollary \ref{corexp} $(1)$. Multiplying by
$2C_3|x|^{2(1-\gamma)}$
and taking $|x| \to \infty$, the right-hand side gives $H_f^+-H_L^->0$.
	Thus there exist $C_7 \in (0,1)$ and $M_+>0$ such that
	\begin{align*}
	C_3|x|^{2(1-\gamma)}V(x)
	&\ge C_7(H_f^+-H_L^-)>0, \quad x \in B_M^c(0).
	\end{align*}
Next consider (ii). Since $\nu(B_{C_1|x|}^c(0))$ is positive, on splitting the integral we can write
\begin{align*}
V(x)&\le \frac{1}{2\varphi(x)}\left(\int_{B_{C_1|x|}(0)}(\Dv)^+ j(|h|)dh-\int_{B_{C_1|x|}(0)}
(\Dv)^- j(|h|)dh\right)\\
&\qquad +\frac{1}{\varphi(x)}\int_{B_{C_1|x|}^c(0)}\varphi(x+h)j(|h|)dh.
\end{align*}
For $|x|$ large enough,
\begin{equation*}
\frac{1}{\varphi(x)}\int_{B_{C_1|x|}^c(0)}\varphi(x+h)j(|h|)dh
\le \frac{\Norm{\varphi}{\infty}}{\varphi(x)} \, \nu(B_{C_1|x|}^c(0)).
\end{equation*}
By Corollary \ref{corexp} $(1)$ we have for a suitable constant $C_9>0$ and sufficiently large $|x|$
\begin{equation*}
\frac{1}{\varphi(x)}\int_{B_{C_1|x|}^c(0)}\varphi(x+h)j(|h|)dh
\le C_9|x|^{\frac{d-\alpha-4}{2}-\delta}e^{\eta_\varphi|x|^\gamma-C_1\sqrt{\eta_\mu}|x|}.
\end{equation*}
Also,
$(\Dv)^+\le L_\varphi(x)|h|^2$, $(\Dv)^-\ge f(x)|h|^\omega$,
and thus by using the bounds on $L_\varphi$ and $f$, we obtain
\begin{align*}
2C_3|x|^{2(1-\gamma)}V(x)&\le C_4\int_{B_{C_1|x|}(0)}|h|^2 j(|h|)dh-C_5\int_{B_{C_1|x|}(0)}|h|^\omega j(|h|)dh\\
& \qquad +C_{10}|x|^{\frac{d-\alpha-4}{2}-\delta+2(1-\gamma)}e^{\eta_\varphi|x|^\gamma-C_1\sqrt{\eta_\mu}|x|},
\end{align*}
where $C_{10}=2C_3C_9$. Taking the limit $|x| \to \infty$ and using that under the assumptions $\eta_\varphi|x|^\gamma
-C_1\sqrt{\eta_\mu} \, |x| \to -\infty$, the right-hand side gives $H_L^+-H_f^-<0$.
so that with $C_{11} \in (0,1)$ and $M_->0$
\begin{align*}
2C_3|x|^{2(1-\gamma)}V(x)
\le C_{11}(H_L^+-H_f^-)<0, \quad x \in B_{M_-}^c(0).
\end{align*}
\end{proof}

\begin{rmk}
\rm{
For $\gamma=0$ we obtain again the case of $\varphi \asymp |x|^{-2\kappa}$ where $2\kappa=-\delta$.
In this case Theorem \ref{signexp2} still holds if $L_\varphi \le C_4|x|^{-(2\kappa+\eta)}$ and
$f(x)\ge C_5|x|^{-(2\kappa+\eta)}$.
Moreover, the theorem continues to hold if $\varphi \in \cZ_{C_1}^{\beta_1}(\R^d)$ for some
$\beta_1 \in L^1_{\rm{rad}}(\R^d,\nu)$ and $|\Dv| \ge f(x)\beta_2(|h|)$ where $\beta_2 \le \beta_1$.

}
\end{rmk}

In spite of the fact that the positivity part of the previous result holds even for $\gamma>1$, the non-balance hypothesis
is not necessary to guarantee positivity. Indeed we can show the following result.
\begin{thm}\label{signexp3}
Let Assumption \ref{eveq} hold with $\varphi \in \cZ_{C_1}(\R^d)$ for some $C_1 \in (0,1)$, and suppose that there exist
$C_\mu,\eta>0$ and $\alpha \in (0,2]$ such that $\mu(t) \sim C_\mu t^{-1-\frac{\alpha}{2}}e^{-\eta t}$ as $t \to \infty$.
Furthermore, suppose that
\begin{enumerate}
\item
there exist $C_2,C_3,M_1,\eta_\varphi,\delta>0$, $\gamma\ge 1$ such that $C_2 |x|^{\delta}e^{-\eta_\varphi|x|^\gamma}\le
\varphi(x)\le C_3 |x|^{\delta}e^{-\eta_\varphi|x|^\gamma}$ for $x \in B_{M_1}^{c}(0)$;
\item
if $\gamma=1$, then $\eta_\varphi>\sqrt{\eta_\mu}$ and $C_1 \in \left(2-\frac{\eta_\varphi}{\sqrt{\eta_\mu}},1\right)$;
\item
$L_\varphi(x)\le C_4|x|^{\delta+2(\gamma-1)}e^{-\eta_\varphi|x|^\gamma}$ with $C_4>0$, for all $x \in B_{M_\varphi}^c(0)$.
\end{enumerate}
Then there exists $M_+>0$ such that $V(x)>0$ for every $x \in B_{M_+}^c(0)$.
\end{thm}
\begin{proof}
Recall that by Proposition \ref{light} and Corollary \ref{corexp}  $\cJ(C_1|x|)$ is bounded by a constant
$C_5>0$ and there exist $C_6,C_7,M_2>0$ such that
\begin{equation*}
j((2-C_1)|x|)\le C_6 |x|^{-\frac{d+\alpha+1}{2}}e^{-(2-C_1)\sqrt{\eta_\mu}|x|}, \qquad
\nu(B_{C_1|x|}^c(0))\le C_7|x|^{\frac{d-\alpha-4}{2}}e^{-C_1\sqrt{\eta_\mu}|x|}
\end{equation*}
for all $x \in B_{M_2}^c(0)$. Thus for sufficiently large $|x|$
\begin{align*}
V(x)&
\ge
\frac{1}{\varphi(x)}\int_{B_{C_1|x|}^c(0)}\varphi(x+h)j(|h|)dh-\frac{1}{2\varphi(x)}\int_{B_{C_1|x|}(0)}|\Dv|j(|h|)dh
-\nu(B_{C_1|x|}^c(0))\\
&\ge
\frac{\Norm{\varphi}{L_1(B_1(0))}}{\varphi(x)}j((2-C_1)|x|)-\frac{L_\varphi(x)}{2\varphi(x)}\cJ(C_1|x|)-\nu(B_{C_1|x|}^c(0))\\
&\ge
C_8|x|^{-\frac{d+\alpha+1}{2}-\delta}e^{\eta_\varphi|x|^\gamma-(2-C_1)\sqrt{\eta_\mu}|x|}-C_9|x|^{2(\gamma-1)}-
C_7|x|^{\frac{d-\alpha-4}{2}}e^{-C_1\sqrt{\eta_\mu}|x|},
\end{align*}
where $C_8=\frac{\Norm{\varphi}{L_1(B_1(0))}}{C_3}C_6$ and $C_9=\frac{C_4C_5}{C_2}$. Multiplying by $|x|^{2(1-\gamma)}$,
using that in the limit $|x| \to \infty$ we get $\eta_\varphi|x|^\gamma-(2-C_1)\sqrt{\eta_\mu}|x| \to \infty$ and
\begin{equation*}
\lim_{|x| \to \infty} C_8|x|^{-\frac{d+\alpha+1}{2}-\delta+2(1-\gamma)}e^{\eta_\varphi|x|^\gamma-(2-C_1)\sqrt{\eta_\mu}|x|}
-C_9-C_7|x|^{\frac{d-\alpha-4}{2}+2(1-\gamma)}e^{-C_1\sqrt{\eta_\mu}|x|}=\infty,
\end{equation*}
we conclude that there exists $M>0$ such that
\begin{align*}
|x|^{2(1-\gamma)}V(x)&\ge C_8|x|^{-\frac{d+\alpha+1}{2}-\delta+2(1-\gamma)}e^{\eta_\varphi|x|^\gamma-(2-C_1)
\sqrt{\eta_\mu}|x|}-C_9-C_7|x|^{\frac{d-\alpha-4}{2}+2(1-\gamma)}e^{-C_1\sqrt{\eta_\mu}|x|} > 1,
\end{align*}
for every $x \in B_M^c(0)$.
\end{proof}


For the massive relativistic operator $L_{m,\alpha}$ we can make a further observation. We show that the effect
of the massive part on the sign of the potential is negligible in sets of the form $B_{M+\delta}(0)\setminus B_M(0)$
if the mass is under a critical value $m_*(\delta)$ depending on the width of the set.
\begin{prop}
\label{massive+-}
Let Assumption \ref{eveq} hold for the massive relativistic Schr\"odinger operator $H = L_{m,\alpha} +V$,
with radially symmetric, decreasing $\varphi \in \cZ_{C_1}(\R^d)$, and $d+\alpha>2$.
\begin{enumerate}
\item
If there exists $M>0$ such that $L_{0,\alpha}\varphi(x) > 0$ for $x \in B_M^c(0)$ and for every
$\delta>0$ there exists $C(\delta)>0$ such that $L_{0,\alpha}\varphi(x)>C(\delta)$ for all
$x \in B_{M+\delta}(0)\setminus B_{M}(0)$, then there exists $m^*(\delta)$ such that for $m<m^*(\delta)$
we have $V(x)>0$ for every $x \in  B_{M+\delta}(0)\setminus B_{M}(0)$.

\vspace{0.1cm}
\item
If there exists $M>0$ such that $L_{0,\alpha}\varphi(x) < 0$ for $x \in B_M^c(0)$ and for any
$\delta>0$ there exists $C(\delta)<0$ such that $L_{0,\alpha}\varphi(x)<C(\delta)$ for all
$x \in B_{M+\delta}(0)\setminus B_{M}(0)$, then there exists $m_*(\delta)$ such that for $m<m_*(\delta)$
we have $V(x)<0$ for every $x \in  B_{M+\delta}(0)\setminus B_{M}(0)$.
\end{enumerate}
\end{prop}
\begin{proof}
Rearranging in \eqref{withG}, we write
$\varphi(x)V(x)=L_{0,\alpha}\varphi(x)-G_{m,\alpha}\varphi(x)$.
Consider claim (1) above. Using Lemma \ref{Glem} above, we get
\begin{equation*}
\varphi(x)V(x) \ge L_{0,\alpha}\varphi(x)-2m\Norm{\varphi}{\infty}.
\end{equation*}
Fix $\delta>0$ and consider $x \in B_{M+\delta}(0)\setminus B_{M}(0)$. Then there exists a constant
$C(\delta) > 0$ such that
\begin{equation*}
\varphi(x)V(x)>C(\delta)-2m\Norm{\varphi}{\infty}.
\end{equation*}
Thus taking $m^*(\delta)=\frac{C(\delta)}{2\Norm{\varphi}{\infty}}$ gives $\varphi(x)V(x)>0$ for $m<m^*(\delta)$.
To show (2) we proceed similarly. Since we also have
\begin{equation*}
\varphi(x)V(x) \le L_{0,\alpha}\varphi(x)+2m\Norm{\varphi}{\infty},
\end{equation*}
fixing $\delta>0$ and taking $x \in B_{M+\delta}(0)\setminus B_M(0)$, we find $C(\delta)<0$ such that
\begin{equation*}
\varphi(x)V(x)<C(\delta)+2m\Norm{\varphi}{\infty}.
\end{equation*}
Taking then $m_*(\delta)=\frac{|C(\delta)|}{2\Norm{\varphi}{\infty}}$, we obtain $\varphi(x)V(x)<0$ for $m < m_*(\delta)$.
\end{proof}

\section{The shape of potentials}
\subsection{Symmetry properties}
Our first result is about rotational symmetry of the potentials. Recall the notation \eqref{defJ}.
\begin{thm}
\label{radsym}
	Let Assumption \ref{eveq} hold with $\varphi \in \cZ(\R^d)$ radially symmetric.	
	Then there exists a function $v:[0,\infty) \to \R$ such that $V(x)=v(|x|)$, for all $x\in \R^d$.
\end{thm}
\begin{proof}
	It suffices to show that for every rotation $R \in {\rm SO}(d)$ we have $V(Rx)=V(x)$. By using
	Proposition \ref{prop3}, rotational symmetry of $\varphi$, the change of variable $w=R^{-1}y$ and the rotational
	invariance of Lebesgue measure, we readily obtain
	\begin{align*}
	V(Rx)
	&=
	\frac{1}{\varphi(Rx)}\lim_{\varepsilon \downarrow 0}\int_{\R^d \setminus B_\varepsilon(Rx)}(\varphi(y)-\varphi(Rx))
	j(|Rx-y|)dy \\
	&=
	\frac{1}{\varphi(x)}\lim_{\varepsilon \downarrow 0}\int_{\R^d \setminus B_\varepsilon(Rx)}(\varphi(y)-\varphi(x))
	j(|Rx-y|)dy \\
	&=
	\frac{1}{\varphi(x)}\lim_{\varepsilon \downarrow 0}\int_{\R^d \setminus B_\varepsilon(x)}
	(\varphi(Rw)-\varphi(x))j(|Rx-Rw|)dw\\
	&=
	\frac{1}{\varphi(x)}\lim_{\varepsilon \downarrow 0}\int_{\R^d \setminus B_\varepsilon(x)}(\varphi(w)-\varphi(x))j(|x-w|)dw
	=V(x).
	\end{align*}
\end{proof}

Next we show results on reflection symmetry of the potentials.
\begin{thm}
\label{reflsym1}
	Let Assumption \ref{eveq} hold with a radially symmetric $\varphi \in \cZ(\R^d)$.
	Furthermore, assume there exists a hyperplane $\pi \ni 0$ such that $\varphi(Px)=\varphi(x)$, where $P$ denotes reflection
	with respect to $\pi$. Then $V(Px)=V(x)$, for all $x\in \R^d$.
\end{thm}
\begin{proof}
	Similarly as in the previous result, using
	Proposition \ref{prop3} and that $P$ is an isometry, we obtain
	\begin{align*}
	V(Px)
	&=
	\frac{1}{\varphi(Px)}\lim_{\varepsilon \downarrow 0}\int_{\R^d \setminus B_\varepsilon(Px)}(\varphi(y)-\varphi(Px))j(|Px-y|)dy=\\
	&=
	\frac{1}{\varphi(x)}\lim_{\varepsilon \downarrow 0}\int_{\R^d \setminus B_\varepsilon(Px)}(\varphi(y)-\varphi(x))j(|Px-y|)dy \\
	&=
	\frac{1}{\varphi(x)}\lim_{\varepsilon \downarrow 0}\int_{\R^d \setminus B_\varepsilon(x)}(\varphi(Pw)-\varphi(x))j(|Px-Pw|)dw\\
	&=
	\frac{1}{\varphi(x)}\lim_{\varepsilon \downarrow 0}\int_{\R^d \setminus B_\varepsilon(x)}(\varphi(w)-\varphi(x))j(|x-w|)dw=V(x).
	\end{align*}
\end{proof}
\begin{thm}
\label{reflsym2}
	Let Assumption \ref{eveq} hold with $\varphi \in \cZ(\R^d)$ radially symmetric and non-negative
	Furthermore, suppose there exists an hyperplane $\pi \ni 0$ such that $\varphi(Px)=-\varphi(x)$,
	where $P$ is reflection with respect to $\pi$.
	If the nodal set of $\varphi$ coincides with $\pi$, then $V(Px)=V(x)$ for every $x \in \R^d$.
\end{thm}
\begin{proof}
	If $x \in \pi$, then $V(Px)=V(x)$ by the fact that $x$ is a fixed point of $P$.\\
	We have for every $x \not \in \pi$,
	\begin{align*}
	V(Px)
	&=
	\frac{1}{\varphi(Px)}\lim_{\varepsilon \downarrow 0}\int_{\R^d \setminus B_\varepsilon(Px)}(\varphi(y)-\varphi(Px))j(|Px-y|)dy\\
	&=
	-\frac{1}{\varphi(x)}\lim_{\varepsilon \downarrow 0}\int_{\R^d \setminus B_\varepsilon(Px)}(\varphi(y)+\varphi(x))j(|Px-y|)dy \\
	&=
	-\frac{1}{\varphi(x)}\lim_{\varepsilon \downarrow 0}\int_{\R^d \setminus B_\varepsilon(x)}(\varphi(Pw)+\varphi(x))j(|Px-Pw|)dw\\
	&=
	-\frac{1}{\varphi(x)}\lim_{\varepsilon \downarrow 0}\int_{\R^d \setminus B_\varepsilon(x)}(-\varphi(w)+\varphi(x))j(|x-w|)dw\\
	&=
	\frac{1}{\varphi(x)}\lim_{\varepsilon \downarrow 0}\int_{\R^d \setminus B_\varepsilon(x)}(\varphi(w)-\varphi(x))j(|x-w|)dw=V(x).
	\end{align*}
\end{proof}

\subsection{Behaviour of the potentials at zero}
\begin{thm}
\label{minatzero}
Let Assumption \ref{eveq} hold with a complete Bernstein function $\Phi\in\mathcal B_0$ such that
$\Phi(u)\asymp u^{\alpha/2}\ell(u)$ as $u \to \infty$, for some $\alpha \in (0,2)$ and a slowly
varying function at infinity $\ell$. Let $\varphi \in \cZ(\R^d)$ satisfy the following properties:
\begin{enumerate}
\item
There exists a constant $r_1>0$ such that $\varphi \in \cZ_{\rm b}(B_{r_1}(0))$; denote
$0<R_1<\inf_{x \in B_{r_2}(0)}R_\varphi(x)$.
\item
There exists a real-valued, decreasing function $\rho \in C^3(\R^+)$, which can be extended to an
even function in $C^3(\R)$, such that $\varphi(x)=\rho(|x|)$.
\item
There exists a unique $R_2$ such that $\rho''(r)<0$ for $r \in (0,R_2)$ and $\rho''(r)>0$ for $r>R_2$,
moreover there exists $R_3>R_2$ such that $\rho'''(r)>0$, for every $r \in (0,R_3)$.
\item
For every $x \in B_{R_2}^c(0)$ the matrix $D^2\varphi(x)$ is positive definite.
\item
There exist constants $r_2,C_1 > 0$ and $\kappa>1$ such that for every $x \in B_{r_2}(0)$
$$
\frac{\varphi(0)|x|^{2\kappa}}{C_1(1+|x|^{2\kappa})}\le \varphi(0)-\varphi(x)\le
\frac{C_1\varphi(0)|x|^{2\kappa}}{1+|x|^{2\kappa}}.
$$
\end{enumerate}
If $\varphi(0)$ is a local maximum of $\varphi$, then $V(0)$ is a strict local minimum of $V$.
\end{thm}
\begin{proof}
	We proceed in several steps.
	
	\medskip
	\noindent
	\emph{Step 1:}
	Using \eqref{Vformula}, we start from the expression
	\begin{align*}
	2(V(0)-V(x))&=\frac{1}{\varphi(0)}\int_{\R^d}{\rm D}_h\varphi(0) j(|h|)dh+
	\frac{1}{\varphi(x)}\int_{\R^d}\Dv j(|h|)dh\\
	&=\frac{1}{\varphi(0)}\int_{\R^d}({\rm D}_h\varphi(0)-\Dv )j(|h|)dh
	-\left(\frac{1}{\varphi(x)}-\frac{1}{\varphi(0)}\right)\varphi(x)V(x)
	\end{align*}
	Choose $R_4>0$ and split the integral as
	\begin{align*}
	2(V(0)-V(x))
	&=\frac{1}{\varphi(0)}\left(\int_{B_{R_4}(0)} + \int_{B_{R_4}^c(0)}\right)({\rm D}_h \varphi(0)-\Dv )j(|h|)dh\\
	&\qquad -\left(\frac{1}{\varphi(x)}-\frac{1}{\varphi(0)}\right)\varphi(x)V(x).
	\end{align*}
	Since
	$\int_{B_{R_4}^c(0)}\Dv j(|h|)dh=2\int_{B_{R_4}^c(0)}(\varphi(x+h)-\varphi(x))j(|h|)dh$,
	we furthermore have
	\begin{align*}
	2(V(0)-V(x))&=\frac{1}{\varphi(0)}\int_{\R^d}({\rm D}_h \varphi(0)-\Dv )j(|h|)dh-
	\left(\frac{1}{\varphi(x)}-\frac{1}{\varphi(0)}\right)\varphi(x)V(x)\\
	&=\frac{1}{\varphi(0)}\int_{B_{R_4}(0)}({\rm D}_h \varphi(0)-\Dv )j(|h|)dh -\left(\frac{1}{\varphi(x)}
	-\frac{1}{\varphi(0)}\right)\varphi(x)V(x)\\
	&\qquad +\frac{2}{\varphi(0)}\int_{B_{R_4}^c(0)}(\varphi(h)-\varphi(0)-\varphi(x+h)+\varphi(x))j(|h|)dh\\
	&=\frac{1}{\varphi(0)}\int_{B_{R_4}(0)}({\rm D}_h \varphi(0)-\Dv )j(|h|)dh
	-\frac{2(\varphi(0)-\varphi(x))}{\varphi(0)} \, \nu(B_{R_4}^c(0))\\
	&\qquad -\frac{2}{\varphi(0)}\int_{B_{R_4}^c(0)}(\varphi(x+h)-\varphi(h))j(|h|)dh
	-\left(\frac{1}{\varphi(x)}-\frac{1}{\varphi(0)}\right)\varphi(x)V(x).
	\end{align*}
	Consider $R_1>0$ as defined and break up the integrals further like
	\begin{align*}
	2(V(0)-V(x))&=\frac{1}{\varphi(0)}\int_{B_{R_1}^c(0) \cap B_{R_4}(0)}({\rm D}_h \varphi(0)-\Dv )j(|h|)dh\\
	&\qquad +\frac{1}{\varphi(0)}\int_{B_{R_1}(0)}({\rm D}_h \varphi(0)-\Dv )j(|h|)dh\\
	&\qquad -\frac{2(\varphi(0)-\varphi(x))}{\varphi(0)} \, \nu(B_{R_4}^c(0))
	-\left(\frac{1}{\varphi(x)}-\frac{1}{\varphi(0)}\right)\varphi(x)V(x)\\
	&\qquad -\frac{2}{\varphi(0)}\int_{B_{R_4}^c(0)}(\varphi(x+h)-\varphi(h))j(|h|)dh. 
	\end{align*}
	Using that $\varphi(x+h)=\varphi(h)+\nabla \varphi(h)\cdot x+\langle D^2 \varphi(\widetilde{x}(x,h)) x,x\rangle$
	with $\widetilde{x}(x,h)\in [h,x+h]$, and $\nabla \varphi(h)=\frac{h}{|h|}\rho'(|h|)$, we obtain
	\begin{align*}
	2(V(0)-V(x))&=\frac{1}{\varphi(0)}\int_{B_{R_1}^c(0) \cap B_{R_4}(0)}({\rm D}_h \varphi(0)-\Dv )j(|h|)dh\\
	&\qquad +\frac{1}{\varphi(0)}\int_{B_{R_1}(0)}({\rm D}_h \varphi(0)-\Dv )j(|h|)dh\\
	&\qquad -\frac{2(\varphi(0)-\varphi(x))}{\varphi(0)} \, \nu(B_{R_4}^c(0))
	-\left(\frac{1}{\varphi(x)}-\frac{1}{\varphi(0)}\right)\varphi(x)V(x)\\
	&\qquad -\frac{2}{\varphi(0)}\int_{B_{R_4}^c(0)}\langle D^2 \varphi(\widetilde{x}(x,h)) x,x\rangle j(|h|)dh. 
	\end{align*}
	Denote
	\begin{align*}
	I_1&=\int_{B_{R_1}^c(0) \cap B_{R_4}(0)}({\rm D}_h \varphi(0)-\Dv )j(|h|)dh\\
	I_2&=\int_{B_{R_1}(0)}({\rm D}_h \varphi(0)-\Dv )j(|h|)dh\\
	I_3&=\int_{B_{R_4}^c(0)}\langle D^2 \varphi(\widetilde{x}(x,h)) x,x\rangle j(|h|)dh.
	\end{align*}
	Next we estimate them one by one.
	
	\medskip
	\noindent
	\emph{Step 2:}
	For every $R_5>R_4$ we have
	\begin{equation*}
	I_3\ge \int_{B_{R_4}^c(0)\cap B_{R_5}(0)}\langle D^2 \varphi(\widetilde{x}(x,h)) x,x\rangle j(|h|)dh.
	\end{equation*}
	Given that $\widetilde{x}(x,h)\in [h,x+h]$, it also belongs to $B_{R_4-|x|}^c(0)\cap B_{R_5+|x|}(0)$. Choose
	$|x|<r_3$ with $r_3<\min\{r_1,r_2\}$ such that $R_4-r_3>R_2$, and denote $A_{R_i,R_j}=B_{R_i}^c(0)\cap B_{R_j}(0)$.
	Consider any $R_6 \in (R_2,R_4-r_3)$ giving that
	$\widetilde{x}(x,h)\in \bar{A}_{R_6,R_5+r_3}$.
	Since $\varphi \in C^2(\R^d)$, we know that $D^2\varphi$ is continuous and  for every $x \in \bar{A}_{R_6,R_5+r_3}$
	the matrix $D^2\varphi(x)$ is positive definite by assumption (4) above. Consider the lowest eigenvalue
	$\lambda_{\rm min}(x)$ for $x \in \bar{A}_{R_6,R_5+r_3}$ and denote $\lambda_{\rm min}=\min_{x \in \bar{A}_{R_6,R_5+r_3}}
	\lambda_{\rm min}(x)>0$. We have
	\begin{equation*}
	\langle D^2 \varphi(\widetilde{x}(x,h)) x,x\rangle\ge \lambda_{\rm min}|x|^2,
	\end{equation*}
	and thus
	\begin{equation*}
	I_3\ge \lambda_{\rm min} |x|^2 \nu(A_{R_4,R_5}).
	\end{equation*}
	
	\medskip
	\noindent
	\emph{Step 3:}
	Concerning $I_2$, we have
	\begin{align*}
	I_2&\le \int_{B_{R_1}(0)}|{\rm D}_h \varphi(0)-\Dv |j(|h|)dh
	\le \int_{B_{R_1}(0)}(|{\rm D}_h \varphi(0)|+|\Dv |)j(|h|)dh\\
	&\le (L_\varphi(0)+L_\varphi(x))\cJ(R_1),
	\end{align*}
	Recall that $\cJ(R) \to 0$ as $R \to 0$.
	
	\medskip
	\noindent
	\emph{Step 4:}
	To estimate $I_1$, fix $h \in A_{R_1,R_4}$ and consider the function $F_h(x)=\Dv$.
	We need first some elementary analysis. Write
	\begin{equation*}
	F_h(x)=F_h(0)+\nabla F_h(0)\cdot x +\frac{1}{2}\langle D^2F_h(0)x,x\rangle-\cR_h(x),
	\end{equation*}
	where $\cR_h(x)=o(|x|^2)$.
	Since
	$\nabla F(x)=\nabla \varphi(x+h)-2\nabla\varphi(x)+\nabla \varphi(x-h)$,
	where $\nabla \varphi(x)=\frac{x}{|x|}\rho'(|x|)$
	and $\nabla \varphi(0)=0$. Thus
	\begin{align*}
	\nabla F(0)&=\nabla \varphi(h)-2\nabla\varphi(0)+\nabla \varphi(-h)
	=\frac{h}{|h|}\rho'(|h|)-\frac{h}{|h|}\rho'(|h|)=0,
	\end{align*}
	allowing to write
	$F_h(0)-F_h(x)=-\frac{1}{2}\langle D^2F_h(0)x,x\rangle+\cR_h(x)$,
	so that
	\begin{align}
	\label{I1toestimate}
	I_1
	&=-\frac{1}{2}\int_{A_{R_1,R_4}}\Big(\sum_{i,j=1}^{d}\frac{\partial^2 F}
	{\partial x_i\partial x_j}(0)x_ix_j+\cR_h(x)\Big)j(|h|)dh.
	\end{align}
	Computing derivatives gives for $i \not = j$
	\begin{align*}
	\frac{\partial^2 \varphi}{\partial x_i\partial x_j}(x)
	&=\frac{x_ix_j}{|x|^3}(|x|\rho''(|x|)-\rho'(|x|)),
	\end{align*}
	and
	\begin{equation*}
	\frac{\partial^2 \varphi}{\partial x_i\partial x_j}(0)=0, \quad
	\frac{\partial^2 F_h}{\partial x_i\partial x_j}(0)=2\frac{h_ih_j}{|h|^3}(|h|\rho''(|h|)-\rho'(|h|)).
	\end{equation*}
	In particular, by symmetry of the domain $A_{R_1,R_4}$ we get
	\begin{equation*}
	\int_{A_{R_1,R_4}}\frac{h_ih_j}{|h|^3}(|h|\rho''(|h|)-\rho'(|h|))x_ix_jj(|h|)dh=0
	\end{equation*}
	and hence in \eqref{I1toestimate} the mixed derivatives do not contribute.
	By further computing,
	\begin{align*}
	\frac{\partial^2 \varphi}{\partial x_i^2}(x)
	=\frac{(|x|^2-x_i^2)\rho'(|x|)+x_i^2|x|\rho''(|x|)}{|x|^3}, \quad
	\frac{\partial^2 \varphi}{\partial x_i^2}(0)=\lim_{k \to 0}\frac{\rho'(k)}{k}=\rho''(0),
	\end{align*}
	moreover
	\begin{align*}
	\frac{\partial^2 F}{\partial x_i^2}(0)
	&=2\frac{(|h|^2-h_i^2)(\rho'(|h|)-|h|\rho''(0))+h_i^2|h|(\rho''(|h|)-\rho''(0))}{|h|^3}.
	\end{align*}
	
	Notice that due to $\rho \in C^3(\R^+)$, there exist functions $\widetilde{h}_1(|h|)$ and
	$\widetilde{h}_2(|h|)$ such that
	\begin{equation*}
	\rho'(|h|)-|h|\rho''(0)=\rho'''(\widetilde{h}_1(|h|))|h|^2
	\quad \mbox{and} \quad \rho''(|h|)-\rho''(0)=\rho'''(\widetilde{h}_2(|h|))|h|.
	\end{equation*}
	This gives
	\begin{align*}
	\frac{\partial^2 F_h}{\partial x_i^2}(0)&=2\frac{\left(\sum_{j \not = i}h_j^2\right)
		\rho'''(\widetilde{h}_1(|h|))|h|^2+h_i^2|h|^2\rho'''(\widetilde{h}_2(|h|))}{|h|^3}.
	\end{align*}
	Recall that $\widetilde{h}_1(|h|),\widetilde{h}_2(|h|)\in [0,|h|]\subset [0,R_4)$.
	We can then choose $R_4 \in (R_2,R_3)$
	such that $\frac{\partial^2 F_h}{\partial x_i^2}(0)>0$ for all $h \in A_{R_1,R_4}$. This yields
	\begin{equation*}
	I_1=-\frac{1}{2}\int_{A_{R_1,R_4}}\Big(\sum_{i=1}^{d}\frac{\partial^2 F}{\partial x_i^2}(0)x_i^2
	+\cR_h(x)\Big)j(|h|)dh\le-\frac{1}{2}\int_{A_{R_1,R_4}}\cR_h(x)j(|h|)dh.
	\end{equation*}
	Again, since $\varphi \in C^3(\R^d)$, the remainder term $\cR_h(x)$ can be written as
	\begin{equation*}
	\cR_h(x)=\frac{1}{6}\sum_{i,j,k=1}^{d}\frac{\partial^3 F_h(x)}{\partial x_i\partial x_j
		\partial x_k}(\widetilde{x}(x,h))x_ix_jx_k
	\end{equation*}
	for some $\widetilde{x}(x,h)\in [0,x]$. Since $x \in B_{r_3}(0)$ and $h \in A_{R_1,R_4}$, take $(x,h)
	\in \overline{B_{r_3}(0)\times B_{R_4}(0)}=C$, which implies that there exists a constant $M>0$ such
	that for every $(x,h)\in C$
	\begin{equation*}
	\cR_h(x)\ge -|\cR_h(x)|\ge -2M|x|^3.
	\end{equation*}
	
	Thus we have
	\begin{equation*}
	\sum_{i=1}^{d}\pdsup{F}{x_i}{2}(0)x_i^2+R_h(x)\ge \sum_{i=1}^{d}\pdsup{F}{x_i}{2}(0)x_i^2-2M|x|^3,
	\end{equation*}
	which is positive for $|x|$ small enough. Choose then $r_4<r_3$ such that for every $x \in B_{r_4}(0)$
	the right-hand side above is positive, and suppose that $R_1<\frac{R_4}{2}$. Then $A_{R_1,R_4}\subset
	A_{R_4/2,R_4}$ and for $x \in B_{r_4}(0)$ we have
	\begin{equation*}
	\int_{A_{R_1,R_4}}\Big(\sum_{i=1}^{d}\frac{\partial^2 F}{\partial x_i^2}(0)x_i^2+\cR_h(x)\Big)j(|h|)dh
	\ge
	\int_{A_{R_4/2,R_4}}\Big(\sum_{i=1}^{d}\frac{\partial^2 F}{\partial x_i^2}(0)x_i^2+\cR_h(x)\Big)
	j(|h|)dh.
	\end{equation*}
	Hence
	\begin{align*}
	I_1&=-\frac{1}{2}\int_{A_{R_1,R_4}}\Big(\sum_{i=1}^{d}\frac{\partial^2 F}{\partial x_i^2}(0)x_i^2
	+\cR_h(x)\Big)j(|h|)dh\\
	&\le-\frac{1}{2}\int_{A_{R_4/2,R_4}}\Big(\sum_{i=1}^{d}\frac{\partial^2 F}{\partial x_i^2}(0)x_i^2
	+\cR_h(x)\Big)j(|h|)dh
	\le-\frac{1}{2}\int_{A_{R_4/2,R_4}}\cR_h(x)j(|h|)dh.
	\end{align*}
	This then leads to
	$I_1\le M_2|x|^3$ where $M_2=M\nu\left(A_{R_4/2,R_4}\right)$.
	
	\medskip
	\noindent
	\emph{Step 5:}
	Using the lower bound in assumption (5) above to estimate the factor $\frac{2(\varphi(0)-\varphi(x))}{\varphi(0)}$,
	a combination of the above steps gives
	\begin{align*}
	2(V(0)-V(x))&\le \frac{1}{\varphi(0)}M_2|x|^3+\frac{L_\varphi(0)+L_\varphi(x)}{\varphi(0)}\cJ(R_1)\\
	&-\frac{2|x|^{2\kappa}}{C_1(1+|x|^{2\kappa})}\nu(B_{R_4}^c(0))
	-\frac{2\lambda_{\rm min}}{\varphi(0)}|x|^2\nu(A_{R_4,R_5})
	-\left(\frac{1}{\varphi(x)}-\frac{1}{\varphi(0)}\right)\varphi(x)V(x).
	\end{align*}
	The same assumption also gives on rearrangement
	\begin{equation*}
	\quad \frac{1}{\varphi(x)}-\frac{1}{\varphi(0)}\ge
	\frac{|x|^{2\kappa}}{\varphi(0)(C_1(1+|x|^{2\kappa})-|x|^{2\kappa}}.
	\end{equation*}
	Since $\rho$ is decreasing, we may choose $r_3<\min\{1,\inf_{x \in B_{r_2}(0)}R_\varphi(x)\}$ so that
	$\varphi(x)\ge \rho(1)$ and thus
	\begin{align*}
	2(V(0)-V(x))
	&\le
	\frac{M_2}{\varphi(0)}|x|^3+\frac{(L_\varphi(0)+L_\varphi(x))}{\varphi(0)}\cJ(R_1)
	\\& \qquad -\frac{2\nu(B_{R_4}^c(0))}{C_1(1+|x|^{2\kappa})}|x|^{2\kappa}
	-\frac{2\lambda_{\rm min}}{\varphi(0)} \, \nu(A_{R_4,R_5}) |x|^2 \\
	& \qquad +\rho(1)\Norm{V}{L^\infty(B_{r_3}(0))}\frac{|x|^{2\kappa}}{\varphi(0)(C_1(1+|x|^{2\kappa})-|x|^{2\kappa}},
	\end{align*}
	where in the last line we used that by assumption (1) we also know that $V$ is bounded in $B_{r_3}(0)$ and thus
	$V(x)\ge-|V(x)|\ge -\Norm{V}{L^\infty(B_{r_3}(0))}$.
	
\begin{figure}
\centering
\begin{tikzpicture}[scale=0.6]

\node at (0.3,0.3) {\large $I_2$};
\draw (0,0) circle (5);
\draw (0,0) circle (1);
\draw (0,0) circle (2);
\draw (0,0) circle (6);
\draw[dashed] (0,0) circle (4);
\draw[dashed] (0,0) circle (7);
\draw[dotted, thick] (0,0) circle (4.5);
\node at (1.3,-0.3) {\large $R_1$};
\node at (5.3,0.3) {\large $R_4$};
\node at (2.3,0.3) {\large $R_2$};
\node at (2.5,-1.7) {\large $R_4-r_3$};
\node at (3.7,3) {\large $R_6$};
\node at (6.2,1.5) {\large $R_5$};
\node at (6.4,4.5) {\large $R_5+r_3$};
\node at (0,9) {};
\node at (0,-9) {};
\node at (9,0) {};
\node at (-9,0) {};
\draw [decoration={markings,mark=at position 1 with
	{\arrow[scale=3,>=stealth]{>}}},postaction={decorate}] (0,-9) -- (0,9);
\draw [decoration={markings,mark=at position 1 with
	{\arrow[scale=3,>=stealth]{>}}},postaction={decorate}] (-9,0) -- (9,0);
\path [draw=none,fill=red, fill opacity = 0.1,even odd rule] (0,0) circle (5) (0,0) circle (1);
\path [draw=none,fill=green, fill opacity = 0.1,even odd rule] (0,0) circle (5) (0,0) circle (6);
\path [draw=none,fill=blue, fill opacity = 0.1,even odd rule] (0,0) circle (1);
\node at (-2.8,-1.5) {\large $I_1$};
\node at (3.2,4.5) {\large $I_3$};
\end{tikzpicture}
\caption{Sketch of domain decomposition in the proof of Theorem \ref{minatzero}.}\label{fig1}
\end{figure}

	\medskip
	\noindent
	\emph{Step 6:}
	To conclude, we need an estimate on $\cJ(R_1)$ obtained in Step 3. Since $R_1$ is featured only in $\cJ$, we have the
	freedom to choose it $x$-dependent. By \cite[Th. $3.4$]{KSV} we know that
	$j(r)\asymp r^{-d-\alpha}\ell(r^{-2})$
	as $r \to 0$, thus we have by the monotone density theorem
	\begin{align*}
	\cJ(R_1)&=\int_{B_{R_1}(0)}|h|^2j(|h|)dh
	=d\omega_d \int_0^{R_1}r^{d+1}j(r)dr\\
	&\le C_2d\omega_d \int_0^{R_1}r^{1-\alpha}\ell(r^{-2})dr
	\le C_3d\omega_d R_1^{2-\alpha}\ell(R_1^{-2})
	\end{align*}
	for a constant $C_3$ and $R_1$ small enough. Choose $R_1(x)=|x|^p$ for some $p>0$, and observe that $|x|^p<R_2$
	for small enough $|x|$. We then obtain
	\begin{align*}
	2(V(0)-V(x))
	&\le
	\frac{1}{\varphi(0)}M|x|^3+C_3d\omega_d\frac{(f(0)+f(x))}{\varphi(0)}|x|^{p(2-\alpha)}\ell(|x|^{-2p})
	\\&
	\qquad -\frac{2|x|^{2\kappa}}{C_1(1+|x|^{2\kappa})}\nu(B_{R_4}^c(0))-\frac{2\lambda_{\rm min}}{\varphi(0)}
	|x|^2\nu(A_{R_4,R_5})\\
	&
	\qquad +\frac{|x|^{2\kappa}}{\varphi(0)(C_1(1+|x|^{2\kappa})-|x|^{2\kappa}}\rho(1)\Norm{V}{L^\infty(B_{r_3}(0))}\\
	&=
	\left(q(|x|)-\frac{2\lambda_{\rm min}}{\varphi(0)}\,\nu(A_{R_4,R_5})\right) |x|^2,
	\end{align*}
	where
	\begin{align*}
	q(|x|)
	&=
	\frac{M}{\varphi(0)}|x|+C_3d\omega_d\frac{f(0)+f(x)}{\varphi(0)}|x|^{p(2-\alpha)-2}\ell(|x|^{-2p})\\
	&
	\qquad
	-\left(\frac{2\nu(B_{R_4}^c(0))}{C_1(1+|x|^{2\kappa})} + \frac{\rho(1)\Norm{V}{L^\infty(B_{r_3}(0))}}
	{\varphi(0)(C_1(1+|x|^{2\kappa})-|x|^{2\kappa})}\right)|x|^{2\kappa-2}.
	\end{align*}
	Since $\kappa>1$, we can choose $p>\frac{2}{2-\alpha}$ so that $q(|x|)\to 0$ as $|x|\to 0$. Since, moreover, $q$ is
	a continuous function, there is $r_5<r_4$ such that
	\begin{equation*}
	q(|x|)\le \frac{\lambda_{\rm min}}{\varphi(0)} \, \nu(A_{R_4,R_5}), \quad x \in B_{r_5}(0).
	\end{equation*}
	Thus for every $x \in B_{r_5}(0)$ we have
	\begin{equation*}
	V(0)-V(x)\le -\frac{\lambda_{\rm min}}{2\varphi(0)} \, \nu(A_{R_4,R_5})<0.
	\end{equation*}
\end{proof}

\begin{thm}
\label{Vnought}
	Let Assumption \ref{eveq} hold with strictly positive, radial and decreasing $\varphi \in \mathcal{Z}(\R^d)$.
	If $\varphi(0)$ is the global maximum of $\varphi$, then $V(0)\le 0$.
\end{thm}
\begin{proof}
	Since
	\begin{equation*}
	V(0)=\frac{1}{2\varphi(0)}\int_{\R^d}(\varphi(-h)-2\varphi(0)+\varphi(h))j(|h|)dh,
	\end{equation*}
	and $\varphi(0)>0$, we have
	\begin{equation*}
	\varphi(-h)-2\varphi(0)+\varphi(h)=2(\varphi(h)-\varphi(0))\le 0,
	\end{equation*}
	and the statement is immediate.
\end{proof}

\subsection{Pinning effect}
\label{pin}
While in the above we have seen what are various properties of potentials generating an eigenvalue at zero,
such as decay and sign properties at infinity, in a final count it is interesting to try to understand the
mechanisms underlying the action of such potentials. From Theorem \ref{Vnought} we know that the potentials
create a well around zero. Intuitively it is appealing to think that for a potential which has a positive
part in a neighbourhood of infinity it is ``easier" to create a bound state than for a potential that is
negative everywhere, and this comes down to an energetic advantage created by the positive potential barrier
far out. Since we know that purely negative potentials can also create zero-energy bound states, it is
reasonable to think that they will use instead a pinning force exerted from a deeper well around zero than
a potential positive at infinity. We conclude this paper by an analysis of this mechanism, and how it
combines with the sign and decay behaviours.

First note that the positivity~--~slow decay and negativity~--~rapid decay pattern discussed in the Introduction
for classical Schr\"odinger operators is supported by our results for non-local cases, however, there is a grey
area and the switch-over is not sharp. Indeed, from the explicit cases \eqref{eq:motiv_ex} and the behaviours
\eqref{exdecays}-\eqref{pos} we see that the critical $\kappa$ for the slow-fast transition is $\frac{\delta}{2}$,
while for positivity-negativity it is $\frac{\delta-\alpha}{2}$. Proposition \ref{suffi+-} is further in line with
these behaviours: we have positivity for $\kappa \ge \frac{d}{2}$, while negativity holds for small values of
$\kappa$, though the latter appears hidden. Indeed, with a $\kappa_1< \frac{d}{2}$ there exists $\alpha^*(\kappa,d)$
such that for every $\alpha \in (0,\alpha^*(\kappa_1,d))$ we obtain a negative potential. Fixing then $\alpha_1 \in
(0,\alpha^*(\kappa_1,d))$, we can construct $\kappa^*(d,\alpha_1)$ such that if $\kappa_2 \in \left(\kappa^*(d,\alpha_1),
\frac{d}{2}\right)$ we obtain a positive potential. This also happens in the examples in \eqref{eq:motiv_ex}. Indeed,
for $\kappa_1< \frac{d}{2}$ (and $l=0$), there exists $\alpha^*$ such that $\kappa_1<\frac{d-\alpha}{2}$ for all
$\alpha \in (0,\alpha^*)$. If we fix $\alpha \in (0,\alpha^*)$, we obtain a negative potential for $\kappa_1$, however,
 we may choose $\kappa_2>\frac{d-\alpha}{2}=\kappa_*(d,\alpha)$ to construct a positive one.

In Theorem \ref{signexp2} for exponentially light cases this picture is more delicate to follow. Indeed positivity or
negativity of the potential is expressed in terms of the excess condition, relating to the shape of the eigenfunction (cf.
Remark \ref{shortrange2} (2)). Consider the case $\gamma=0$ with $\varphi \asymp |x|^{-2\kappa}$ where $2\kappa=-\delta$,
$L_\varphi \le C_4|x|^{-(2\kappa+\eta)}$ and $f(x)\ge C_5|x|^{-(2\kappa+\eta)}$. Then we have shown that $V(x)\asymp
|x|^{-\eta}$. As we have seen, for regular enough $\varphi$ (e.g., satisfying the assumptions of Proposition \ref{shapeeigen}),
$L_\varphi$ is generally close to a second derivative, hence we may expect  $\eta=2$. For less regular functions $L_\varphi$
plays the role of a H\"older constant for the gradient or H\"older-Zygmund constant for the function itself if we require
$\varphi \in \cZ_{C_1}^\beta(\R^d)$ for $\beta(r)=r^\omega$, with some $\omega\in (0,2)$ and $\beta \in L^1_{\rm{\rm rad}}
(\R^d,\nu)$, and then again we expect $\eta\leq2$.

Now we show how these behaviours further combine with behaviour at zero.
\begin{prop}
\label{intest}
Let Assumption \ref{eveq} hold for two different potentials $V_+,V_-$ with corresponding zero-energy eigenfunctions
$\varphi_{+}, \varphi_- \in \cZ_{C_1}(\R^d)$. Suppose that both $\varphi_{+}$ and $\varphi_-$ are radially symmetric
and decreasing, with the property that $\varphi_+(0)=\varphi_-(0)=a$. Suppose, furthermore, that there exists a unique
$R>0$ such that $\varphi_+(x)>\varphi_-(x)$ for $0<|x|<R$. Denote
\begin{equation*}
I_{\pm}(R)=\int_0^{R}r^{d-1}\rho_{\pm}(r)j(r)dr, \qquad J_{\pm}(R)=\int_{R}^{\infty}r^{d-1}\rho_{\pm}(r)j(r)dr,
\end{equation*}
where $\varphi_{\pm}(x)=\rho_{\pm}(|x|)$. Then $V_-(0)>V_+(0)$ if and only if $I_-(R)-I_+(R)>J_+(R)-J_-(R)$.
\end{prop}
\begin{proof}
Write
\begin{equation*}
V_-(0)-V_+(0)=\frac{1}{2a}\int_{\R^d}(\Dvm-\Dvp)j(|h|)dh.
\end{equation*}
Notice that
\begin{equation*}
\Dvm-\Dvp=\varphi_-(h)-2a+\varphi_-(-h)-\varphi_+(h)+2a-\varphi_+(-h)=2(\rho_-(|h|)-\rho_+(|h|)).
\end{equation*}
Hence
\begin{align*}
V_-(0)-V_+(0)
&=
\frac{1}{a}\int_{\R^d}(\rho_-(|h|)-\rho_+(|h|))j(|h|)dh=\frac{d\omega_d}{a}\int_{0}^{\infty}r^{d-1}
(\rho_-(r)-\rho_+(r))j(r)dh \\
& =\frac{d\omega_d}{a} \int_{0}^{R}r^{d-1}(\rho_-(r)-\rho_+(r))j(r)dh+
\frac{d\omega_d}{a} \int_{R}^{\infty}r^{d-1}(\rho_-(r)-\rho_+(r))j(r)dh\\
&=\frac{d\omega_d}{a} (I_-(R)-I_+(R))- \frac{d\omega_d}{a}(J_+(R)-J_-(R)).
\end{align*}
\end{proof}
Depending on the behaviour of $j$ (and thus of $\Phi$) and $\varphi$, we can derive estimates on the
integrals $I_{\pm}$, $J_{\pm}$. The examples below can be further studied for cases also including extra
slowly varying components.
\begin{example}
\label{IJexs}
{\rm
\hspace{100cm}
\begin{trivlist}
\item[\, (1)]
If $\Phi(u) \sim u^{\alpha/2}$ as $u \downarrow 0$, $\alpha \in (0,2)$, then Proposition \ref{prop2} gives
$j(r)\sim C(d,\alpha)r^{-d-\alpha}$. Thus, if $\varphi_\pm(x) \asymp |x|^{-2\kappa_\pm}$ as $|x| \to \infty$ with
$\kappa_\pm \in \left(0,\frac{d+\alpha}{2}\right)$ and $\kappa_+>\kappa_-$, then
\begin{equation*}
J_+(R)<C_+(R)C(d,\alpha)\frac{R^{-2\kappa_+-\alpha}}{2\kappa_++\alpha}, \qquad J_-(R)>
C_-(R)C(d,\alpha)\frac{R^{-2\kappa_--\alpha}}{2\kappa_-+\alpha},
\end{equation*}
with suitable constants $0<C_-(R)<1<C_+(R)$.

\medskip
\item[\, (2)]
If $\Phi(u) \sim u^{\alpha/2}$ as $u \to \infty$, $\alpha \in (0,2)$, by \cite[Th. $3.4$]{KSV} we have
$j(r)\sim C(d,\alpha)r^{-d-\alpha}$ as $r \downarrow 0$. Assuming
\begin{equation*}
\varphi(0)-\varphi_\pm(x)\sim \frac{\varphi(0)|x|^{2\kappa_\pm}}{1+|x|^{2\kappa_\pm}}, \quad |x| \to 0,
\end{equation*}
for some $\kappa_\pm \in \left(\frac{\alpha}{2},\frac{d+\alpha}{2}\right)$ with $\kappa_+>\kappa_-$,
we can estimate $I_{\pm}(R)$ to get
\begin{equation*}
I_-(R)>C_-(R)C(d,\alpha)\int_0^{R}\frac{r^{2\kappa_--\alpha-1}}{1+r^{2\kappa_-}}dr>\frac{C_-(R)C(d,\alpha)
R^{2\kappa_--\alpha}}{(1+R^{2\kappa_-})(2\kappa_--\alpha)},
\end{equation*}
and
\begin{equation*}
I_+(R)<C_+(R)C(d,\alpha)\int_0^{R}\frac{r^{2\kappa_+-\alpha-1}}{1+r^{2\kappa_+}}dr<\frac{C_+(R)C(d,\alpha)
R^{2\kappa_+-\alpha}}{2\kappa_+-\alpha},
\end{equation*}
with constants $0<C_-(R)<1<C_+(R)$. If $R<1$, we can also obtain tighter estimates by observing that
\begin{equation*}
\int_0^{R}\frac{r^{2\kappa_\pm-\alpha-1}}{1+r^{2\kappa_\pm}}dr=
\frac{R^{2\kappa_\pm-\alpha}}{2\kappa_\pm}\Phi_{L}
\left(-R^{2\kappa_\pm},1,\frac{2\kappa_\pm-\alpha}{2\kappa_\pm}\right),
\end{equation*}
with Lerch's function \cite{L, J74}
\begin{equation*}
\Phi_{L}(z,\sigma,\zeta)=\sum_{n=0}^{\infty}\frac{z^n}{(n+\zeta)^\sigma},
\quad |z|<1, \, \Re(\zeta)>0, \, \sigma \not \in -\N.
\end{equation*}
In particular, in this case we get
\begin{align*}
I_-(R)&>C_-(R)C(d,\alpha)\frac{R^{2\kappa_--\alpha}}{2\kappa_-}\Phi_{L}
\left(-R^{2\kappa_-},1,\frac{2\kappa_--\alpha}{2\kappa_-}\right),\\
I_+(R)&<C_+(R)C(d,\alpha)\frac{R^{2\kappa_+-\alpha}}{2\kappa_+}\Phi_{L}
\left(-R^{2\kappa_+},1,\frac{2\kappa_+-\alpha}{2\kappa_+}\right),
\end{align*}
with $0<C_-(R)<1<C_+(R)$.

\medskip
\item[\, (3)]
If $\Phi(u) \sim u^{\alpha/2}$, $\alpha \in (0,2)$, as $u \downarrow 0$, Proposition \ref{prop2}
gives  $j(r)\sim C(d,\alpha)r^{-d-\alpha}$. Thus with $\varphi_\pm(x) \asymp |x|^{\delta_\pm}
e^{-\eta_{\varphi_{\pm}}|x|^{\gamma_{\pm}}}$ as $|x| \to \infty$, for $\delta_{\pm} \in \R$ and
$\eta_{\varphi_{\pm}},\gamma_{\pm}>0$, we have that
\begin{align*}
J_+(R)&<\frac{C_+(R)C(d,\alpha)}{\gamma_+ \eta_{\varphi_+}^{(\delta_+-\alpha)/\gamma_+}} \,
\Gamma\Big(\frac{\delta_+-\alpha}{\gamma_+},\eta_{\varphi_+}R^{\gamma_+}\Big)  \\
J_-(R) & >\frac{C_-(R)C(d,\alpha)}{\gamma_- \eta_{\varphi_-}^{(\delta_--\alpha)/\gamma_-}} \,
\Gamma\Big(\frac{\delta_--\alpha}{\gamma_-},\eta_{\varphi_-}R^{\gamma_-}\Big),
\end{align*}
with constants $0<C_-(R)<1<C_+(R)$.

\medskip
\item[\, (4)]
If $\mu(t) \sim C_\mu t^{-1-\frac{\alpha}{2}}e^{-\eta_\mu t}$ as $t \to \infty$, with $\alpha \in (0,2]$
and $\eta_\mu>0$, we have by Proposition \ref{light} that $j(r)\sim C(C_\mu, d,\alpha,\eta_\mu)
r^{-\frac{d+\alpha+1}{2}}e^{-\sqrt{\eta_\mu}r}$. Thus with $\varphi_\pm \asymp |x|^{-2\kappa_\pm}$ as $|x|
\to \infty$, for $\kappa_{\pm} \in \left(0,\frac{d+\alpha}{2}\right)$ with $\kappa_+>\kappa_-$, we obtain
\begin{align*}
J_+(R) &< \frac{C_+(R)C(C_\mu,d,\alpha,\eta_\mu)}{\eta_\mu^{(d-\alpha-1-4\kappa_+)/4}} \,
\Gamma\left(\frac{d-\alpha-1-4\kappa_+}{4},\sqrt{\eta_\mu}R\right) \\
J_-(R) & > \frac{C_-(R)C(C_\mu,d,\alpha,\eta_\mu)}{\eta_\mu^{(d-\alpha-1-4\kappa_-)/4}} \,
\Gamma\left(\frac{d-\alpha-1-4\kappa_-}{4},\sqrt{\eta_\mu}R\right),
\end{align*}
with appropriate constants $0<C_-(R)<1<C_+(R)$.
\end{trivlist}
}
\end{example}

By using these estimates we may relate the behaviour of $V(0)$ with the sign of $V$ at infinity; we leave the
details to the reader to see how this depends on the characteristics of the operators chosen. For illustration,
suppose as in Example \ref{IJexs} (1) that $\Phi$ is regularly varying at $0$
and assume that $\varphi$ is polynomially bounded. We know that if $\kappa <d/2$, then $V(x)\sim |x|^{-\alpha}$.
Taking $R<1$, we have
\begin{equation*}
J_+(R)-J_-(R)<C(d,\alpha)\left(C_+^1(R)\frac{R^{-2\kappa_+-\alpha}}{2\kappa_++\alpha}
-C_-^1(R)\frac{R^{-2\kappa_--\alpha}}{2\kappa_-+\alpha}\right).
\end{equation*}
On the other hand, if we ask for a profile condition on $\varphi$, we also have
\begin{equation*}
I_-(R)-I_+(R)>C(d,\alpha)\left(\frac{C^2_-(R)R^{2\kappa_--\alpha}}{(1+R^{2\kappa_-})(2\kappa_--\alpha)}
-\frac{C^2_+(R)R^{2\kappa_+-\alpha}}{2\kappa_+-\alpha}\right).
\end{equation*}
Let now $V_+$ be positive at infinity and $V_-$ negative everywhere, with $\varphi_{+},\varphi_-$ polynomially
bounded at infinity and near zero, and recall that $\kappa_+>\kappa_-$. Consider
\begin{equation*}
\frac{C^2_-(R)R^{2\kappa_--\alpha}}{2\kappa_--\alpha}-\frac{C^2_+(R)R^{2\kappa_+-\alpha}}{2\kappa_+-\alpha}
\ge
C_+^1(R)\frac{R^{-2\kappa_+-\alpha}}{2\kappa_++\alpha}-C_-^1(R)
\frac{R^{-2\kappa_--\alpha}}{(1+R^{2\kappa_+})(2\kappa_-+\alpha)},
\end{equation*}
i.e.,
\begin{equation}\label{rel}
\frac{C^2_-(R)R^{2\kappa_-}}{2\kappa_--\alpha}+C_-^1(R)\frac{R^{-2\kappa_-}}{(1+R^{2\kappa_-})(2\kappa_-+\alpha)}
\ge C_+^1(R)\frac{R^{-2\kappa_+}}{2\kappa_++\alpha}+\frac{C^2_+(R)R^{2\kappa_+}}{2\kappa_+-\alpha}.
\end{equation}
Suppose, for instance, that $\alpha<1$ and $\kappa_+>\frac{d-1}{2}$. Then provided $\kappa_- \to \frac{\alpha}{2}$,
the left-hand side goes to infinity, thus for fixed $\kappa_+$ we find a value of $\kappa_-$ such that \eqref{rel}
holds and so $V_-(0)>V_+(0)$.

\begin{rmk}
{\rm
We can observe this phenomenon in further detail on the example of the potentials $V_{\kappa,\alpha}(x)$ explicitly
given in \eqref{eq:motiv_ex}. These examples also show that the conditions of Proposition \ref{intest} can be
satisfied. Choose $l=0$, i.e., $\delta=d$. A calculation gives
\begin{equation*}
|V_{\kappa,\alpha}(0)|=\frac{2^\alpha}{\Gamma\left(\frac{d}{2}\right)} \, \Gamma\left(\frac{d+\alpha}{2}\right)
\frac{\Gamma\left(\frac{\alpha}{2}+\kappa\right)}{\Gamma(\kappa)}.
\end{equation*}
Consider $\kappa_+>\kappa_-$. We want to show that $|V_{\kappa_+,\alpha}(0)|>|V_{\kappa_-,\alpha}(0)|$, or equivalently,
\begin{equation*}
\Gamma(\kappa_-)\Gamma\left(\frac{\alpha}{2}+\kappa_+\right)>\Gamma(\kappa_+)\Gamma\left(\frac{\alpha}{2}+\kappa_-\right).
\end{equation*}
Writing $q=\frac{\kappa_-}{2}$, $m=\frac{\kappa_-}{2}$, $n=\frac{\alpha+\kappa_-}{2}$ and $p=\kappa_+-\frac{\kappa_-}{2}$,
we have
\begin{equation}\label{eqGamma}
\Gamma(p+n)\Gamma\left(q+m\right)>\Gamma(p+q)\Gamma\left(m+n\right).
\end{equation}
Also,
\begin{equation*}
(p-m)(q-n)=-\frac{\alpha}{2}(\kappa_+-\kappa_-)<0.
\end{equation*}
Making use of \cite[Th. 1]{DAB}, which says that for $m,n,p,q \geq 0$ the relation $(p-m)(q-n)<0$ implies
\eqref{eqGamma}, the result follows.
}
\end{rmk}

It is interesting to see how this manifests in the case of massive relativistic Schr\"odinger operators.
\begin{prop}
Let Assumption \ref{eveq} hold for the massive relativistic Schr\"odinger operator for two different potentials
$V_+,V_-$ with $\varphi_+,\varphi_- \in \cZ_{C_1}(\R^d)$. Suppose both $\varphi_+,\varphi_-$ are radially symmetric,
decreasing, and $\varphi_+(0)=\varphi_-(0)=a$. If $L_{0,\alpha}\varphi_-(0)>L_{0,\alpha}\varphi_+(0)$, then there
exists a constant $m_*(\varphi_+,\varphi_-) > 0$ such that $V_-(0)>V_+(0)$ for all $m<m_*(\varphi_+,\varphi_-)$.
\end{prop}
\begin{proof}
Using \eqref{withG} and that $\varphi_+(0)=\varphi_-(0)=a$, we obtain
\begin{equation*}
aV_+(0)=L_{0,\alpha}\varphi_+(0)-G_{m,\alpha}\varphi_+(0)
\quad \mbox{and} \quad
aV_-(0)=L_{0,\alpha}\varphi_-(0)-G_{m,\alpha}\varphi_-(0),
\end{equation*}
that is
\begin{align*}
a(V_-(0)-V_+(0))&=L_{0,\alpha}\varphi_-(0)-L_{0,\alpha}\varphi_+(0)+G_{m,\alpha}\varphi_-(0)-G_{m,\alpha}\varphi_+(0)\\
&\ge L_{0,\alpha}\varphi_-(0)-L_{0,\alpha}\varphi_+(0)-|G_{m,\alpha}\varphi_-(0)|-|G_{m,\alpha}\varphi_+(0)|.
\end{align*}
By Lemma \ref{Glem} we know that
\begin{equation*}
|G_{m,\alpha}\varphi_+(0)| \le 2ma \quad \mbox{and} \quad |G_{m,\alpha}\varphi_-(0)| \le 2ma,
\end{equation*}
and hence
\begin{align*}
a(V_+(0)-V_-(0))&\ge L_{0,\alpha}\varphi_-(0)-L_{0,\alpha}\varphi_+(0)-4ma.
\end{align*}
Thus with $m_*(\varphi_+,\varphi_-)=\frac{1}{4a}(L_{0,\alpha}\varphi_-(0)-L_{0,\alpha}\varphi_+(0))$ we have $V_-(0)-V_+(0)>0$
for every $m<m_*(\varphi_+,\varphi_-)$.
\end{proof}

\section{Appendix}
In this appendix we provide a proof of Lemma \ref{genC1}.

\begin{proof}
The statement is obvious for $C<C_1$, thus we choose $C>C_1$. Consider the balls $B_{C_1}(0)$, $B_C(0)$ and the closed
ring $A_{C_1,C}(0)=\overline{B_{C_1}^c(0)\cap B_C(0)}$. Fix $\varepsilon,\delta>0$ such that $\varepsilon<C_1$, $\delta
<1-C$, denote $S_{C_1-\varepsilon}(0)=\partial B_{C_1-\varepsilon}(0)$, and define the family of open sets
\begin{equation*}
\mathcal{A}=\{B_{\varepsilon+C-C_1+\delta}(x): \ x \in S_{C_1-\varepsilon}\}.
\end{equation*}
It is direct to check that $\mathcal{A}$ is a covering of the compact set $A_{C_1,C}(0)$, hence there exists a finite
sub-covering $\widetilde{\mathcal{A}}$. Consider the finite covering given by
\begin{equation*}
\mathcal{B}_{1}(0)=\{B_{\varepsilon+C-C_1+\delta}(x): \ B_{\varepsilon+C-C_1+\delta}(x) \in \widetilde{\mathcal{A}} \
\mbox{ or } \ B_{\varepsilon+C-C_1+\delta}(-x) \in \widetilde{\mathcal{A}}\}.
\end{equation*}
Define
\begin{equation*}
\mathcal{B}_{|x|}(x)=\{x+B_{|x|(\varepsilon+C-C_1+\delta)}(|x|y): \ B_{\varepsilon+C-C_1+\delta}(y) \in \mathcal{B}_{1}(0)\},
\end{equation*}
which is a finite covering of $A_{C_1|x|,C|x|}(0)=\overline{B_{C_1|x|}^c(x)\cap B_{C|x|}(x)}$.

Choose $x \in \R^d$ such that $(1-C)|x|>M_{f}$, and $h \in B_{C|x|}(0)$. If $h \in B_{C_1|x|}(0)$ we already have
$|\DDv|\le L_f(x)|h|^2$, thus we suppose that $C_1|x|\le |h|<C|x|$. There exists a ball $B^1 \in \mathcal{B}_{|x|}(x)$
such that $x+h_1 \in B^+$, we denote its center by $y_1$. Then there exist $h_2 \in \R^d$ such that $x+h_1=y_++h_2$,
and $h_3 \in \R^d$ such that $y_+=x+h_3$. By construction of $\mathcal{B}_{|x|}(x)$ there exists a ball $B^- \in
\mathcal{B}_{|x|}(x)$ centered in $y_-=x-h_3$ such that $x-h_1 \in B^-$ and $x-h_1=y_--h_2$. Now consider the centered
difference
\begin{align*}
|f(x+h_1)-2f(x)+f(x-h_1)|&\le|f(y_++h_2)-2f(y_+)+f(y_+-h_2)|\\
&\quad +|-f(y_+-h_2)+2f(y_+)-2f(x)+2f(y_-)-f(y_-+h_2)|\\
& \quad +|f(y_-+h_2)-2f(y_-)+f(y_--h_2)|\\
&:=T_1+T_2+T_3.
\end{align*}
We estimate the three terms $T_1, T_2, T_3$.

Notice that $y_-, y_+ \in B_{C_1|x|}(x)$, in particular $y_-,y_+ \in B_{(1-C)|x|}^c(0)\subset B_{M_f}^c(0)$. Also,
notice that the balls $B^+,B^-$ have radius $(\varepsilon+\delta+C-C_1)|x|$, while we know that $|y| \ge
(1-C_1+\varepsilon)|x|$. We show that $\varepsilon$ and $\delta$ can be chosen in such a way that  $\varepsilon
+\delta +C-C_1<C_1(1-C_1+\varepsilon)$. To do that, rewrite the inequality as
	 \begin{equation*}
	 (1-C_1)\varepsilon+\delta<2C_1-C_1^2-C.
	 \end{equation*}
To find a value $C>C_1$ such that this inequality is satisfied for some $\varepsilon$ and $\delta$, it is necessary
that the right-hand side is positive, i.e.,
	 \begin{equation*}
	 C<2C_1-C_1^2.
	 \end{equation*}
Note that since $C_1<1$, we have $2C_1-C_1^2>C_1$.

We choose $C=rC_1$ with a suitable value $r>1$, so that
$1<r<2-C_1$ needs to be satisfied.
For $C=rC_1$ we may choose
\begin{equation*}
\varepsilon<\frac{2C_1-C_1^2-C}{1-C_1}
\end{equation*}
and thus
\begin{equation*}
\delta<2C_1-C_1^2-C-(1-C_1)\varepsilon.
\end{equation*}
Since with $y_{\pm}+h_2 \in B^\pm$ we have $h_2 \in B_{C_1|y_{\pm}|}(0)$ as $C_1|y_{\pm}|$ is larger than the
radius of $B^\pm$, due to $y_{\pm}\in B_{M_f}^c(0)$ we have
\begin{equation*}
T_1\le L_f(y_+)|h_2|^2 \quad \mbox{ and } \quad T_3\le L_f(y_-)|h_2|^2.
\end{equation*}

Next we turn to the term $T_2$. For $d=1$ it is clear that $y^\pm$ are the extrema of the interval
$B_{(C_1-\varepsilon)|x|}(x)$ and that if $y^++h_2 \not \in B_{C_1|x|}(x)$, then $y^+-h_2 \in B_{C_1|x|}(x)$ provided
$\varepsilon<\frac{2C_1-C}{2}$ and $\delta<2C_1-2\varepsilon-C$. The same holds for $y^-$.

For $d \ge 2$ consider the hyperplane $\pi$ tangent to $\partial B_{(C_1-\varepsilon)|x|}(x)$ in the point $y^+$. This
plane cuts the ball $B^+$ in half, moreover, it is a secant plane for $B_{C_1|x|}(x)$ and the intersection $\pi \cap
B_{C_1|x|}(x)$ is a ball $B \subset \R^{d-1}$. We determine the radius of this ball. To do this, consider a rigid shift
moving the ball $B_{C_1|x|}(x)$ in the origin and sets $y^+$ in the $x_d$ axis. In this way $\pi$ is parallel to the
hyperplane $x_d=0$. After this rigid shift it suffices to work in the two-dimensional case by projection to
$\{(x_1,\dots,x_d) \in \R^d, \ x_i=0, \ i=3,\dots,d\}$. By Euclid's second theorem in geometry, denoting by $R_B$ the
radius of the $d-1$-dimensional ball $B$, we have
\begin{equation*}
R_B^2=(2C_1-\varepsilon)\varepsilon|x|^2
\end{equation*}
and then $R_B=\sqrt{(2C_1-\varepsilon)\varepsilon}|x|$. We need this radius to be larger than the radius of $B^+$. To
achieve this, denote $\Delta = \delta+C-C_1$ to obtain
\begin{equation*}
\sqrt{(2C_1-\varepsilon)\varepsilon}>\varepsilon+\Delta.
\end{equation*}
Since $\Delta>0$, the inequality becomes
$2\varepsilon^2+2(\Delta-C_1)\varepsilon+\Delta^2<0$,
which admits solutions if and only if
$\Delta^2+2\Delta C_1-C_1^2<0$.
Since $\Delta>0$, we only have to require
$\Delta<(\sqrt{2}-1)C_1$,
that is
\begin{equation*}
\delta<(\sqrt{2}-1)C_1-(C-C_1)=(\sqrt{2}-r)C_1
\end{equation*}
implying $r<\sqrt{2}$.
	
Hence under the conditions $r<\sqrt{2}$ and $\delta<(\sqrt{2}-1)C_1-(C-C_1)=(\sqrt{2}-r)C_1$ we have
\begin{equation*}
\frac{C_1-\Delta-\sqrt{-\Delta^2-2\Delta C_1+C_1^2}}{2}<\varepsilon<
\frac{C_1-\Delta+\sqrt{-\Delta^2-2\Delta C_1+C_1^2}}{2}.
\end{equation*}
Notice that
\begin{equation*}
\frac{C_1-\Delta+\sqrt{-\Delta^2-2\Delta C_1+C_1^2}}{2}> \frac{C_1}{2}
\end{equation*}
if and only if
$\Delta< \frac{\sqrt{3}-1}{2}C_1$
thus, arguing as before, we get $r<\frac{\sqrt{3}+1}{2}$ and
$\delta< \frac{\sqrt{3}+1-2r}{2}C_1$,
and so we may require $\varepsilon<\frac{C_1}{2}$ to satisfy the upper bound. To get a lower bound,
observe that with $C=rC_1$ we have
\begin{equation*}
\frac{2C_1-C_1^2-C}{1-C_1}<\frac{C_1}{2}
\end{equation*}
if and only if
$r>\frac{3-C_1}{2}$.
Hence by choosing $r\le \frac{3-C_1}{2}$, we have
\begin{equation*}
\frac{2C_1-C_1^2-C}{1-C_1}\ge \frac{C_1}{2}
\quad \mbox{and} \quad 	
\frac{C_1-\Delta-\sqrt{-\Delta^2-2\Delta C_1+C_1^2}}{2}<\frac{C_1-(C-C_1)}{2}<\frac{C_1}{2}.
\end{equation*}

In summary, we are led to choose
\begin{equation}\label{rdef}
r=\begin{cases}
\frac{\sqrt{3}+3}{4} & C_1 \in \left(0,\frac{3-\sqrt{3}}{2}\right) \vspace{0.2cm} \\
\frac{3-C_1}{2} & C_1 \in \left[\frac{3-\sqrt{3}}{2},1\right],
\end{cases}
\qquad \varepsilon \in \left(\frac{2-r}{2}C_1,\frac{C_1}{2}\right),
\end{equation}
and
\begin{equation*}
0<\delta<\min\Big\{1-C,2C_1-C_1^2-C-(1-C_1)\varepsilon,(\sqrt{2}-r)C_1,\frac{\sqrt{3}+1-2r}{2}C_1\Big\}.
\end{equation*}
With these choices it then follows that if $y_++h_2\not \in B_{C_1|x|}(x)$, then $y_+-h_2 \in B_{C_1|x|}(x)$ and
also $y_-+h_2 \in B_{C_1|x|}(x)$. Thus we have
\begin{align*}
T_2&=|-f(x+h_3-h_2)+2f(x+h_3)-2f(x)+2f(x-h_3)-f(x-h_3+h_2)|\\
&\le |-f(x+h_3-h_2)+2f(x)-f(x-h_3+h_2)|\\
& \qquad +2|f(x+h_3)-2f(x)+f(x-h_3)|\\
&\le L_f(x)|h_3-h_2|^2+2L_f(x)|h_3|^2
\end{align*}
Since $x+h_3-h_2 \in B_{C_1|x|}(x)$, we have $h_3-h_2 \in B_{C_1|x|}(0)$ and then, due to $|h_1|>C_1|x|$, we have
$|h_3-h_2|^2<|h_1|^2$. With the same argument, we can also bound $|h_3|^2$. Furthermore, we have $h_2 \in
B_{C_1|y^+|}(0)$ and then, using $|y^+|\le (1+C_1-\varepsilon)|x|$ and $|x|<\frac{|h_1|}{C_1}$, we obtain $|h_2|^2
\le \left(\frac{1+C_1-\varepsilon}{C_1}\right)^2|h_1|^2$. Thus we get
\begin{equation*}
|f(x+h_1)-2f(x)+f(x-h_1)|\le\left( \left(\frac{1+C_1}{C_1}\right)^2(L_f(y^+)+L_f(y^-))+3L_f(x)\right)|h_1|^2.
\end{equation*}
Given that $y^\pm$ depend on $x$, we can define
\begin{equation*}
\left(\frac{1+C_1}{C_1}\right)^2(L_f(y^+)+L_f(y^-))+3L_f(x)=L_{f,C}(x).
\end{equation*}
Thus in case $L_f(x)\le C_{g,1}g(x)$, we have
\begin{equation*}
L_{f,C}(x)\le C_{g,1}\left( \left(\frac{1+C_1}{C_1}\right)^2(g(|y^+|)+g(|y^-|))+3g(|x|)\right).
\end{equation*}
However, $|y^\pm| \ge (1-C_1+\varepsilon)|x|\ge (1-C_1)|x|$, and since $g$ is decreasing, we obtain
\begin{equation*}
L_{f,C}(x)\le C_{g,1}\left( \left(2\frac{1+C_1}{C_1}\right)^2+3\right)g((1-C_1)|x|).
\end{equation*}
This shows the lemma for $C=rC_1$, when $r$ is defined by \eqref{rdef}.

Next we show that the lemma extends to all $C \in (C_1,1)$. If $C_1<\frac{3-\sqrt{3}}{2}$, we can define $r_1=
\frac{\sqrt{3}+3}{4}$, $C_2=r_1C_1$, $C_3=r_1C_2=r_1^nC_1$ and so on. Thus, since $r_1>1$, there exists $N\in \N$
such that $C_N\ge \frac{3-\sqrt{3}}{2}$. Hence, starting from $C_1<\frac{3-\sqrt{3}}{2}$, we can extend the result
to $C \in \left(C_1,\frac{3-\sqrt{3}}{2}\right]$. If $C_1 \ge  \frac{3-\sqrt{3}}{2}$, write $r_1=\frac{3-C_1}{2}$,
$C_2=r_1C_1$. Notice that setting $r_2=\frac{3-C_2}{2}$ and $\widetilde{C}_3=r_2C_2$, we get $C_2>C_1$ since $r_1>1$,
and the property goes on to hold also for $C_3=r_2C_1$. We can then proceed inductively by putting
$$
 C_{n+1}=r_nC_1, \; n \ge 1, \quad \mbox{and} \quad r_1=\frac{3-C_1}{2}, \; r_n=\frac{3-C_n}{2}, \; n \ge 2,
$$
giving
$$
r_1=\frac{3-C_1}{2}, \quad r_n=\frac{3-r_{n-1}C_1}{2}, \; n \ge 2.
$$
The recursion has the fixed point
\begin{equation*}
r^*=\frac{3}{2+C_1}.
\end{equation*}
Then for $q_n=r_n-r^*$ we have
$$
q_1=\frac{C_1-C_1^2}{2(2+C_1)}, \quad q_n=-\frac{C_1}{2}q_{n-1}, \; n \ge 2,
$$
and
\begin{equation*}
q_n=\frac{C_1-C_1^2}{2(2+C_1)}\left(-\frac{C_1}{2}\right)^{n-1}.
\end{equation*}
Since $C_1/2<1$ we have $q_n \to 0$ and $r_n \to r^*$. Moreover, since $r_2>r^*$, we can extend to
$C \in (0,\frac{3}{2+C_1}C_1]$.
Thus consider
$$
C_1+\delta_1=C_2 \le \frac{3}{2+C_1}C_1, \quad C_n+\delta_n=C_{n+1}\le \frac{3}{2+C_n}C_n, \; n \ge 2,
$$	
with a sequence $\delta_n>0$ such that $\sum_{n=1}^{\infty}\delta_n=1-C_1$. We have
$$
\delta_1\le \frac{1-C_1}{2+C_1}C_1, \quad \delta_n\le \frac{1-C_n}{2+C_n}C_n, \; n \ge 2.
$$
Solving it for $\delta_n=kr^n$ with some $k>0$ and $r \in (0,1)$, in the new variable $q=kr$
we have for $r<1$
$$
qr^{n-1}\le \frac{1-C_n}{2+C_n}C_n, \quad \frac{q}{1-r}=1-C_1,
$$
leading to
	 \begin{equation}\label{ineq}
	 (1-C_1)(1-r)r^{n-1}< \frac{1-C_n}{2+C_n}C_n.
	 \end{equation}
and
\begin{align*} C_n&=C_1+\sum_{i=1}^{n-1}\delta_i
=C_1+q\frac{1-r^{n-1}}{1-r}
=1-(1-C_1)r^{n-1}.
	 \end{align*}
Thus inequality \eqref{ineq} becomes
\begin{equation*}
(1-r)(3-(1-C_1)r^{n-1})< 1-(1-C_1)r^{n-1},
\end{equation*}
or equivalently,
\begin{equation*}
f_n(r) := 2-3r+(1-C_1)r^{n}<0
\end{equation*}
We see that $f_n(r)$ is a continuous function for each $n$ and since $f_n(1)=-C_1<0$, there exists $r_0 \in (0,1)$ such
that $f_n(r_0)<0$. Then it is direct to show that $f_n(r)\le f_1(r)$, thus fixing $r_0=\frac{4+C_1}{2(2+C_1)}$ gives
$f_n(r_0)<0$ for all $n \in \N$. With this $r_0$, consider $q_0=(1-C_1)(1-r_0)$ and $k_0=\frac{q_0}{r_0}$. With these
choices $r_0,q_0,k_0$ we find $\delta_n=k_0r_0^n$ such that $\sum_{n=1}^{\infty}\delta_n=1-C_1$, extending the result
to the full set $(0,1)$ and completing the proof.
\end{proof}


\begin{thebibliography}{99}

\bibitem{AB}
L. Acu\~na Valverde, R. Ba\~nuelos: Heat content and small time asymptotics for Schr\"odinger operators on $\R^d$,
\emph{Potential Anal.} \textbf{42} (2015), 457-482


\bibitem{A70}
S. Agmon:
Lower bounds for solutions of Schr\"odinger equations,  \emph{J. Analyse Math.} \textbf{23} (1970), 1-25

\bibitem{ABG}
W.O. Amrein, A.M. Berthier, V. Georgescu:
Lower bounds for zero energy eigenfunctions of Schr\"odinger operators, \emph{Helv. Phys. Acta} \textbf{57} (1984),
301-306

\bibitem{A09}
A. Arai:
Necessary and sufficient conditions for a Hamiltonian with discrete eigenvalues to have time operators,
\emph{ Lett. Math. Phys.} \textbf{87}, 67--80, 2009,

\bibitem{AH}
A. Arai, F. Hiroshima:
Ultra-weak time operators of Schr\"odinger operators, \emph{ Ann. Henri Poincar\'e} \textbf{18}, 2995--3033, 2017

\bibitem{BBAC}
F. Bardou, J.P. Bouchaud, A. Aspect, C. Cohen-Tannoudji:
\emph{L\'evy Statistics and Laser Cooling}, Cambridge University Press, 2003


\bibitem{BY90}
R. Benguria, C. Yarur:
Sharp condition on the decay of the potential for the absence of a zero-energy ground state of the Schr\"odinger
equation, \emph{ J. Phys. A} \textbf{23} (1990), 1513-1518


\bibitem{BGT}
N.H. Bingham, C.M. Goldie, J. Teugels: \emph{Regular Variation}, Cambridge University Press, 1989



\bibitem{BL19a}
A. Biswas, J. L\H{o}rinczi:
Universal constraints on the location of extrema of eigenfunctions of non-local Schr\"odinger operators,
\emph{J. Diff. Equations} \textbf{267} (2019), 267-306


\bibitem{BL19b}
A. Biswas, J. L\H{o}rinczi:
Maximum principles and Aleksandrov-Bakelman-Pucci type estimates for non-local Schr\"{o}dinger equations with exterior
conditions, \emph{SIAM J. Math. Anal.} \textbf{51} (2019), 1543--1581


\bibitem{Betal}
K. Bogdan et al:
\emph{Potential Analysis of Stable processes and its Extensions} (ed. P. Graczyk, A. St\'os),
Lecture Notes in Mathematics 1980, Springer, Berlin, 2009

\bibitem{BB99}
K. Bogdan, T. Byczkowski:
Potential theory for the $\alpha$-stable Schr\"odinger operator on bounded Lipschitz domains,
\emph{ Studia Math.} \textbf{133} (1999), 53--92

\bibitem{BB00}
K. Bogdan, T. Byczkowski:
Potential theory of Schr\"odinger operator based on fractional Laplacian, \emph{Probab. Math. Statist.}
\textbf{20} (2000), 293--335

\bibitem{CMS}
R. Carmona, W.C. Masters, B. Simon:
\emph{Relativistic Schr\"odinger operators: asymptotic behaviour of the eigenfunctions}, \emph{J. Funct. Anal.}
\textbf{91} 1990, 117--142

\bibitem{CK01}
M. Christ, A. Kiselev:
One-dimensional Schr\"odinger operators with slowly decaying potentials: spectra and asymptotics, notes to
\emph{Workshop on Oscillatory Integrals and Dispersive Equations, IPAM}, 2001



\bibitem{DL}
I. Daubechies, E.H. Lieb: One-electron relativistic molecules with Coulomb interaction,
\emph{Commun. Math. Phys.} \textbf{90} (1983), 497-510


\bibitem{DLN}
C.-S. Deng, W. Liu, E. Nane:
Finite time blowup of solutions to SPDEs with Bernstein functions of the Laplacian, \emph{arXiv:2001.00320}


\bibitem{DS07}
S. Denisov, A. Kiselev:
Spectral properties of Schr\"odinger operators with decaying potentials, in: \emph{Proceedings of Symposia
in Pure Mathematics (B. Simon Festschrift)}, Vol. \textbf{76}, AMS, 2007

\bibitem{DS09}
J. Derezi\'nski, E. Skibsted:
Quantum scattering at low energies, \emph{J. Funct. Anal.} \textbf{257} (2009), 1828-1920


\bibitem{DAB}
S.S. Dragomir, R.P. Agarwal, N.S. Barnett:
Inequalities for Beta and Gamma functions via some classical and new integral inequalities,
\emph{J. lnequal. Appl.} \textbf{5} (2000), 103-165


\bibitem{DL}
S.O. Durugo, J. L\H orinczi:
Spectral properties of the massless relativistic quartic oscillator, \emph{J. Diff. Equations} \textbf{264}
(2018), 3775-3809


\bibitem{EK}
M.S.P. Eastham, H. Kalf:
\emph{Schr\"odinger-Type Operators with Continuous Spectra}, Pitman, 1982

\bibitem{FF15}
M.M. Fall, V. Felli:
Unique continuation properties for relativistic Schr\"odinger operators with a singular potential,
\emph{Discrete Contin. Dyn. Syst.} \textbf{35} (2015), 5827-5867

\bibitem{FLS}
R.L. Frank, E.H. Lieb, R. Seiringer:
Hardy-Lieb-Thirring inequalities for fractional Schr\"odinger operators, \emph{J. Amer. Math. Soc.} \textbf{21}
(2008), 925-950

\bibitem{FS04}
S. Fournais, E. Skibsted:
Zero energy asymptotics of the resolvent for a class of slowly decaying potentials, \emph{Math. Z.}
\textbf{248} (2004), 593-633


\bibitem{GR}
I. Gradshteyn, I. Ryzhik:
\emph{Table of Integrals, Series, and Products}, Elsevier/Academic Press, 7th ed., 2007


\bibitem{H77}
I.W.~Herbst:
Spectral theory of the operator $(p^2+m^2)^{1/2} - Ze^2/r$,
\emph{Commun. Math. Phys.} \textbf{53} (1977), 285--294

\bibitem{HS}
I.W. Herbst, E. Skibsted:
Decay of eigenfunctions of elliptic PDE’s, I, \emph{Adv. Math.} \textbf{270} (2015), 138-180


\bibitem{HIL12}
F. Hiroshima, T. Ichinose, J. L\H{o}rinczi:
Path integral representation for Schr\"odinger operators with Bernstein functions of the Laplacian,
\emph{Rev. Math. Phys.} \textbf{24} (2012), 1250013


\bibitem{HIL13}
F. Hiroshima, T. Ichinose, J. L\H orinczi: Probabilistic representation and fall-off of bound states
of relativistic Schr\"odinger operators with spin 1/2, \emph{Publ. Res. Inst. Math. Sci.} \textbf{49} (2013),
189-214

\bibitem{HIL17}
F. Hiroshima, T. Ichinose, J. L\H orinczi: Kato's inequality for magnetic relativistic Schr\"odinger
operators, \emph{Publ. RIMS} \textbf{53} (2017), 79-117

\bibitem{HL12}
F. Hiroshima, J. L\H orinczi:
Lieb-Thirring bound for Schr\"odinger operators with Bernstein functions of the Laplacian,
\emph{Commun. Stoch. Anal.} \textbf{6} (2012), 589-602

\bibitem{HL14}
F. Hiroshima, J. L\H{o}rinczi:
The spectrum of non-local discrete Schr\"odinger operators with a $\delta$-potential, \emph{Pacific J. Math.
Ind.} \textbf{6} (2014), 1-6



\bibitem{JW}
N. Jacob, F.-Y. Wang:
Higher order eigenvalues for non-local Schr\"odinger operators, \emph{Comm. Pure Appl. Anal.} \textbf{17} (2018),
191--208


\bibitem{JL18}
C. J\"ah, J. L\H{o}rinczi:
Eigenvalues at the continuum edge for fractional Schr\"odinger operators, work in progress, preprint, 2020

\bibitem{JK79}
A. Jensen, T. Kato:
Spectral properties of Schr\"odinger operators and time-decay of the wave functions, \emph{Duke Math. J.}
\textbf{46} (1979), 583-611

\bibitem{JRF}
A.S. Jensen, K. Riisager, D.V. Fedorov:
Structure and reactions of quantum halos, \emph{Rev. Mod. Phys.} \textbf{76} (2004), 215-261

\bibitem{J74}
B.R. Johnson:
Generalized Lerch zeta function, \emph{Pacific J. Math.} \textbf{53} (1974), 189-193


\bibitem{KKM}
K. Kaleta, M. Kwa\'snicki, J. Ma\l ecki:
Asymptotic estimate of eigenvalues of pseudo-differential operators in an interval, \emph{J. Math. Anal. Appl.}
\textbf{439} (2016), 896--924

\bibitem{KKL}
K. Kaleta, M. Kwa\'snicki,  J. L\H{o}rinczi:
Contractivity and ground state domination properties for non-local Schr\"odinger operators, \emph{J. Spectr. Theory}
\textbf{8} (2018), 165--189

\bibitem{KL12}
K. Kaleta, J. L\H{o}rinczi: Fractional $P(\phi)_1$-processes and Gibbs measures,
\emph{Stoch. Proc. Appl.} \textbf{122} (2012), 3580-3617

\bibitem{KL15}
K. Kaleta, J. L\H{o}rinczi:
Pointwise estimates of the eigenfunctions and intrinsic ultracontractivity-type properties of Feynman-Kac semigroups
for a class of L\'evy processes, \emph{Ann. Probab.} \textbf{43} (2015), 1350-1398

\bibitem{KL16}
K. Kaleta, J. L\H{o}rinczi:
Transition in the decay rates of stationary distributions of L\'evy motion in an energy landscape,
\emph{Phys. Rev. E} \textbf{93} (2016), 022135

\bibitem{KL17}
K. Kaleta, J. L\H{o}rinczi: Fall-off of eigenfunctions for non-local Schr\"odinger operators with decaying potentials,
\emph{Potential Anal.} \textbf{46} (2017), 647-688

\bibitem{KL19}
K. Kaleta, J. L\H{o}rinczi:
Zero-energy bound state decay for non-local Schr\"odinger operators, \emph{Commun. Math. Phys.} \textbf{374} (2020)
2151-2191

\bibitem{KPP18}
K. Kaleta, K. Pietruska-Pa{\l}uba:
Lifschitz singularity for subordinate Brownian motions in presence of the Poissonian potential on the Sierpi\'ski gasket
\emph{Stoch. Proc. Appl.} \textbf{128} (2018), 3897--3939

\bibitem{KPP20}
K. Kaleta, K. Pietruska-Pa{\l}uba:
The quenched asymptotics for nonlocal Schr\"dinger operators with Poissonian potentials, \emph{Potential Anal.} \textbf{52}
(2020), 161--202



\bibitem{K89}
C.E. Kenig:
Restriction theorems, Carleman estimates, uniform Sobolev inequalities and unique continuation, in: \emph{Harmonic
Analysis and Partial Differential Equations (El Escorial, 1987)}, in: Harmonic Analysis and Partial Differential
Equations (J. Garc\'ia-Cuerva, ed.), Lecture Notes in Mathematics \textbf{1384}, Springer, 1989, pp. 69-90

\bibitem{KN00}
C.E. Kenig, N. Nadirashvili:
A counterexample in unique continuation, \emph{Math. Res. Lett.} \textbf{7} (2000), 625-630

\bibitem{KSV}
P. Kim, R. Song, Z.Vondra\v cek:
Potential theory for subordinate Brownian motions revisited, in: \emph{Stochastic Analysis and Applications to Finance.
Essays in Honour of Jia-an Yan}, T. Zhang, X.Zhou (eds.), World Scientific, 2012, pp. 243-290



\bibitem{BG}
R. Klages, G. Radons, I.M. Sokolov:
\emph{Anomalous Transport: Foundations and Applications}, Wiley, 2008

\bibitem{KS}
M. Klaus, B. Simon:
Coupling constant thresholds in nonrelativistic quantum mechanics. I. Short-range two-body case,  \emph{Ann. Physics} \textbf{130}
(1980),  251-281; II. Two-cluster thresholds in $N$-body systems, \emph{Commun. Math. Phys.} \textbf{78} (1980), 153-168


\bibitem{KT02}
H. Koch, D. T\u ataru: Sharp counterexamples in unique continuation for second order elliptic equations,
\emph{J. reine angew. Math.} \textbf{542} (2002), 133-146

\bibitem{KMV}
Yu. Kondratiev, S. Molchanov, B. Vainberg:
Spectral analysis of non-local Schr\"odinger operators, \emph{J. Funct. Anal.} \textbf{273} (2017), 1020-1048

\bibitem{KS19}
F. K\"uhn, R. Schilling:
On the domain of fractional Laplacians and related generators of Feller processes, \emph{J. Funct. Anal.}
\textbf{276} (2019), 2397-2439

\bibitem{KM}
M. Kwa\'snicki, J. Mucha:
Extension technique for complete Bernstein functions of the Laplace operator, \emph{J. Evol. Equ.} \textbf{18} (2018),
1341--1379

\bibitem{L}
M. Lerch:
Note sur la fonction $\mathfrak K(w,x,s)= \sum_{k=0}^\infty \frac{e^{2k\pi i x}}{(w+k)^s}$, \emph{Acta Math.}
\textbf{11} (1887), 19-24

\bibitem{L81}
E.H. Lieb:
Thomas-Fermi and related theories of atoms and molecules, \emph{Rev. Mod. Phys.} \textbf{53} (1981), 603-641

\bibitem{LS10}
E.H. Lieb, R. Seiringer:
\emph{The Stability of Matter in Quantum Mechanics}, Cambridge University Press, 2010

\bibitem{L15}
S. Longhi:
Fractional Schr\"odinger equation in optics, \emph{Opt. Lett.} \textbf{40} (2015), 1117–-1120

\bibitem{LHB}
J. L\H{o}rinczi, F. Hiroshima, V. Betz:
\emph{Feynman-Kac-Type Theorems and Gibbs Measures on Path Space. With Applications to Rigorous Quantum
Field Theory}, de Gruyter Studies in Mathematics \textbf{34}, Walter de Gruyter, 2011; 2nd rev. exp. ed.,
vol. 1, 2020

\bibitem{LM}
J. L\H{o}rinczi, J. Ma{\l}ecki: Spectral properties of the massless relativistic harmonic
oscillator, \emph{J. Diff. Equations} \textbf{253} (2012), 2846-2871

\bibitem{LS17}
J. L\H{o}rinczi, I. Sasaki:
Embedded eigenvalues and Neumann-Wigner potentials for relativistic Schr\"odinger operators, \emph{J. Funct. Anal.}
\textbf{273} (2017), 1548-1575

\bibitem{LS}
J. L\H{o}rinczi, I. Sasaki:
Absence of embedded eigenvalues for a class of non-local Schr\"odinger operators, preprint (2020)

\bibitem{M06}
M. Maceda:
On the Birman-Schwinger principle applied to $\sqrt{-\Delta + m^2} - m$, \emph{J. Math. Phys.} \textbf{47} (2006),
033506

\bibitem{M27}
M.A. Marchaud:
Sur les d\'eriv\'ees et sur les diff\'erences des fonctions de variables r\'eelles, Th\`eses de l'entre-deux-guerres, 1927
(via numdam.org)

\bibitem{N94}
S. Nakamura:
Low-energy asymptotics for Schr\"odinger operators with slowly decreasing potentials, \emph{Comm. Math. Phys.}
\textbf{161} (1994), 63-76


\bibitem{V}
E. Di Nezza, G. Palatucci, E. Valdinoci:
Hitchhiker’s guide to the fractional Sobolev spaces, \emph{Bull. Sci. Math.} \textbf{136} (2012), 521-573



\bibitem{R87}
A.G. Ramm:
Sufficient conditions for zero not to be an eigenvalue of the Schr\"odinger operator, \emph{J. Math. Phys.}
\textbf{28} (1987), 1341-1343; \emph{J. Math. Phys.} \textbf{29} (1988), 1431-1432


\bibitem{RS3}
M. Reed, B. Simon:
\emph{Methods of Modern Mathematical Physics}, vols. 3-4, Academic Press, 1979

\bibitem{RU16}
S. Richard, T. Umeda:
Low energy spectral and scattering theory for relativistic Schr\"odinger operators,
\emph{Hokkaido Math.~J.}, \textbf{45} (2016), 141-179

\bibitem{RSV}
L. Roncal, D. Stan, L. Vega:
Carleman type inequalities for fractional relativistic operators, \emph{arXiv:1909.10065}, 2019


\bibitem{R15}
A. R\"uland:
Unique continuation for fractional Schr\"odinger equations with rough potentials, \emph{Comm. Part. Diff. Eqs.}
\textbf{40} (2015), 77-114

\bibitem{R17}
A. R\"uland:
On quantitative unique continuation properties of fractional Schr\"odinger equations: Doubling, vanishing order
and nodal domain estimates, \emph{Trans. Amer. Math. Soc.} \textbf{369} (2017), 2311-2362

\bibitem{RW}
A. R\"uland, J-N. Wang:
On the fractional Landis conjecture, \emph{J. Funct. Anal.} \textbf{277} (2019), 3236-3270




\bibitem{Ryz}
M. Ryznar:
Estimates of Green function for relativistic $\alpha$-stable process, \emph{ Potential Anal.} \textbf{17} (2002),
1-23


\bibitem{Sch07}
W. Schlag:
Dispersive estimates for Schr\"odinger operators: A survey,
in: \emph{Mathematical Aspects of Nonlinear Dispersive Equations} (J. Bourgain, C.E. Kenig, S. Klainerman, eds.),
Princeton University Press, 2007, pp. 255--285

\bibitem{S14}
I. Seo:
On unique continuation for Schr\"odinger operators of fractional and higher orders, \emph{Math. Nachr.} \textbf{287}
(2014), 699-703

\bibitem{S15}
I. Seo:
Unique continuation for fractional Schr\"odinger operators in three and higher dimensions, \emph{Proc. AMS} \textbf{143}
(2015), 1661-1664

\bibitem{SW}
E. Skibsted, X.P. Wang:
Two-body threshold spectral analysis, the critical case, \emph{J. Funct. Anal.} \textbf{260} (2011), 1766-1794


\bibitem{S81}
B. Simon:
Large time behavior of the $L^p$ norm of Schr\"odinger semigroups, \emph{J. Funct. Anal.} \textbf{40} (1981), 66-83

\bibitem{SSV}
R. Schilling, R. Song, Z. Vondra\v{c}ek:
\emph{Bernstein Functions}, Walter de Gruyter, 2010

\bibitem{S93}
A.V. Sobolev:
The Efimov effect. Discrete spectrum asymptotics, \emph{Commun. Math. Phys.} \textbf{156} (1993), 127-168

\bibitem{T91}
H. Tamura:
The Efimov effect of three-body Schr\"odinger operators, \emph{J. Funct. Anal.} \textbf{95} (1991), 433-459


\bibitem{T}
H. Triebel:
\emph{Interpolation Theory, Function Spaces, Differential Operators}, 2nd ed., Johann Ambrosius Barth Verlag, Heidelberg,
1995


\bibitem{W74}
R.A.~Weder:
Spectral properties of one-body relativistic spin-zero Hamiltonians,
\emph{Ann. Inst. H. Poincar\'e},  Sect. A (N.S.) \textbf{20} (1974), 211--220


\bibitem{Y82}
D. Yafaev:
The low energy scattering for slowly decreasing potentials, \emph{Commun. Math. Phys.} \textbf{85} (1982), 177-196

\bibitem{ZLK}
R.K.P. Zia, R. Lipowsky, D. M. Kroll:
Quantum unbinding in potentials with $1/r^p$ tails, \emph{Am. J. Phys.} \textbf{56} (1988), 160-163


\bibitem{Z2}
Y. Zhang et al:
Resonant mode conversions and Rabi oscillations in a fractional Schr\"odinger equation, \emph{Optics Express}
\textbf{25} (2017), 32401--32410
\end{thebibliography}
\end{document}